\documentclass[a4paper,twoside,10pt]{article}

\usepackage[a4paper,left=3cm,right=3cm, top=3cm, bottom=3cm]{geometry}

\usepackage{amsmath}
\usepackage{mathtools}
\usepackage{mathrsfs}
\usepackage{stmaryrd}
\usepackage{amssymb}
\usepackage{bm}
\usepackage{hyperref}
\usepackage{color}

\usepackage{amsmath}
\usepackage{amsthm}
\usepackage{amssymb}
\usepackage{stmaryrd}
\usepackage{float}
\usepackage{bigints}
\usepackage{cite}
\usepackage{color}
\usepackage[abs]{overpic}
\usepackage[font=footnotesize,labelfont=bf]{caption}
\usepackage{cases}
\usepackage{tikz}
\usepackage{algorithmic}
\usepackage{rotating}
\usepackage{blkarray}
\usetikzlibrary{matrix,calc,arrows}
\usepackage[author={Lorenzo}]{pdfcomment}
\usepackage{soul,xcolor}
\usepackage{enumitem}  
\usepackage{verbatim}
\usepackage{graphicx}
\usepackage{subcaption}
\usepackage{mwe}
\usepackage{algorithm}

\restylefloat{table}
\theoremstyle{plain}
\newtheorem{thm}{Theorem}[section]
\newtheorem{cor}[thm]{Corollary}
\newtheorem{lem}[thm]{Lemma}

\theoremstyle{definition}

\theoremstyle{remark}
\newtheorem{remark}{Remark}

\begin{document}


%
%

\title{Interpolation and stability estimates for edge and face   virtual elements of general order}

\author{L. Beir\~{a}o  da Veiga\thanks{Dipartimento di Matematica e Applicazioni, Universit\`{a} degli Studi di Milano-Bicocca,
20125, Milano, Italy, (lourenco.beirao@unimib.it, lorenzo.mascotto@unimib.it)}
\thanks{IMATI-CNR, 27100, Pavia, Italy},
L. Mascotto\footnotemark[1]~\footnotemark[2]~\thanks{Faculty of Mathematics, University of  Vienna, Oskar Morgenstern Platz 1090 Vienna,  Austria, (lorenzo.mascotto@univie.ac.at)},
J. Meng\thanks{School of Mathematics and Statistics,Xi'an Jiaotong University,
710049, Shaanxi, P.R. China, (mengjian0710@stu.xjtu.edu.cn)}}
\maketitle


\begin{abstract}
\noindent We develop interpolation error estimates for general order standard and serendipity edge and face  virtual elements in two and three dimensions.
Contextually, we investigate the stability properties of the associated~$L^{2}$ discrete bilinear forms.
These results are fundamental tools in the analysis of general order virtual elements, e.g., for electromagnetic problems.
\end{abstract}



\section{Introduction}
\label{intro}
The virtual element method (VEM)\cite{beiraobrezzimariniru2013basic} can be interpreted as an extension of the finite element method (FEM) to polytopal meshes.
Trial and test spaces typically contain a polynomial subspace \emph{plus} other nonpolynomial functions that are never computed explicitly.
Rather, these functions are evaluated via cleverly chosen degrees of freedom (DoFs) and allow for the design of (nodal, edge, face \dots) conforming global spaces.
Such DoFs can be used to compute certain polynomial projections and stabilizations:
the former are needed for the polynomial consistency of the scheme; the latter for its well-posedness.

A preliminary version of $\textbf{H}(\textrm{div})$ virtual elements was first introduced for 2D problems in Ref.~\cite{brezzirichardmarini2014}
as the extension of Raviart-Thomas or Brezzi-Douglas-Marini elements to polygonal meshes.
In order to cope with a sufficiently wide range of problems in mixed form and electromagnetic problems, see for instance Refs. \cite{boffibookmixed,monkmaxwell2003}, in Ref.~\cite{beiraobrezzimariniru2015} the authors developed
several variants of~$\textbf{H}(\textrm{div})$ and~$\textbf{H}(\textbf{curl})$ VE spaces in two and three dimensions.
Furthermore,  serendipity edge and face virtual element spaces were first considered in Ref.~\cite{bebremaru2016RLMASerendipity}; serendipity spaces allow for a reduction of the number of internal DoFs without affecting the convergence and stability properties of the VEM.
This fact has a paramount impact on the performance of the method in the  three dimensional case, notably in the reduction of the face DoFs,
as bulk DoFs in 3D can be removed by static condensation.
Although the spaces introduced in Ref.~\cite{bebremaru2016RLMASerendipity} are more efficient than those in Ref.~\cite{beiraobrezzimariniru2015},
they have the important drawback of missing the full discrete De-Rham diagram, only recovering part of it.
This shortcoming was finally handled in a series of paper, which represent the current ``state of the art'' of VEM De Rham complexes, dealing with the general order 2D case\cite{lourencobrezzidassi2017cmame},
the lowest order 3D case\cite{beiraobrezzidasimaru2018cmame}, and the 3D general order case\cite{beiraobrezzidasimaru2018siam}.
All these papers also treat the magnetostatic equations as a simple model problem; more involved problems can be found, e.g., in  
Refs.~\cite{BeiraodaVeiga-Dassi-Manzini-Mascotto:2021,Cao-Chen-Guo:2021}. The lowest order case\cite{beiraobrezzidasimaru2018cmame} was published independently of the general order case\cite{beiraobrezzidasimaru2018siam}
not only with the aim of reaching different communities, but also because the former case allows for a simpler definition of the VE spaces.

Compared to its nodal counterpart \cite{lalorusso2016,brennerguansung2017,brennerlisung2018m3as,caochenlongsiam2018,chenhuang2018calcolo,ganagianigeorgosuntton2017,morariverarodstklov20151},
the interpolation and stability theory for edge and face virtual elements is still rather limited.
In  Ref.~\cite{beiveigalomorariro2017acoustic}, interpolation estimates for  $\textbf{H}(\textrm{div})$ virtual element spaces in 2D were proved, while $\textbf{H}(\textbf{curl}^{2})$ virtual element spaces in 2D were tackled in Ref.~\cite{zhaom3as2021}.
The extension to two-dimensional face virtual elements with curved edges, including interpolation properties, was considered  in  Ref.~\cite{francofulosciscova2021CMAME}.
Most importantly, both interpolation error estimates and stability properties for the lowest order edge and face virtual element spaces of  Ref.~\cite{beiraobrezzidasimaru2018cmame} were derived in Ref.~\cite{beiraolourenco2020} in two and three dimensions.

The aim of this paper is to prove interpolation estimates and stability properties for general order standard and serendipity edge and face virtual element spaces in 2D and 3D\cite{lourencobrezzidassi2017cmame,beiraobrezzidasimaru2018siam,beiraobrezzimariniru2015,bebremaru2016RLMASerendipity}.
Amongst the several variants of \textbf{H}(\textrm{div}) and \textbf{H}(\textbf{curl}) spaces,
we focus on those in Ref.~\cite{beiraobrezzidasimaru2018siam}.
The ideas outlined in the paper can be extended to other settings as well.

Compared with the proofs for the lowest order spaces\cite{beiraolourenco2020}, the general order case hides many additional difficulties of technical nature.
For instance, many more DoFs types (moments of various kinds on edges, faces, volumes) appear
and serendipity spaces are employed.
Indeed, while in the lowest order spaces the serendipity construction can be avoided by a simpler, yet equivalent, definition,
it is in the general order case that the peculiar definition of serendipity VE spaces appears in its full complexity.
To the authors knowledge, this is the first contribution
where the interpolation and stability analysis of serendipity VE spaces (of any kind) is tackled.
Although many relevant ideas are contained in the proofs of the ``lesser'' lemmas, we give here a short guideline of our main results:
\begin{itemize}
\item Theorems~\ref{thm31edge2d} and~\ref{theorem31} contain interpolation estimates for 2D standard and serendipity edge elements, respectively;
\item Theorems~\ref{theorem41original} and~\ref{theorem41} quickly extend the above results to 2D standard and serendipity face elements, respectively;
\item Theorems~\ref{theorem3dedgeinter} and~\ref{theorem433dedges} contain interpolation estimates for 3D standard and serendipity edge elements, respectively;
\item Theorem~\ref{theorem3dfaceinter} contains interpolation estimates for 3D standard face elements;
\item Theorems~\ref{stab2dedge} and~\ref{stab2dedges} contain the stability estimates for 2D standard and serendipity edge spaces, respectively;
\item Remark~\ref{remark:stability-2D-face} extends the stability estimates to 2D standard and serendipity face spaces;
\item Theorem~\ref{thm3dstaedge} and Remark~\ref{remark:stability-3D-face} contain the stability estimates for 3D standard and serendipity edge spaces, respectively;
\item Theorem~\ref{them3dface} contains the stability estimates for 3D standard face spaces.
\end{itemize}

The remainder of the  paper is organized as follows:
in Section~\ref{sec2}, we introduce the necessary functional spaces and mesh assumptions, and recall some technical results needed for the error estimates;
in Sections~\ref{sec3} and~\ref{sec4}, we prove the interpolation error estimates for edge and face virtual element spaces in 2D and 3D, respectively;
in Section~\ref{sec5}, we define several stabilizations for edge and face virtual element spaces,
and prove their stability properties.

 \section{{Preliminaries}} \label{sec2}
The outline of this section is as follows:
in Section~\ref{sec20}, we introduce the functional space setting;
in Section~\ref{sec21}, we detail the assumptions on the regularity of   the mesh decompositions;
in Sections~\ref{subsec23}, \ref{subsec24}, and~\ref{subsec25}, we state some technical results,
namely polynomial inverse inequalities and decompositions, Sobolev trace inequalities, and Poincar\'{e} and Friedrichs inequalities, respectively.

\subsection{{Sobolev spaces}} \label{sec20}
Throughout the paper, given $m, p\in \mathbb{N}_{0}$ and a bounded Lipschitz domain~$D\subseteq \mathbb{R}^{d}$ ($d=1,2,3$) with boundary~$\partial D$,
we shall use standard notations\cite{brennerscott2008} for the scalar Sobolev space $W^{m,p}(D)$ equipped with the  norm  $\|\cdot\|_{W^{m,p}(D)}$ and the seminorm $|\cdot|_{W^{m,p}(D)}$.
If~$p=2$, we denote~$W^{m,2}(D)$ by~$H^{m}(D)$ equipped with the  norm  $\|\cdot\|_{m,D}$, the seminorm  $|\cdot|_{m,D}$, and the inner product $(\cdot,\cdot)_{D}$.
We set $H^{0}(D)=L^{2}(D)$; in the corresponding norm, we omit the subscript~$0$.
Let $H^{-m}(D)$ be  the dual space of $H^{m}(D)$ equipped with the negative norm $\|\cdot\|_{-m,D}$.
For~$k\in \mathbb{N}_{0}$, $\mathbb{P}_{k}(D)$ denotes the space of polynomials of degree at most~$k$ on~$D$ and~$\pi_{k,d}$ its dimension.
We set~$\mathbb{P}_{-\ell}(D)=\{0\}$ for all $\ell \in \mathbb{N}$.
Moreover, $\mathbb{P}_{k}^{0}(D)$ denotes the subspace of~$\mathbb{P}_{k}(D)$ of functions with zero average on either~$\partial D$ or~$D$.
We shall use the  boldface to denote vector variables and spaces;
for example, $\textbf{v}$, $\textbf{H}^{m}(D)$, and~$\textbf{L}^{2}(D)$
denote the vector version of a function~$v$, a Sobolev space, and a Lebesgue space.

With an abuse of notation, we denote local sets of coordinates
in two and three dimensions by~$[x_1,x_2]$ and~$[x_1,x_2,x_3]$, respectively.
Given a  function $\phi:F\subseteq\mathbb{R}^{2}\rightarrow \mathbb{R}$ and a field $ \textbf{v}=[v_{1},v_{2}]^{T}: F\subseteq\mathbb{R}^{2}\rightarrow \mathbb{R}^{2}$,  we define the operators
 \begin{equation*}
 \begin{aligned}
 \label{definition1}
&\bm{\nabla}_{F} \phi=\left[\frac{\partial \phi}{\partial x_{1}}, \frac{\partial  \phi}{\partial x_{2}}\right]^{T}, \hspace{0.25cm} \textbf{curl}_{F} \hspace{0.05cm}\phi=\left[\frac{\partial \phi}{\partial x_{2}}, -\frac{\partial  \phi}{\partial x_{1}}\right]^{T}, \hspace{0.25cm}
  \Delta_F \phi = \frac{\partial^2 \phi}{\partial^2 x_{1}} + \frac{\partial^2 \phi}{\partial^2 x_{2}}, \\
&\textrm{rot}_{F} \hspace{0.05cm}  \textbf{v}= \frac{\partial v_{2}}{\partial x_{1}}- \frac{\partial v_{1}}{\partial x_{2}}, \hspace{0.4cm} \textrm{div}_{F}\hspace{0.05cm}\textbf{v}= \frac{\partial v_{1}}{\partial x_{1}}+ \frac{\partial v_{2}}{\partial x_{2}}.
\end{aligned}
 \end{equation*}
In three dimensions, given a function $\phi: \mathbb{R}^{3}\rightarrow \mathbb{R}$ and a field $\textbf{v}=[v_{1},v_{2},v_{3}]^{T}:\mathbb{R}^{3}\rightarrow \mathbb{R}^{3}$, we define
\begin{equation*}
\begin{aligned} \label{definition13d}
&\bm{\nabla}  \phi=\left[\frac{\partial \phi}{\partial x_{1}}, \frac{\partial  \phi}{\partial x_{2}}, \frac{\partial  \phi}{\partial x_{3}}\right]^{T},
\hspace{0.4cm} \Delta \phi = \frac{\partial^2 \phi}{\partial^2 x_{1}} + \frac{\partial^2 \phi}{\partial^2 x_{2}} + \frac{\partial^2 \phi}{\partial^2 x_{3}},
 \\
 &\textbf{curl} \hspace{0.05cm}\textbf{v}=\left[\frac{\partial v_{3}}{\partial x_{2}}-\frac{\partial v_{2}}{\partial x_{3}}, \frac{\partial v_{1}}{\partial x_{3}}-\frac{\partial v_{3}}{\partial x_{1}}, \frac{\partial v_{2}}{\partial x_{1}}-\frac{\partial v_{1}}{\partial x_{2}}  \right]^{T}, \hspace{0.25cm}\textrm{div} \hspace{0.05cm}\textbf{v}= \frac{\partial v_{1}}{\partial x_{1}}+ \frac{\partial v_{2}}{\partial x_{2}}+ \frac{\partial v_{3}}{\partial x_{3}}.
\end{aligned}
\end{equation*}
Next,  given a polygon $F$ and a polyhedron $E$,
we denote the usual~$\textrm{div}$, $\textrm{rot}$, and~$\textbf{curl}$ spaces
by~$\mathbf{H}(\textrm{div}_{F}, F)$, $\mathbf{H}(\textrm{rot}_{F}, F)$, $\mathbf{H}(\textrm{div}, E)$, and~$\mathbf{H}(\textbf{curl}, E)$.

\subsection{Mesh regularity assumptions}  \label{sec21}
Let~$\mathcal{T}_{h}$ be a sequence of decompositions of a given polyhedral domain $\Omega\subseteq \mathbb{R}^{2}$ or~$\mathbb R^3$ into nonoverlapping polygonal/polyhedral elements~$E$.
For each~$E$, we denote its two-dimensional boundary by~$\partial E$
and the one-dimensional boundary of each face~$F$ in $\partial E$ by~$\partial F$.
For any geometric object~$D$ of dimension~$d$ ($d=1,2,3$),
i.e., an edge~$e$, a face~$F$, or an element~$E$, we denote its barycenter, its measure (length, area, or volume, respectively), and its diameter by~$\textbf{b}_{D}$,  $|D|$, and~$h_{D}$, respectively.
We denote the unit outer normal to the boundary $\partial E$ by $\textbf{n}_{\partial E}$ and  the restriction to the face $F$ of $\textbf{n}_{\partial E}$ by $\textbf{n}_{F}$.
For each face $F$, we also denote the unit outer normal to $\partial F$ in the plane containing $F$  by $\textbf{n}_{\partial F}$  and   the restriction to the edge $e$ of $\textbf{n}_{\partial F}$ in the plane containing $F$   by $\textbf{n}_{e}$.
Further, the unit   tangential vector $\textbf{t}_{e}$  along  the edge $e$ is defined as the vector pointed in counter-clockwise sense of $\textbf{n}_{e}$ (for example, $\textbf{t}_{e}=(-n_2,n_1)$ if $\textbf{n}_{e}=(n_1,n_2)$ in two dimensions), and  $\textbf{t}_{\partial F}$ is locally defined by  $\textbf{t}_{\partial F}|_{e}:=\mathbf{t}_{e}$.

Henceforth, we demand the following mesh regularity assumption:
\begin{itemize}
\item[(\textbf{M})] For~$d=2$, there exists a uniform constant~$\rho >0$ such that, for every polygon~$F$,
\begin{itemize}
\item[(i)] $F$ is star-shaped with respect to a disk of radius $\geq\rho h_{F}$;
\item[(ii)] every edge $e$ of~$\partial F$ satisfies $h_{e}\geq \rho h_{F}$.
\end{itemize}
For~$d=3$, there exists a uniform constant~$\rho >0$ such that, for every element~$E$,
 \begin{itemize}
\item[(i)] $E$ is star-shaped with respect to a ball of radius $\geq\rho h_{E}$;
\item[(ii)] every face $F$ of~$\partial E$ is star-shaped
with respect to  a disk with  radius $\geq\rho h_{F}$;
\item[(iii)] for every face~$F$ of~$\partial E$, every edge $e$ of~$\partial F$ satisfies $h_{e}\geq \rho h_{F}\geq \rho^{2} h_{E}$.
\end{itemize}
\end{itemize}
In certain cases that will be indicated explicitly,
we shall also require the following uniform convexity condition:
\begin{itemize}
\item[(\textbf{MC})] in two dimensions, every polygonal element $F$ is convex and there exists a constant $\varepsilon>0$ such that each internal angle $\theta$ of element $F$ satisfies $\varepsilon\leq \theta\leq \pi-\varepsilon$;
in the three dimensional case, each face $F$ of the mesh satisfies such condition.
\end{itemize}

\begin{remark}
\label{rem21}
An immediate consequence of the above mesh regularity assumptions is that each three-dimensional element~$E$
or each two-dimensional face~$F$
are uniformly Lipschitz domains that
admit a shape-regular tessellation~$\widetilde{\mathcal{T}}_{h}$ into simplices,
i.e., a partition of~$E$ into tetrahedra or~$F$ into triangles.
Such a decomposition is obtained by connecting each edge/face
(in two and three dimensions, respectively)
with the center of the ball in assumption (\textbf{M}).
\end{remark}

\noindent In what follows, given two positive quantities~$a$ and~$b$,
we use the short-hand notation ``$a\lesssim b$'' if there exists a positive constant~$c$ independent of the discretization parameters
such that ``$a\le c \ b$''.
Moreover, we write ``$a\approx b$'' if and only
if ``$a\lesssim b$'' and ``$b\lesssim a$''.
When keeping track of the constant is necessary, we shall use explicit generic constants 
$C$, $C'$, $C_1,\cdots$  that are independent of the mesh and may vary at different occurrences.
Furthermore, $D$ will denote a generic polytopal domain
(polygon in~$\mathbb R^2$ or polyhedron in~$\mathbb R^3$)
representing either an element or a face of the mesh, thus satisfying the above assumptions (\textbf{M}).

Throughout, the explanation of the identities
and upper and lower bounds will appear
either in the preceding text or as an equation reference above the equality symbol ``$=$''
or the inequality symbols ``$\le$'', ``$\ge$'' etc, whichever we believe it is easier for the reader.

\subsection{Polynomial properties} \label{subsec23}
The following polynomial inverse estimates in a polytopal domain $D \subset \mathbb{R}^{d}$ ($d=2,3$) are valid:
for all~$p_k \in \mathbb P_k(D)$,
\begin{equation} \label{Polynomialinverseestimates}
 \|p_{k}\|_{\partial D} \lesssim h_{D}^{-\frac12}\|{p_{k}}\|_{D} ,
 \qquad
 |p_{k}|_{1, D} \lesssim h_{D}^{-1}\|{p_{k}}\|_{D} ,
 \qquad
\|p_{k}\|_{D}   \lesssim h_{D}^{-1} \| p_{k} \|_{-1,D}.
\end{equation}
Furthermore, for each piecewise polynomial~$p_k$ of degree at most~$k$ over~$\partial D$, we have
\begin{equation}
\label{Polynomialinverseestimates1}
\|p_{k}\|_{\partial D} \lesssim  h_{D}^{-\frac12}\|{p_{k}}\|_{-\frac12,\partial D},
\end{equation}
where $\| \cdot \|_{-\frac12,\partial D}$ denotes the scaled $H^{-\frac12}(\partial D)$ dual norm
\[
\vert \cdot \vert_{-\frac12,\partial D}
:= \sup_{\varphi \in H^{\frac12}(\partial D)}
   \frac{(\cdot,\varphi)_{\partial D}}
   {\vert \varphi \vert_{\frac12,\partial D}+
   h_D^{\frac12} \Vert \varphi \Vert_{\partial D}}.
\]
The proof of the above inverse estimates hinges upon the existence of a shape-regular simplicial tessellation, see Remark~\ref{rem21},
and standard polynomial inverse estimates on simplices as in Section~$3.6$ of Ref.~\cite{verfurth2013}.

Let~$b_{D}$ be the cubic ($d=2$) or quartic ($d=3$) piecewise bubble function associated with the shape-regular tessellation of the element~$D$,
see Remark~\ref{rem21},
with unitary $L^\infty$ norm.
The following result which establishes standard estimate  for bubble functions will be useful:
\begin{align}
\label{bubblefuncionproer}
\|p_{k}\|_{D}^{2} 
\lesssim \int_{D} b_{D} p_{k}^{2}
\lesssim \|p_{k}\|_{D}^{2}\hspace{0.3cm} \forall p_{k}\in \mathbb{P}_{k}(D).
\end{align}
A proof of this result is obtained by using Theorem 3.3 in Ref.~\cite{apostfemcmame1997} and standard manipulations.

Moreover, the following decompositions of polynomial vector spaces  are valid; see, e.g.,  Refs.~\cite{lourencobrezzidassi2017cmame,beiraobrezzimariniru2015}.    {Given a polygon $F$}, we have
\begin{align}
\label{decompo1}
(\mathbb{P}_{k}(F))^{2}=\textbf{curl}_{F}\hspace{0.05cm} \mathbb{P}_{k+1}(F)\oplus \textbf{x}\mathbb{P}_{k-1}(F),
\end{align}
which  implies  that  $ \textrm{div}_{F}$ is  an  isomorphism  between  $\left\{\textbf{x}\mathbb{P}_{k}(F)\right\}$   and $\mathbb{P}_{k}(F)$.
Moreover,
\begin{align} \label{decompo2}
(\mathbb{P}_{k}(F))^{2}=\bm{\nabla}_{F}\mathbb{P}_{k+1}(F)\oplus \textbf{x}^{\perp}\mathbb{P}_{k-1}(F),
\end{align}
which implies that $\textrm{rot}_{F}$ is  an  isomorphism  between  $\left\{\textbf{x}^{\perp}\mathbb{P}_{k}(F)\right\}$  and $\mathbb{P}_{k}(F)$.

Given a polyhedron $E$, we have
\begin{align} \label{decompo13d}
(\mathbb{P}_{k}(E))^{3}=\textbf{curl}\hspace{0.05cm} (\mathbb{P}_{k+1}(E))^{3}\oplus \textbf{x}\mathbb{P}_{k-1}(E),
\end{align}
 which  implies  that   $\textrm{div}$  is  an  isomorphism  between  $\left\{\textbf{x}\mathbb{P}_{k}(E)\right\}$   and $\mathbb{P}_{k}(E)$.
Furthermore,
 \begin{align}
\label{decompo23d}
(\mathbb{P}_{k}(E))^{3}=\bm{\nabla}\mathbb{P}_{k+1}(E)\oplus \textbf{x}\wedge (\mathbb{P}_{k-1}(E))^{3},
\end{align}
which implies that for each $ \mathbf{p}_{k}\in (\mathbb{P}_{k}(E))^{3}$ with $\textrm{div}\hspace{0.05cm} \mathbf{p}_{k}=0$,
there exists $\mathbf{q}_{k}\in (\mathbb{P}_{k}(E))^{3}$
such that $\textbf{curl}(\textbf{x} \wedge \mathbf{q}_{k})= \mathbf{p}_{k}$.

 \subsection{Trace inequalities} \label{subsec24}
The following trace inequalities are valid; see, e.g., in~\cite[Theorem~A.$20$]{Schwab1998}:
given a polytopal domain~$D$,
representing either an element or a face of the mesh, there hold
\begin{align} \label{Traceinequality3_1}
\|v\|_{\partial D}&\lesssim h_{D}^{-\frac12}\|v\|_{D}+ h_{D}^{ \delta-\frac12}|v|_{ \delta, D} \qquad \forall v\in H^{ \delta}(D),  \frac12 < \delta < \frac32 \\
\label{Traceinequality3_1_1}
|v|_{\varepsilon,\partial D} &\lesssim h_{D}^{-(\varepsilon+\frac12)}\|v\|_{D}+  |v|_{\varepsilon+\frac12, D} \qquad  \forall v\in H^{\varepsilon+\frac12}(D), 0<\varepsilon < 1.
\end{align}
If additionally $1\slash2 <\delta\leq 1$ and $v$ has zero average on either~$\partial D$ or~$D$, then we have
\begin{align} \label{Traceinequality4}
 \|v\|_{\partial D}\lesssim h_{D}^{ \delta-\frac12}|v|_{ \delta, D}.
\end{align}
For functions~$v$ with zero average on either~$\partial D$ or~$D$,
we also recall the multiplicative trace inequality
\begin{align}  \label{Traceinequality4last}
 \|v\|_{\partial D}\lesssim \|v\|^{\frac12}_{D} |v|^{\frac12}_{1, D}.
\end{align}
Let $F$ be a polygon and $E$ be a polyhedron, respectively, 
representing either a face $F$ or an element $E$ of the mesh, thus satisfying the above assumptions (\textbf{M}).
For $\textbf{w}\in \textbf{H}(\textrm{div}_{F},F)$,
$\textbf{v}\in \textbf{H}(\textrm{rot}_{F},F)$,
$\bm{\phi}\in \textbf{H}(\textrm{div},E)$,
$\bm{\psi}\in \textbf{H}(\textbf{curl},E)$,
and $\bm{\chi}\in\textbf{H}(\textrm{div},E)\cap \textbf{H}(\textbf{curl},E)$,
the following trace inequalities are valid;
the following trace inequalities are valid; see, e.g.,
Theorems~$3.29$ and~$3.24$ in~\cite{co1990,monkmaxwell2003},
and page~$367$ in~\cite{co1990}:
\begin{align}  \label{Traceinequality2}
\|\textbf{v}\cdot \textbf{t}_{\partial F}\|_{-\frac12,\partial F} &\lesssim \|\textbf{v}\|_{F}+ h_{F}\|\textrm{rot}_{F}\hspace{0.05cm}\textbf{w}\|_{F},\\
\label{Traceinequality3}
\|\bm{\phi}\cdot \textbf{n}_{\partial E}\|_{-\frac12,\partial E}& \lesssim \|\bm{\phi}\|_{E} + h_{E} \|\textrm{div}\hspace{0.05cm}\bm{\phi}\|_{E},\\
\label{Traceinequality4_1}
\|\bm{\psi}\wedge \textbf{n}_{\partial E}\|_{-\frac12,\partial E} &\lesssim \|\bm{\psi}\|_{E}+ h_{E} \|\textbf{curl}\hspace{0.05cm}\bm{\psi}\|_{E},\\
\label{Traceinequality5}
\|\bm{\chi} \!\wedge\! \textbf{n}_{\partial E}\|_{\partial E} 
\!\lesssim\! h_{E}^{-\frac12} \|\bm{\chi}\|_{E}&
\!+\! h^{\frac12}_{E}\|\textrm{div}\hspace{0.05cm}\bm{\chi}\|_{E}
\!+\! h_{E}^{\frac12}\|\textbf{curl}\hspace{0.05cm}\bm{\chi}\|_{E}
\!+\! \|\bm{\chi} \!\cdot\! \textbf{n}_{\partial E}\|_{\partial E}.
\end{align}
All constants involved in the bounds above are uniform, i.e. independent of the particular element~$E$ or face~$F$ in $\{\mathcal{T}_{h}\}_h$, since the mesh assumptions {\bf (M)} guarantee that the parameters associated to the star-shaped and Lipschitz properties are uniform in the mesh family.

\subsection{Poincar\'{e} and Friedrichs inequalities}
\label{subsec25}
For each $v\in H^{1}(D)$, $D\subseteq \mathbb{R}^{d}$ ($d=2,3$),  if $v$ has zero average on either~$\partial D$ or~$D$, then we have the following Poincar\'{e} inequality; see, e.g., Section 5.3 in Ref.~\cite{brennerscott2008}:
\begin{align} \label{poincarefridine}
  h_{D}^{-1}\|v\|_{D} \lesssim  |v|_{1, D}.
\end{align}
Let $E\in \mathcal{T}_{h}$ be a polyhedral element and $\textbf{v}\in \textbf{H}(\textbf{curl}, E)\cap \textbf{H}(\textrm{div}, E)$ be a divergence free function satisfying  $\textbf{v}\wedge \textbf{n}_{\partial E}\in \textbf{L}^2(\partial E)$.
Then, the following Friedrichs inequality is valid; see, e.g., Corollary 3.51 in  Ref.~\cite{monkmaxwell2003} or Lemma 2.2 in Ref.~\cite{beiraolourenco2020}:
\begin{align} \label{poincarefridinecurl}
h_{E}^{-1}\|\textbf{v}\|_{E} \lesssim    h_{E}^{-\frac12}\| \textbf{v}\wedge  \textbf{n}_{\partial E}\|_{\partial E} +  \|\textbf{curl}\hspace{0.05cm}  \textbf{v}\|_{E}.
\end{align}
Similarly, let $\textbf{v}\in \textbf{H}(\textbf{curl}, E)\cap \textbf{H}(\textrm{div}, E)$ be a divergence free function satisfying  $\textbf{v}\cdot \textbf{n}_{\partial E}\in L^2(\partial E)$.
Then, the following  Friedrichs inequality is also valid; see Corollary~$3.51$ in Ref.~\cite{monkmaxwell2003}:
\begin{align} \label{poincarefridinediv}
h_{E}^{-1}\|\textbf{v}\|_{E} \lesssim h_{E}^{-\frac12} \| \textbf{v}\cdot \textbf{n}_{\partial E}\|_{\partial E}+ \|\textbf{curl}\hspace{0.05cm}  \textbf{v}\|_{E}.
\end{align}

 \section{{Interpolation properties} of edge and face virtual element spaces in 2D} \label{sec3}
Here, we prove interpolation properties of general order for standard and serendipity edge and face virtual element spaces on polygons.
These polygons can be interpreted as elements of a two-dimensional mesh or as faces of a three-dimensional mesh;
we shall often refer to them as ``faces''.
In what follows,  we shall concentrate on interpolation and stability results on local elements,
since  the corresponding global results follow by a summation on all the elements.
In  Section \ref{sec31}, we begin  with   edge virtual element spaces on polygons;
in Section~\ref{sec4_1}, we consider the serendipity  edge virtual element space  in 2D,
which allows us to reduce the number of internal DoFs  of the standard  edge   virtual element  space introduced in Section \ref{sec31};
in Section~\ref{sec4faceVEM}, we extend the results of edge virtual element spaces to face virtual element spaces in 2D.

\subsection{Standard edge virtual element space on polygons} \label{sec31}
Given a face $F$ and an integer  $k\geq 1$, the   edge virtual element space is defined as\cite{beiraobrezzidasimaru2018siam}
\begin{equation}
\begin{aligned}
\label{finitspace}
\textbf{V}_{k}^{e}(F)=\big\{\textbf{v}_{h}\in \textbf{L}^{2}(F): \textrm{div}_{F}\hspace{0.05cm} \textbf{v}_{h}\in \mathbb{P}_{k}(F), \   &\textrm{rot}_{F} \hspace{0.05cm}  \textbf{v}_{h} \in  \mathbb{P}_{k-1}(F), \\ &\textbf{v}_{h}\cdot \textbf{t}_{e}\in  \mathbb{P}_{k}(e) \ \ \forall e\subseteq \partial F \big\}.
\end{aligned}
\end{equation}
The following linear operators are a set of unisolvent DoFs:
\begin{align}
\label{dof1}
&\bullet \ \textrm{the\ moments}\ \int_{e} \textbf{v}_{h}\cdot \textbf{t}_{e}p_{k}
&&\forall p_{k}\in \mathbb{P}_{k}(e),  \hspace{0.1cm} \forall e\subseteq \partial F;\\
\label{dof2}
&\bullet \ \textrm{the\ moments}\ \int_{F} \textbf{v}_{h}\cdot \textbf{x}_{F} p_{k}  
&&\forall p_{k}\in \mathbb{P}_{k}(F);\\
\label{dof3}
&\bullet\ \textrm{the\ rot-moments}\  \int_{F} \textrm{rot}_{F}\hspace{0.05cm}\textbf{v}_{h}  p_{k-1}^{0}  
&&\forall p_{k-1}^{0}\in \mathbb{P}_{k-1}^{0}(F)   \hspace{0.2cm} \textrm{only\ for}\  k>1,
   \end{align}
where $\textbf{x}_{F} :=\textbf{x}-\textbf{b}_{F}$.

The inclusion $(\mathbb{P}_{k}(F))^{2}\subseteq \textbf{V}_{k}^{e}(F)$  is valid and the $\textbf{L}^{2}$ projection $ \mathbf{\Pi}_{{k+1}}^{0,F}: \textbf{V}_{k}^{e}(F)\rightarrow  (\mathbb{P}_{k+1}(F))^{2}$ is computable by the DoFs~\eqref{dof1}--\eqref{dof3}; see  Refs.~\cite{lourencobrezzidassi2017cmame,beiraobrezzidasimaru2018siam}.

\begin{remark}[Generality of the approach]
To keep the theoretical analysis as clear as possible, we chose the~$\textbf{V}_{k}^{e}(F)$ that corresponds to that of Ref.~\cite{beiraobrezzidasimaru2018siam}.
We might have employed other definitions;
see, e.g., Refs.~\cite{lourencobrezzidassi2017cmame,bebremaru2016RLMASerendipity}.
This would simply result in a change of the polynomial orders appearing in~\eqref{finitspace}, and \eqref{dof1}--\eqref{dof3}:
the notation would be heavier but the theoretical extension would trivially follow the same steps here shown for~\eqref{finitspace}.
This same consideration applies to all the virtual element spaces introduced in the following.
\end{remark}

We begin with the proof of the following auxiliary bound for   functions belonging to $  \textbf{V}_{k}^{e}(F)$.
\begin{lem}
\label{lem2degde}
For each $\textbf{v}_{h}\in \textbf{V}_{k}^{e}(F)$, we have
\begin{equation}
  \begin{aligned}
 \label{prioribound2dedge}
\|\textbf{v}_{h}\|_{F}
\!\lesssim\! 
h_{F}\|{\rm rot}_{F}\hspace{0.05cm}\textbf{v}_{h}\|_{F}  \!
+ \! h_{F}^{\frac12} \|\textbf{v}_{h}\!\cdot\!\textbf{t}_{\partial F}\|_{\partial F}  
\!+\!\!\!\! \sup_{p_{k}\in \mathbb{P}_{k}(F)}\!\!
\frac{ \int_{F}  \textbf{v}_{h} \cdot  \textbf{x}_{F} p_{k}}{\|\textbf{x}_{F} p_{k}\|_{F}}.
\end{aligned}
\end{equation}
\end{lem}
\begin{proof}
Since $\textrm{rot}_{F}\hspace{0.05cm}\textbf{curl}_{F}=-\Delta_{F}$,
the following Helmholtz decomposition of $\textbf{v}_{h}$ is valid:
\begin{align} \label{vhHeldecom2dedge}
\textbf{v}_{h}= \textbf{curl}_{F}\hspace{0.05cm}\rho+\bm{\nabla}_{F} \sigma,
\end{align}
where $\rho\in H^{1}(F)\setminus \mathbb{R}$ and $\sigma\in H^{1}(F)$ satisfy weakly
\begin{align} \label{vhHeldecom12dedge}
-\Delta_{F} \rho= \textrm{rot}_{F}\hspace{0.05cm}\textbf{v}_{h}\ \textrm{in}\ F,\hspace{0.2cm} \textbf{curl}_{F}\hspace{0.05cm}\rho\cdot   \textbf{t}_{\partial F}= \textbf{v}_{h}\cdot   \textbf{t}_{\partial F}\ \textrm{on}\ \partial F,
\end{align}
and
\begin{align} \label{vhHeldecom22dedge}
\Delta_{F} \sigma= \textrm{div}_{F}\hspace{0.05cm} \textbf{v}_{h} \ \textrm{in}\ F,\hspace{0.2cm}  \sigma=0\ \textrm{on}\ \partial F.
\end{align}
By the orthogonality $(\textbf{curl}_{F}\hspace{0.05cm}\rho, \bm{\nabla}_{F} \sigma)_{F}=0$,  we also have
\begin{align}
 \label{vhHeldecom2_2dedge} \|\textbf{v}_{h}\|^{2}_{F}=\|\textbf{curl}_{F}\hspace{0.05cm}\rho\|^{2}_{F}+\|\bm{\nabla}_{F} \sigma\|^{2}_{F}.
\end{align}
We show an upper bound on the two  terms on  the right-hand side of \eqref{vhHeldecom2_2dedge}:
using $\textrm{rot}_{F}\hspace{0.05cm}\textbf{curl}_{F}=- \Delta_{F}$
and $\|\bm{\nabla}_{F} \rho\|_{F} = \|\textbf{curl}_{F} \hspace{0.05cm} \rho\|_{F}$,\footnote{Henceforth, IBP stands for integration by parts}
\begin{eqnarray}
\begin{aligned}[b]
 \label{vhHeldecom32dedge}
& \!\! \|\textbf{curl}_{F}\hspace{0.05cm}\rho\|^{2}_{F} 
\!\!\!\overset{\text{IBP}}{=}\!\!\!
-\!\int_{F} \!\!\rho(\Delta_{F}  \rho)
\!+\!\!\! \int_{\partial F} \!\!\!\!\!\!\! \rho(\textbf{curl}_{F}\hspace{0.05cm}\rho \!\cdot\! \textbf{t}_{\partial F})
\!\!\!\overset{\eqref{vhHeldecom12dedge}}{\lesssim}\!\!\!
\|\rho\|_{F} \| \textrm{rot}_{F}\hspace{0.05cm}\textbf{v}_{h}\|_{F} 
\!+\! \|\rho\|_{\partial F} \|  \textbf{v}_{h} \!\cdot\! \textbf{t}_{\partial F}\|_{\partial F}\\
& \overset{\eqref{Traceinequality4},\eqref{poincarefridine}}{\lesssim}
 h_{F}\|\bm{\nabla}_{F}  \rho\|_{F} \| \textrm{rot}_{F}\hspace{0.05cm}\textbf{v}_{h}\|_{F} + h_{F}^{\frac12}\|\bm{\nabla}_{F}  \rho\|_{F} \|  \textbf{v}_{h}\cdot \textbf{t}_{\partial F}\|_{\partial F}\\
& \lesssim \left(h_{F} \|\textrm{rot}_{F}\hspace{0.05cm}\textbf{v}_{h}\|_{F} 
+ h_{F}^{\frac12} \|  \textbf{v}_{h}\cdot \textbf{t}_{\partial F}\|_{\partial F} \right) 
\|\textbf{curl}_{F}\hspace{0.05cm}\rho\|_{F}.
\end{aligned}
 \end{eqnarray}
By using \eqref{decompo1}, the fact that $\textrm{div}_{F}\hspace{0.05cm}  \textbf{v}_{h} \in \mathbb{P}_{k}(F)$,
and a scaling argument, there exists a polynomial $q_{k}\in \mathbb{P}_{k}(F)$ such that
\begin{align} \label{standardinterpolation2_proof1_12dedge}
\textrm{div}_{F}\hspace{0.05cm} (\textbf{x}_{F}q_{k})=\textrm{div}_{F}\hspace{0.05cm}  \textbf{v}_{h}\  \textrm{and}\ \  \|\textbf{x}_{F}q_{k}\|_{F}
\lesssim h_{F}\|\textrm{div}_{F}\hspace{0.05cm}  \textbf{v}_{h}\|_{F}.
\end{align}
We  have the following inverse estimate involving edge virtual element functions:
 \begin{align} \label{inversediv2dedge}
\|\textrm{div}_{F}\hspace{0.05cm} \textbf{v}_{h}\|_{F}
&\lesssim h^{-1}_{F}\| \textbf{v}_{h}\|_{F}\hspace{0.7cm} \forall \textbf{v}_{h}\in \textbf{V}_{k-1}^{e}(F).
 \end{align}
To prove~\eqref{inversediv2dedge}, we  split the face~$F$ into a shape-regular sub-triangulation $\widetilde{\mathcal{T}}_{h}$; see Remark \ref{rem21}.
Let~$b_{{F}}$ be the usual positive cubic bubble function
over each triangle $\widetilde{F}\in\widetilde{\mathcal{T}}_{h}$
scaled such that~$\|b_{{F}}\|_{\infty,\widetilde{F}}=1$.
By using that~$\textrm{div}_{F}\hspace{0.05cm} \textbf{v}_{h}\in \mathbb{P}_{k}(F)$, and the polynomial inverse inequalities~\eqref{bubblefuncionproer} and~\eqref{Polynomialinverseestimates}, we have
\[
\begin{split}
\|\textrm{div}_{F}\hspace{0.05cm} \textbf{v}_{h} \!\|^{\!2}_{F}&
\! \lesssim\!  (b_{F}\textrm{div}_{F}\hspace{0.05cm} \textbf{v}_{h},\textrm{div}_{F}\hspace{0.05cm} \textbf{v}_{h}\!)_{F}
\!=\! - \! (\bm{\nabla}_{F}\!(b_{F}\textrm{div}_{F}\hspace{0.05cm}   \!\textbf{v}_{h}),\textbf{v}_{h}\!)_{F}
\!\lesssim\! h^{\!\!-1\!}_{F}\|\textrm{div}_{F}\hspace{0.05cm} \textbf{v}_{h}\|_{F}\| \textbf{v}_{h}\|_{F},
\end{split}
\]
which proves \eqref{inversediv2dedge}.

Next, we cope with the second term on the right-hand side of~\eqref{vhHeldecom2_2dedge}:
\begin{eqnarray}
\begin{aligned} \label{gradientestimates2dedge}
&\|\bm{\nabla}_{F}\sigma\|^{2}_{F} 
\overset{\text{IBP}, \eqref{vhHeldecom22dedge}, \eqref{standardinterpolation2_proof1_12dedge}}{=}
-\int_{F}\textrm{div}_{F}\hspace{0.05cm} (\textbf{x}_{F}q_{k})\sigma
\overset{\text{IBP}, \eqref{vhHeldecom22dedge}}{=}
\int_{F}  (\textbf{x}_{F}q_{k})\cdot \bm{\nabla}_{F}\sigma\\
& \overset{\eqref{vhHeldecom2dedge}}{=}
\int_{F}  (\textbf{x}_{F}q_{k})\cdot  (\textbf{v}_{h}- \textbf{curl}_{F}\hspace{0.05cm}\rho)\\
&\leq \|\textbf{x}_{F} q_{k}\|_{F}\sup_{p_{k}\in \mathbb{P}_{k}(F)}\frac{ \int_{F}  \textbf{v}_{h} \cdot \textbf{x}_{F} p_{k}}{\|\textbf{x}_{F} p_{k}\|_{F}} +  \|\textbf{x}_{F}q_{k}\|_{F}\|\textbf{curl}_{F}\hspace{0.05cm}\rho\|_{F}  \\
& \overset{\eqref{standardinterpolation2_proof1_12dedge}, \eqref{inversediv2dedge}}{\lesssim}
\left( \sup_{p_{k}\in \mathbb{P}_{k}(F)}\frac{ \int_{F}  \textbf{v}_{h} \cdot \textbf{x}_{F} p_{k}}{\|\textbf{x}_{F} p_{k}\|_{F}}+\|\textbf{curl}_{F}\hspace{0.05cm}\rho\|_{F}\right)  \|\textbf{v}_{h}\|_{F}.
\end{aligned}
\end{eqnarray}
Substituting~\eqref{vhHeldecom32dedge} and~\eqref{gradientestimates2dedge} into \eqref{vhHeldecom2_2dedge},
and using~\eqref{vhHeldecom2_2dedge} and~\eqref{vhHeldecom32dedge} again,
we can obtain~\eqref{prioribound2dedge}.
 \end{proof}
The following bound,
which generalizes Lemma~$4.4$ in Ref.~\cite{beiraolourenco2020}
will be useful in the sequel.
\begin{lem}
\label{lemma44}
For each face $ F \subseteq \partial E$ and given~$\varepsilon>0$,
let $\textbf{v}\in \textbf{H}^{\varepsilon}(F)\cap \textbf{H}({\rm rot}_{F}, F)$
such that $\textbf{v}\cdot \textbf{t}_{e}$ is integrable on each edge of~$F$.
Then, the following bound is valid: for all $e$ in~$\partial F$ and~$p_{k}$ in~$\mathbb{P}_{k}(e)$,
\begin{equation}
\begin{aligned}
\label{L1BOUND}
\left|\int_{e}(\textbf{v} \cdot\textbf{t}_{e})p_{k}\right|
\lesssim \|p_{k}\|_{L^{\infty}(e)} \left( \|\textbf{v} \| _{F}+ h_{F}^{\varepsilon}|\textbf{v}| _{\varepsilon,F}+h_{F}\|{\rm rot}_{F} \hspace{0.05cm}\textbf{v} \|_{F} \right) .
\end{aligned}
\end{equation}
The last term on the right-hand side can be neglected if~$\varepsilon>1/2$.
\end{lem}
 \begin{proof}
The inequality is trivial for $\varepsilon>\frac12$ by using the trace inequality \eqref{Traceinequality3_1}.
Therefore, we assume~$0<\varepsilon\leq \frac12$.
Recalling Remark \ref{rem21}, we split
the face $F$ into a shape-regular triangulation $\widetilde{\mathcal{T}}_{h}$. Let $T\in \widetilde{\mathcal{T}}_{h}$ be the  triangle such that $e\subseteq \partial T$.
We first prove the following inequality: for all fixed~$p>2$
and $p_{k}\in \mathbb{P}_{k}(e) \ \forall e\subseteq \partial F$,
\begin{equation}
\begin{aligned} \label{L1BOUND_proof1}
\left|\int_{e}(\textbf{v} \cdot\textbf{t}_{e})p_{k}\right|   
\lesssim \|p_{k}\|_{L^{\infty}(e)}\left(   h_{F}^{1-2/p}\|\textbf{v}\|_{L^{p}(T)}+h_{F}\|{\rm rot}_{F} \hspace{0.05cm}\textbf{v} \|_{T} \right).
\end{aligned}
\end{equation}
Let $\hat{T}$ be the affine equivalent reference element to the triangle $T$ and $\hat{e}$ be the edge of $\hat{T}$ corresponding to the edge $e\subseteq \partial T$ through the Piola transform; see Definition 3.4.1 in Ref.~\cite{brennerscott2008}.
Let $\hat{q}_{k}: \hat{T}\rightarrow \mathbb{R}$ be the prolongation of $\hat{p}_{k}$ (~$\hat{\cdot}$ denoting the usual pull-back of $\cdot$~; see Remark~$3.4.2$ in Ref.~\cite{brennerscott2008}) by the constant extension along the normal direction to $\hat{e}$.
From the trace theorem on Lipschitz domains\cite{brennerscott2008}, the trace operator is surjective from $W^{1,p'}(\hat{T})$ to $W^{1/p,p'}(\partial \hat{T})$, where $p'$ denotes the dual index to $p$, i.e. $1/p+1/p'=1$, $p>2$.
Further, the space $W^{1/p,p'}(\partial \hat{T})$ contains piecewise discontinuous functions over $\partial \hat{T}$ since $p>2$.
In particular, there exists a function $\hat{w}$ such that $\hat{w}=1$ on $\hat{e}$, $\hat{w}=0$ on $\partial\hat{T}/\hat{e}$, and $\|\hat{w}\|_{W^{1,p'}(\hat{T})}< \infty$.
The function~$\hat{w}\hat{q}_{k}$ belongs to~$W^{1,p'}(\hat{T})$.

Using a scaling argument, an integration by parts,  the H\"{o}lder inequality, and the norm equivalence
of polynomial functions with fixed degree on the reference triangle~$\hat{T}$, we have
\[
\begin{split}
&\left|\int_{e}(\textbf{v} \cdot\textbf{t}_{e})p_{k} \right|   
\lesssim h_{F}  \left|\int_{\hat{e}}(\hat{\textbf{v}} \cdot\hat{\textbf{t}}_{\hat{e}})\hat{p}_{k}\right|
= h_{F}  \left|\int_{\partial \hat{T}}(\hat{\textbf{v}} \cdot\hat{\textbf{t}}_{\hat{e}})(\hat{w}\hat{q}_{k})\right|\\
& = h_{F} \left|\int_{\hat{T}} \hat{\textrm{rot}}_{\hat{F}}\hspace{0.05cm}\hat{\textbf{v}}  (\hat{w}\hat{q}_{k})   -\int_{\hat{T}}  \hat{\textbf{v}}\cdot \textbf{curl}_{\hat{F}} \hspace{0.05cm} (\hat{w}\hat{q}_{k}) \right|\\
& \!\lesssim\!  h_{F}\left(\! \|\hat{\textrm{rot}}_{\hat{F}}\hat{\textbf{v}} \|_{\hat{T}} \|\hat{w}\hat{q}_{k}\|_{\hat{T}}
\!+\! \| \hat{\textbf{v}} \|_{L^{p}(\hat{T})} \vert \hat{w}\hat{q}_{k}\vert_{W^{1,p'}(\hat{T})}  \!\right)
\!\lesssim\! h_{F} \bigg(\! \|\hat{\textrm{rot}}_{\hat{F}}\hat{\textbf{v}} \|_{\hat{T}} \|\hat{w}\|_{\hat{T}}\| \hat{q}_{k} \|_{L^{\infty}(\hat{T})}\\
& \hspace{0.5cm}+ \| \hat{\textbf{v}} \|_{L^{p}(\hat{T})}\left(\|\hat{w}\|_{W^{1,p'}(\hat{T})} \|\hat{q}_{k}\|_{L^{\infty}(\hat{T})}+\|\hat{w}\|_{L^{p'}(\hat{T})}\| \hat{q}_{k}\|_{W^{1,\infty}(\hat{T})} \right) \bigg)\\
& \lesssim h_{F}\left( \|\hat{\textrm{rot}}_{\hat{F}}\hat{\textbf{v}} \|_{\hat{T}}+ \| \hat{\textbf{v}} \|_{L^{p}(\hat{T})}\right)\|\hat{w}\|_{W^{1,p'}(\hat{T})}\|\hat{q}_{k}\|_{L^{\infty}(\hat{T})} \\
& \lesssim \|p_{k}\|_{L^{\infty}(e)} \left( h_{F}^{1-2/p}\|  \textbf{v}  \|_{L^{p}(T)}+ h_{F}\| \textrm{rot}_{F}\textbf{v}\|_{T}  \right),
\end{split}
\]
which completes the proof of \eqref{L1BOUND_proof1}.
By taking $p=2/(1-\varepsilon)>2$ in \eqref{L1BOUND_proof1},
noting that $T\subseteq F$,
and  using the (scaled) embedding $H^{\varepsilon}(F)\hookrightarrow L^{p}(F)$, we get~\eqref{L1BOUND}.
\end{proof}
 
The DoFs interpolation operator $\widetilde{\textbf{I}}^{e}_{h}$ on the space $\textbf{V}_{k}^{e}(F)$ is well defined for each function $\textbf{v}$ in $\textbf{H}^{s}(F)\cap \textbf{H}(\textrm{rot}_{F},F)$ with $\textbf{v} \cdot \textbf{t}_{e}$ integrable on each edge. We impose
\begin{subequations} \label{eq:standard_interpolation}
\begin{align}  \label{eq:standard_interpolation1}
&  \int_{e}(\textbf{v}-\widetilde{\textbf{I}}^{e}_{h}\textbf{v})\cdot \textbf{t}_{e}p_{k}=0 &&\forall p_{k}\in \mathbb{P}_{k}(e),  \hspace{0.1cm} \forall e\subseteq \partial F;\\
\label{eq:standard_interpolation2}
&\int_{F} (\textbf{v}-\widetilde{\textbf{I}}^{e}_{h}\textbf{v})\cdot \textbf{x}_{F} p_{k}  =0 && \forall p_{k}\in \mathbb{P}_{k}(F); \\
\label{eq:standard_interpolation3}
&\int_{F} \textrm{rot}_{F}\hspace{0.05cm}(\textbf{v}-\widetilde{\textbf{I}}^{e}_{h}\textbf{v} ) p_{k-1}^{0}  =0 && \forall p_{k-1}^{0}\in \mathbb{P}_{k-1}^{0}(F)  \hspace{0.2cm} \textrm{only\ for}\  k>1.
 \end{align}
\end{subequations}

Next, we prove interpolation properties of the operator~$\widetilde{\textbf{I}}^{e}_{h}$.
\begin{thm}
\label{thm31edge2d}
For each $\textbf{v}\in \textbf{H}^{s}(F)$, $0 < s\leq {k+1}$,  with ${\rm rot}_{F}\hspace{0.05cm}\textbf{v}\in  H^{r}(F) $, $0\leq r\leq {k}$, and $\textbf{v} \cdot \textbf{t}_{e}$ integrable on each edge, we have
  \begin{align}
 \label{standardinterpolation12dedge}
 &\|\textbf{v}-\widetilde{\textbf{I}}^{e}_{h}\textbf{v}\|_{F}
 \lesssim   h_{F}^{s} |\textbf{v}|_{s,F}
 + h_F \| {\rm rot}_{F}\hspace{0.05cm}\textbf{v} \|_{F} ,\\
 \label{standardinterpolation22dedge}
 &\|{\rm rot}_{F}\hspace{0.05cm}(\textbf{v}-\widetilde{\textbf{I}}_{h}^{e} \textbf{v})\|_{F} \lesssim h_{F}^{r}|{\rm rot}_{F}\hspace{0.05cm}\textbf{v}|_{r,F}.
\end{align}
The second term on the right-hand side  of \eqref{standardinterpolation12dedge} can be neglected if $s \ge 1$.
\end{thm}
\begin{proof}
For each $p_{k-1}\in \mathbb{P}_{k-1}(F)$, we write
\[
\int_{F}\textrm{rot}_{F}\hspace{0.05cm}(\textbf{v} -\widetilde{\textbf{I}}^{e}_{h}\textbf{v})p_{k-1} 
\overset{\text{IBP},\eqref{eq:standard_interpolation1},\eqref{eq:standard_interpolation3}}{=} 0. 
\]
This and the fact that $\textrm{rot}_{F}\hspace{0.05cm}( \widetilde{\textbf{I}}^{e}_{h}\textbf{v})\in \mathbb{P}_{k-1}(F)$ imply that
\begin{align}
 \label{standardinterpolation22dedge_proof1}
 \textrm{rot}_{F}\hspace{0.05cm}(\widetilde{\textbf{I}}^{e}_{h}\textbf{v}) =  \Pi_{k-1}^{0,F}(\textrm{rot}_{F}\hspace{0.05cm} \textbf{v}).
\end{align}
Then, \eqref{standardinterpolation22dedge} follows from standard polynomial approximation properties.

Next, we focus on \eqref{standardinterpolation12dedge}.
By \eqref{eq:standard_interpolation1} and  the fact that $\widetilde{\textbf{I}}^{e}_{h}\textbf{v} \cdot \textbf{t}_{e}\in \mathbb{P}_{k}(e)$, we have
\begin{align}
\label{projectiondefine}
 \Pi_{k}^{0,e}(\textbf{v}\cdot \textbf{t}_{e})=\widetilde{\textbf{I}}^{e}_{h}\textbf{v}\cdot \textbf{t}_{e} \hspace{0.3cm}    \forall e\subseteq \partial F.
\end{align}
Since $\mathbf{\Pi}_{k}^{0,F} \mathbf{v}\in(\mathbb{P}_{k}(F))^{2}\subseteq  \textbf{V}_{k}^{e}(F)$, we have
\[
\begin{split}
& \| \mathbf{\Pi}_{k}^{0,F} \mathbf{v}-\widetilde{\textbf{I}}^{e}_{h}\textbf{v} \|_{F}
\overset{\eqref{prioribound2dedge}}{\lesssim} 
h_{F}\|\textrm{rot}_{F}\hspace{0.05cm}(\mathbf{\Pi}_{k}^{0,F} \mathbf{v}-\widetilde{\textbf{I}}^{e}_{h}\textbf{v})\|_{F}\\
&\hspace{0.5cm} +h_{F}^{\frac12}\|(\mathbf{\Pi}_{k}^{0,F} \mathbf{v}-\widetilde{\textbf{I}}^{e}_{h}\textbf{v})\cdot\textbf{t}_{\partial F}\|_{\partial F}
+\sup_{p_{k}\in \mathbb{P}_{k}(F)} \frac{\int_{F}  (\mathbf{\Pi}_{k}^{0,F} \mathbf{v} - \widetilde{\textbf{I}}^{e}_{h}\textbf{v}) \cdot \textbf{x}_{F} p_{k}}{\|\textbf{x}_{F} p_{k}\|_{F}}.
\end{split}
\]
As for the boundary term, also using~\eqref{projectiondefine}, we have
\[
\begin{aligned}
h_{F}^{\frac12} \|(\mathbf{\Pi}_{k}^{0,F} \mathbf{v} - \widetilde{\textbf{I}}^{e}_{h} \textbf{v}) \cdot \textbf{t}_{\partial F} \|_{\partial F}
& \lesssim 
h_{F}^{\frac12} \sum_{e \subset \partial F} \sup_{p_k \in \mathbb P_k(e)}
\frac{\big((\mathbf{\Pi}_{k}^{0,F} \mathbf{v} - \widetilde{\textbf{I}}^{e}_{h} \textbf{v}) 
\cdot \textbf{t}_{\partial F} , p_k \big)_{e}}{\Vert p_k\Vert_{e}} \\
& =
h_{F}^{\frac12} \sum_{e \subset \partial F} \sup_{p_k \in \mathbb P_k(e)}
\frac{\big((\textbf{v} - \mathbf{\Pi}_{k}^{0,F} \mathbf{v} ) \cdot \textbf{t}_{\partial F} , p_k \big)_{e}}{\Vert p_k\Vert_{e}}.
\end{aligned}
\]
Using~\eqref{L1BOUND} with~$\varepsilon=s$ and a polynomial inverse inequality, we deduce
\[
\begin{split}
h_{F}^{\frac12} \|(\mathbf{\Pi}_{k}^{0,F} \mathbf{v} - \widetilde{\textbf{I}}^{e}_{h} \textbf{v}) \cdot \textbf{t}_{\partial F} \|_{\partial F}
\lesssim & \|\textbf{v}-\mathbf{\Pi}_{k}^{0,F} \mathbf{v}\|_{F}
+
h_{F}^{s} \vert \textbf{v} - \mathbf{\Pi}_{k}^{0,F} \mathbf{v} \vert_{s,F}\\
& +
h_{F} \Vert \textrm{rot}_{F} \hspace{0.05cm} (\textbf{v} - \mathbf{\Pi}_{k}^{0,F} \mathbf{v}) \Vert_{F}.
\end{split} 
\]
Further, the definition of~$\widetilde{\textbf{I}}^{e}_{h}$ in~\eqref{eq:standard_interpolation} entails
\[
\int_{F}  (\mathbf{\Pi}_{k}^{0,F} \mathbf{v} - \widetilde{\textbf{I}}^{e}_{h}\textbf{v}) \cdot \textbf{x}_{F} p_{k}
=
\int_{F}  (\mathbf{\Pi}_{k}^{0,F} \mathbf{v} - \textbf{v}) \cdot \textbf{x}_{F} p_{k}.
\]
Thus, we write
\begin{eqnarray}
\begin{aligned} \label{standardinterpolation12dedge_proof1111}
\| \mathbf{\Pi}_{k}^{0,F} \mathbf{v}-\widetilde{\textbf{I}}^{e}_{h}\textbf{v} \|_{F}
\lesssim
& \|\textbf{v}-\mathbf{\Pi}_{k}^{0,F} \mathbf{v}\|_{F}
+ h_{F}^{s} \vert \textbf{v} - \mathbf{\Pi}_{k}^{0,F} \mathbf{v} \vert_{s,F}\\
& + h_{F}\|\textrm{rot}_{F}\hspace{0.05cm}(\textbf{v}-\mathbf{\Pi}_{k}^{0,F} \mathbf{v})\|_{F}
+ h_{F}\|\textrm{rot}_{F}\hspace{0.05cm}( \mathbf{v}-\widetilde{\textbf{I}}^{e}_{h}\textbf{v})\|_{F}. 
\end{aligned}
\end{eqnarray}
If~$s\ge 1$, then we apply~\eqref{standardinterpolation12dedge_proof1111}, \eqref{standardinterpolation22dedge} with~$r=s-1$,
and standard polynomial approximation properties, leading to
\begin{eqnarray}
\begin{aligned}[b] \label{eq50}
\| \mathbf{\Pi}_{k}^{0,F} \textbf{v}
& -\widetilde{\textbf{I}}^{e}_{h}\textbf{v}\|_{F}
\!\lesssim\! \|\textbf{v} \!-\! \mathbf{\Pi}_{k}^{0,F} \mathbf{v}\|_{F}
\!+\! h_{F}  | \textbf{v} - \mathbf{\Pi}_{k}^{0,F} \mathbf{v}|_{1,  F}
\!+\! h_{F}\|\textrm{rot}_{F}\hspace{0.05cm}(\textbf{v} \!-\! \mathbf{\Pi}_{k}^{0,F} \mathbf{v})\|_{F}\\
& \!+\! h_{F}\|\textrm{rot}_{F}\hspace{0.05cm}( \mathbf{v} \!-\! \widetilde{\textbf{I}}^{e}_{h}\textbf{v})\|_{F}
\lesssim h_{F}^{s}\left(|\textbf{v}|_{s,F}
+ |\textrm{rot}_{F}\hspace{0.05cm}\mathbf{v}|_{s-1,F}\right) 
\lesssim  h_{F}^{s} |\textbf{v}|_{s,F}.
\end{aligned}
 \end{eqnarray}
Instead, if~$0 < s < 1$, we replace the term~$\mathbf{\Pi}_{k}^{0,F} \mathbf{v}$ by~$\mathbf{\Pi}_{0}^{0,F} \mathbf{v}$ in~\eqref{standardinterpolation12dedge_proof1111}.
Then, we apply \eqref{standardinterpolation22dedge} with~$r=0$
and standard polynomial approximation properties, yielding
\begin{eqnarray}
\begin{aligned}[b] \label{standardinterpolation12dedge_proof3492}
\| \mathbf{\Pi}_{0}^{0,F} \textbf{v}-\widetilde{\textbf{I}}^{e}_{h}\textbf{v}\|_{F} 
& \!\lesssim\! \|\textbf{v} \!-\!\mathbf{\Pi}_{0}^{0,F} \mathbf{v}\|_{F}
\!+\! h_{F}^{s}| \textbf{v} - \mathbf{\Pi}_{0}^{0,F} \mathbf{v}|_{s,F} 
\!+\! h_{\!F}\|\textrm{rot}_{F}\hspace{0.05cm}(\textbf{v} \!-\! \mathbf{\Pi}_{0}^{0,F} \mathbf{v})\|_{F} \\
& \quad \!+\! h_{\!F}\|\textrm{rot}_{F}\hspace{0.05cm}( \mathbf{v} \!-\! \widetilde{\textbf{I}}^{e}_{h} \textbf{v})\|_{F}
\lesssim  h_{F}^{s} |\textbf{v}|_{s,F} + 
h_F \| \textrm{rot}_{F}\hspace{0.05cm}\mathbf{v} \|_{F}.
\end{aligned}
\end{eqnarray}
Bounds~\eqref{eq50} and~\eqref{standardinterpolation12dedge_proof3492}
combined with a triangle inequality
and standard polynomial approximation results
prove the assertion \eqref{standardinterpolation12dedge}.
\end{proof}

\subsection{Serendipity  edge virtual element space on polygons}
\label{sec4_1}
As in Refs.~\cite{beiraobrezzidasimaru2018siam,bebremaru2016CFSerendipity,bebremaru2016RLMASerendipity}, we set $\eta_{F}$
as the minimum number of straight lines necessary to cover the boundary of~$F$ and define $\beta_{F}:={k+1}-\eta_{F}$.
Next, we introduce a well defined projection $\mathbf{\Pi}_{S}^{e}:  \textbf{V}_{k}^{e}(F)\rightarrow  (\mathbb{P}_{k}(F))^{2}$ as\cite{beiraobrezzidasimaru2018siam}
  \begin{subequations}
 \label{eq:standard_projection}
\begin{align}
 \label{eq:standard_projection1}
&  \int_{\partial F}[(\textbf{v}_{h}-\mathbf{\Pi}_{S}^{e} \textbf{v}_{h})\cdot \textbf{t}_{\partial F}][\bm{\nabla}_{F} p_{k+1}\cdot \textbf{t}_{\partial F}]=0 \hspace{0.4cm}\forall p_{k+1}\in \mathbb{P}_{k+1}(F); \\
\label{eq:standard_projection2}
&\int_{\partial F} (\textbf{v}_{h}- \mathbf{\Pi}_{S}^{e} \textbf{v}_{h})\cdot \textbf{t}_{\partial F}  =0;\\
\label{eq:standard_projection3}
&\int_{F} \textrm{rot}_{F}\hspace{0.05cm}(\textbf{v}_{h}-  \mathbf{\Pi}_{S}^{e} \textbf{v}_{h}) p_{k-1}^{0}   =0 \hspace{0.4cm}\forall p_{k-1}^{0}\in \mathbb{P}_{k-1}^{0}(F) \hspace{0.2cm} \textrm{only\ for}\ k>1;\\
 \label{eq:standard_projection4}
&\int_{F}  (\textbf{v}_{h}-  \mathbf{\Pi}_{S}^{e} \textbf{v}_{h})\cdot \textbf{x}_{F} p_{\beta_F}    =0\hspace{0.6cm}\forall p_{\beta_F} \in \mathbb{P}_{\beta_F}(F) \hspace{0.2cm} \textrm{only\ for}\ \beta_{F}\geq 0.
 \end{align}
\end{subequations}
\begin{remark}
To handle the serendipity VEM  in the present section we assume the additional (uniform) convexity condition  (\textbf{MC})  in Section \ref{sec21}.
For the particular case $\beta_F < 0$, such a condition could be relaxed at the price of additional technicalities that we prefer to avoid.
\end{remark}
Based on the space~$\textbf{V}_{k}^{e}(F)$ in~\eqref{finitspace} and the projection operator~$\mathbf{\Pi}_{S}^{e}$ in~\eqref{eq:standard_projection},
we define the serendipity edge virtual element space on the face $F$ as
 \begin{align}
 \label{svspace}
 \textbf{SV}_{k}^{e}(F)=\left\{\textbf{v}_{h}\in \textbf{V}_{k}^{e}(F): \int_{F}(\textbf{v}_{h}-\mathbf{\Pi}_{S}^{e} \textbf{v}_{h})\cdot\textbf{x}_{F} p =0\hspace{0.2cm} \forall p  \in \mathbb{P}_{\beta_{F}|k}(F)\right\},
\end{align}
where~$\mathbb{P}_{\beta_{F}|k}(F)$ is chosen to satisfy
$\mathbb P_k(F) = \mathbb P_{\beta_F} \oplus \mathbb P_{\beta_{F}|k}(F)$.
It can be checked that $ (\mathbb{P}_{k}(F))^{2}\subseteq \textbf{SV}_{k}^{e}(F) \subseteq \textbf{V}_{k}^{e}(F) $.
A set of unisolvent  DoFs $\{\textrm{dof}_{i}^{F}\}_{i=1}^{N_{d}}$
for the space $ \textbf{SV}_{k}^{e}(F)$
with  $N_{d}=N_{e}\pi_{k,1}+\pi_{k-1,2}+\pi_{\beta_{F},2}-1$ is given by \eqref{dof1}, \eqref{dof3}, and the internal moments
\begin{align}
\label{DoFsinternal}
\int_{F} \textbf{v}_{h}\cdot \textbf{x}_{F} p_{\beta_{F}}    \hspace{0.3cm} \forall p_{\beta_{F}}\in \mathbb{P}_{\beta_{F}}(F)  \hspace{0.2cm} \textrm{only\ for}\ \beta_{F}\geq 0.
\end{align}
This choice reduces the internal  DoFs of the standard edge virtual element space $\textbf{V}_{k}^{e}(F)$ by~$(\pi_{k,2}-\pi_{\beta_{F},2})$.
Notably, we   can compute the moments of order up to $\beta_{F}$ given in \eqref{DoFsinternal},
whereas  the remaining moments of order up to $k$ can be computed by those of the projection $\mathbf{\Pi}_{S}^{e}$; see \eqref{svspace}.

By Proposition 5.2 in Ref.~\cite{lourencobrezzidassi2017cmame}, we have that a set of unisolvent DoFs  $\{\textrm{DoF}^{F}_{i}\}_{i=1}^{N_{D} }$ with $N_{D}=2\pi_{k,2}$  for the space $ (\mathbb{P}_{k}(F))^{2}$
is given by the functionals used to define~$\mathbf{\Pi}_{S}^{e}$ in~\eqref{eq:standard_projection}.

For  sufficiently large  constants $\gamma, \hat{\gamma}\in \mathbb{R}^{+}$,  which we shall fix in the proofs of Corollary~\ref{cor31}
and Lemma~\ref{lemma41} below,
we introduce a norm $\interleave \cdot\interleave_{F}$  on~$(\mathbb{P}_{k}(F))^{2}$ induced by~\eqref{eq:standard_projection}:
 \begin{equation}
\begin{aligned}
 \label{newnorm1}
& \interleave \textbf{s}_{k}\interleave_{F}:=\widetilde{\gamma}\left| \int_{\partial F}  \textbf{s}_{k} \cdot \textbf{t}_{\partial F}  \right | +\gamma\sup_{p_{k-1}^{0}\in \mathbb{P}_{k-1}^{0}(F)}\frac{h_{F} \int_{F} \textrm{rot}_{F}\hspace{0.05cm} \textbf{s}_{k} p_{k-1}^{0}   }{\|p_{k-1}^{0}\|_{F} } \\ 
& \quad +\!
\hat{\gamma}
\!\!\!\!\!\!\!\!\sup_{p_{k+1}\in \mathbb{P}_{k+1}(F)}\!\!\!\!\!\!\!\!
\frac{h_{F}^{\frac12}\!
\int_{\partial F}(\textbf{s}_{k} \!\cdot\! \textbf{t}_{\partial F})
(\bm{\nabla}_{F} p_{k+1} \!\cdot\! \textbf{t}_{\partial F}) }{\|\bm{\nabla}_{F} p_{k+1}\cdot \textbf{t}_{\partial F}\|_{\partial F}}
+
\!\!\!\!\!\!\!\!\sup_{p_{\beta_{F}}\in \mathbb{P}_{\beta_{F}}(F)}\!\!\!\!\!\!\!\!
\frac{ h_{F}^{-1} \int_{F} \textbf{s}_{k} \!\cdot\! \textbf{x}_{F} p_{\beta_{F}}  }{\|p_{\beta_{F}} \|_{F}},
\end{aligned}
\end{equation}
where $\widetilde{\gamma}:= \gamma h_{F}/|F|^{\frac12}$.

By the mesh regularity assumptions in Section~\ref{sec21}, $h_{F}/|F|^{\frac12}$ is a uniformly bounded constant.
Further, the operator $\interleave \cdot\interleave_{F}$ can be applied to all sufficiently smooth functions.

We first  prove  a critical polynomial estimate that we shall employ in the following analysis.
\begin{lem} \label{pro31}
If the assumption (\textbf{MC}) in Section~\ref{sec21} is valid,
then the following bound holds true:
\begin{align} \label{eq:polunomialimportantcase1}
\| p_{k}\|_{F}
\lesssim h_{F}^{\frac12}\|  p_{k} \|_{\partial F}+\sup_{p_{\beta_{F}}\in \mathbb{P}_{\beta_{F}}(F)}\frac{  \int_{F}  {p}_{k} p_{\beta_{F}} }{\|p_{\beta_{F}} \|_{F}} \hspace{0.5cm} \forall p_{k}\in \mathbb{P}_{k}(F),
 \end{align}
where~$C$ only depends on~$\varepsilon$, $k$, and the shape-regularity parameter~$\rho$.
\end{lem}
\begin{proof}
It suffices to prove the result  when $h_{F}=1$ and then use a scaling argument.
It is not  restrictive to assume that $F$ has a vertex in the origin of the~$[x,y]$ coordinate axes and an edge lies on the ``$y=0$'' axis.
Given any vertex~$v_{i}$ of~$F$, we denote its coordinates by $[v_{i,x},v_{i,y}]$.
We define the set of admissible polygons
\begin{equation*}
\begin{aligned}
\textrm{S}:=  \big\{&F:\   F\  \textrm{is\ a\ convex\ polygon\ with}\ \eta_{F} \ \textrm{edges\ and\ vertices\  counter-clockwise }  \\& \textrm{    ordered } \{v_1,v_2,\cdots,v_{\eta_{F}}\} \textrm{ with } v_{1}=(0,0), v_{2,y}=0; \textrm{ furthermore } h_{F}=1 ,
\\& \ h_{e}\geq \rho\ \forall e\subseteq \partial F, \ \varepsilon <\theta< \pi-\varepsilon\  \textrm{for\ each\ internal\  angle}\  \theta\ \textrm{of}\ F \big\}.
\end{aligned}
\end{equation*}
We also define the (injective) application $\mathcal{I}: \textrm{S} \rightarrow \mathbb{R}^{2\eta_{F}}$  by
\[
F \longmapsto
[v_{1,x},v_{1,y},v_{2,x},v_{2,y},\cdots,v_{\eta_{F},x},v_{\eta_{F},y}].
\]
Under the geometric assumptions of Section~\ref{sec21},
$\mathcal{I}(\textrm{S})$ is a bounded and closed subset in $\mathbb{R}^{2\eta_{F}}$.
 \begin{figure}
\begin{center}
 \includegraphics[scale=0.3]{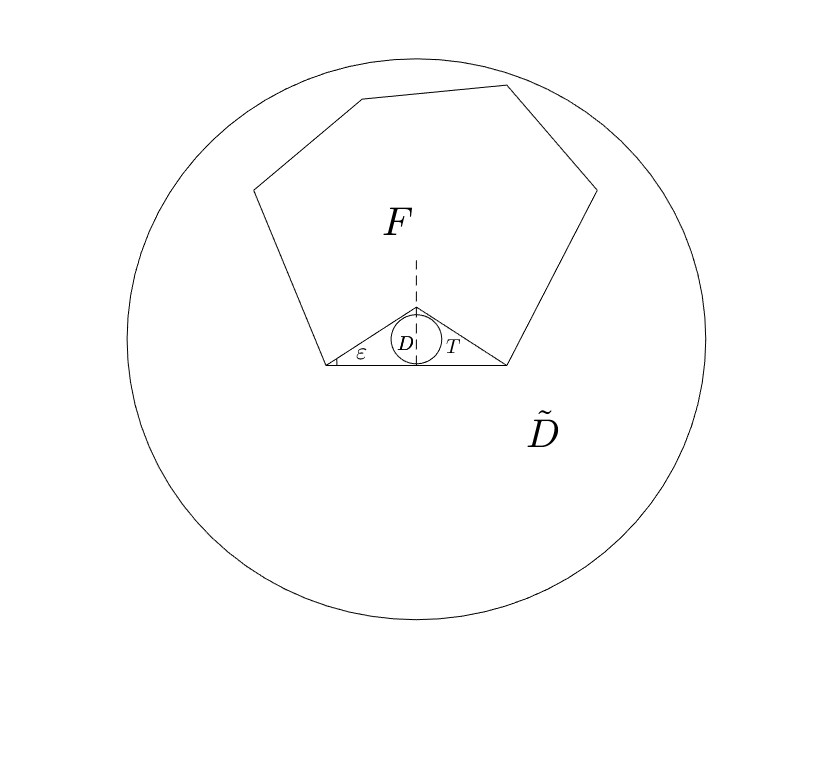}
\end{center}
 \setlength{\abovecaptionskip}{-1.3cm}
\caption{The sample figure on the element $F\in \textrm{S}$.}
\label{fig1}
\end{figure}
For each polygon $F\in$ S, we denote  the edge connecting $v_{i}$ to $v_{i+1}$ by $e_i$, with the usual notation $v_{\eta_{F}+1}=v_{1}$.
By the assumptions that $h_{e}\geq \rho$ $\forall e\subseteq \partial F$ and    each internal  angle $\theta$ of the convex polygon $F$ satisfies $\varepsilon <\theta< \pi-\varepsilon$,
there exists an isosceles triangle $T$ with   basis $e_{1}$ and  height $h\geq\beta$ (for a uniformly  positive constant $\beta$) that is contained in all $F$ of $S$.
Therefore, it exists a disk $D\subseteq T \subseteq F$ such that its radius is uniformly bounded by $h_{F}$ from below.
Meanwhile, we denote the disk of radius $R=1$ that is concentric with $D$ and containing $F$ by~$\widetilde{D}$;
see Figure~\ref{fig1} for a graphical example.
We have
 \begin{align}
 \label{liftingoperator235}
    D\subseteq F \subseteq \widetilde{D} \qquad \forall F\in \textrm{S}.
 \end{align}
We are now in the position of proving~\eqref{eq:polunomialimportantcase1}  by contradiction.
If \eqref{eq:polunomialimportantcase1} were false, then we could find a sequence of elements $\{F_{m}\}_{m\in \mathbb{N}}$ in S
and a sequence of polynomials $\{p_{m}\}_{m\in \mathbb{N}}\in  \mathbb{P}_{k}(F_{m}) $  such that
\begin{align} \label{contradictioncondition}
 \|p_{m}\|_{F_{m}}=1, 
 \qquad 
 \|p_{m}\|_{\partial F_{m}}\leq \frac{1}{m},
 \qquad 
  \!\!\!\!\!\!\sup_{p_{\beta}\in \mathbb{P}_{\beta_{F}}(F)}\!\!\!\!
  \frac{  \int_{F}  {p}_{m} p_{\beta_{F}} }{\|p_{\beta_{F}} \|_{F}}\leq \frac{1}{m}  \hspace{0.2cm}\forall m\in \mathbb{N}.
\end{align}
Since $\mathcal{I}$(S) is bounded and closed, there exists a subsequence   $\mathcal{I}(F_{m_{j}})\subseteq \mathbb{R}^{2\eta_{F}}$ that converges to  $\mathcal{I}(F)$ for some $F\in \textrm{S}$  as $j\rightarrow+\infty$.
In particular, all vertexes of $F_{m_{j}}$ converge  to those of  $F\in \textrm{S}$  as $j\rightarrow+\infty$.
By \eqref{liftingoperator235} and \eqref{contradictioncondition}, we have
\[
     \|p_{m_{j}}\|_{D}\leq  \|p_{m_{j}}\|_{F_{m_{j}}}=1,
\]
which implies that $\{p_{m_{j}}\}_{j\in\mathbb{N}}\in  \mathbb{P}_{k}(D) $ is a bounded sequence.

Then, there exists a subsequence $\{p_{m_{j_{l}} }\}_{l\in\mathbb{N}}$ such that $p_{m_{j_{l}} }\rightarrow p_{k}\in  \mathbb{P}_{k}(D) $ as $l\rightarrow+\infty$.
By \eqref{liftingoperator235} and standard polynomial properties, it follows that
\[
  1=\|p_{m_{j_{l}} }\|_{F_{m_{j_{l}} }}\leq \|p_{m_{j_{l}} }\|_{\widetilde{D}} \lesssim  \|p_{m_{j_{l}} }\|_{D}.
\]
By taking $l\rightarrow +\infty$, this yields
\begin{align}
\label{contrac11}
    p_{k}\neq 0\ \textrm{in}\  D\subseteq F.
\end{align}
Since the ordered vertices of $F_{m_j}$ converge to those of $F$ due to $\mathcal{I}(  F_{m_j})\rightarrow\mathcal{I}(F)$,
we have the boundary convergence $\partial F_{m_j} \rightarrow \partial F $  as $j\rightarrow+\infty$.
By~\eqref{contradictioncondition}, we know that the subsequences  $\{p_{m_{j_{l}}} \}_{l\in\mathbb{N}}$ and $\{\partial F_{m_{j_{l}}}\}_{l\in\mathbb{N}}$  satisfy
\[
    \|p_{m_{j_{l}}}\|_{\partial F_{m_{j_{l}}}} \leq \frac{1}{m_{j_{l}}},
\]
which entails that  $p_{k}{|_{\partial F}}=0$  by  taking $l\rightarrow + \infty$.
Then, there exists $\hat{p}_{\beta_{F}}\in \mathbb{P}_{\beta_{F}}(F)$ such that
\begin{align} \label{eq731}
 p_{k}=b_{\eta_{F}}  \hat{p}_{\beta_{F}},
\end{align}
where $b_{\eta_{F}}$ is the polynomial of degree  $\eta_{F}$ that vanishes identically on  $\partial F$ and  is equal to 1 at the barycenter   of the element~$F$.
Since~$F$ is convex, we have $b_{\eta_{F}}>0$ in $F$;
see, e.g., Ref.~\cite{beiraobrezzidasimaru2018cmame}.
Letting $\ell \rightarrow +\infty$, recalling the last inequality of \eqref{contradictioncondition},
and combining the resulting inequality and~\eqref{eq731} together,
we arrive at
\[
 \int_{F}b_{\eta_{F}}  (\hat{p}_{\beta_{F}})^{2} =0,
\]
which implies  that $\hat{p}_{\beta_{F}}\equiv0$. By \eqref{eq731}, it follows  that
\[
    p_{k}\equiv 0\  \textrm{in} \ F.
\]
Yet, this and~\eqref{contrac11} contradict each  other, whence the assertion follows.
\end{proof}
\begin{cor} \label{cor31}
Under the same assumptions of Lemma \ref{pro31},
for $\hat\gamma$ sufficiently large and independent of~$F$,
and each $p^{0}_{k}\in \mathbb{P}_{k}^{0}(F)$, we have
\[
\|\bm{\nabla}_{F} p^{0}_{k}\|_{F}
\lesssim \hat\gamma h_{F}^{\frac12}\|\bm{\nabla}_{F} p^{0}_{k}\cdot \textbf{t}_{\partial F}\|_{\partial F}+
 \!\!\!\!\!\!\sup_{p_{\beta_{F}}\in \mathbb{P}_{\beta_{F}}(F)}\!\!\!\!\!\!
 \frac{ h_{F}^{-1} \int_{F}  \bm{\nabla}_{F} p^{0}_{k} \cdot \textbf{x}_{F}p_{\beta_{F}} }{\|p_{\beta_{F}} \|_{F}}.
 \]
\end{cor}
\begin{proof}
We write
\[
\begin{split}
&\hat\gamma h_{F}^{\frac12}\|\bm{\nabla}_{F} p^{0}_{k}\cdot \textbf{t}_{\partial F}\|_{\partial F}
+\sup_{p_{\beta_{F}}\in \mathbb{P}_{\beta_{F}}(F)}\frac{ h_{F}^{-1} \int_{F}  \bm{\nabla}_{F} p^{0}_{k} \cdot \textbf{x}_{F}p_{\beta_{F}} }{\|p_{\beta_{F}} \|_{F}}\\
& \!\!\!\!\overset{\eqref{poincarefridine},\text{IBP}}{\geq }\!\!\!\!
C' \hat\gamma h_{F}^{-\frac12} \| p^{0}_{k} \|_{\partial F}
+ \!\!\!\!\!\!\!\!\!\!\sup_{p_{\beta_{F}}\in \mathbb{P}_{\beta_{F}}(F)}\!\!\!\!\!\!\!\!\!\!
\frac{h_{F}^{-1} \int_{F}  p^{0}_{k}\textrm{div} (\textbf{x}_{F} p_{\beta_{F}}) 
-h_{F}^{-1} \int_{\partial F}  p^{0}_{k} \textbf{x}_{F}\cdot \textbf{n}_{\partial F} p_{\beta_{F}} }{\|p_{\beta_{F}} \|_{F}}\\
& \overset{\eqref{decompo1}}{\geq}
(\hat\gamma C'-C'')h_{F}^{-\frac12} \| p^{0}_{k} \|_{\partial F} 
+  \!\!\!\!\!\!\!\!\!\!\sup_{p'_{\beta_{F}}\in \mathbb{P}_{\beta_{F}}(F)}\!\!\!\!\!\!\!\!
\frac{h_{F}^{-1} \int_{F}  p^{0}_{k}  p'_{\beta_{F}}}{\|p'_{\beta_{F}} \|_{F}}
\!\!\overset{\eqref{eq:polunomialimportantcase1}}{\geq}\!\!
C h_{F}^{-1}\| p^{0}_{k} \|_{F} 
\!\!\overset{\eqref{Polynomialinverseestimates}}{\geq}\!\!
C\|\bm{\nabla}_{F} p^{0}_{k}\|_{F},
\end{split}
\]
where we have chosen $\hat\gamma$ sufficiently large.
\end{proof}

Next, we prove lower and upper bounds on the operator $\interleave \cdot \interleave_{F}$ introduced in~\eqref{newnorm1}
with respect to the $L^{2}$ norm $\|\cdot\|_{F}.$
\begin{lem}
\label{lemma41}
For given $\varepsilon>0$,  the following bounds are valid:
\begin{align}  \label{equivalencenorm}
&\|\textbf{s}_{k}\|_{F} \lesssim \interleave \textbf{s}_{k} \interleave_{F} \hspace{0.5cm} \forall \textbf{s}_{k}\in (\mathbb{P}_{k}(F))^{2}, \\
   \label{equivalencenorm2}
&\interleave\hspace{-0.08cm}\textbf{v}_{h}\interleave_{F}\lesssim  \|\textbf{v}_{h}\|_{F} \hspace{0.5cm} \forall \textbf{v}_{h}\in \textbf{V}_{k}^{e}(F),\\
     \label{equivalencenorm3}
&\interleave \hspace{-0.08cm}\textbf{v}\interleave_{F}\lesssim  \|\textbf{v}\|_{F}+h^{\varepsilon}_{F}|\textbf{v}|_{\varepsilon,F}+h_{F}\|{\rm rot}_{F}\textbf{v}\|_{F} \hspace{0.2cm} \forall \textbf{v}\in \textbf{H}^{\varepsilon}(F)\cap \textbf{H}({\rm rot}_{F}, F).
\end{align}
\end{lem}
 \begin{proof}
First, we prove \eqref{equivalencenorm}.
From \eqref{decompo2} and  $\textbf{s}_{k} \in (\mathbb{P}_{k}(F))^{2}$,
there exist $q_{k+1}^{{0}}\in \mathbb{P}^{0}_{k+1}(F)$ and $q_{k-1}\in  \mathbb{P}_{k-1}(F)$ such that
\begin{align}  \label{thisdecom}
 \textbf{s}_{k}= \bm{\nabla}_{F} q^{0}_{k+1}+\textbf{x}_{F}^{\perp}q_{k-1}.
 \end{align}
Define $\widetilde{\textrm{rot}_{F}\hspace{0.05cm}\textbf{s}_{k}}:=\textrm{rot}_{F}\hspace{0.05cm}\textbf{s}_{k}-\frac{1}{|F|}\int_{F}\textrm{rot}_{F}\hspace{0.05cm}\textbf{s}_{k}$ and observe that
\begin{equation} \label{orthogonality-average-rotor}
\int_{F} \textrm{rot}_{F}\hspace{0.05cm} \textbf{s}_{k} \widetilde{\textrm{rot}_{F}\hspace{0.05cm}\textbf{s}_{k}} 
=\int_{F} \widetilde{\textrm{rot}_{F}\hspace{0.05cm}\textbf{s}_{k}} \widetilde{\textrm{rot}_{F}\hspace{0.05cm}\textbf{s}_{k}} .
\end{equation}
By taking $p_{k-1}^{0}=\widetilde{\textrm{rot}_{F}\hspace{0.05cm}\textbf{s}_{k}}$ and $p_{k+1}=q_{k+1}^{0}$ that realize~\eqref{thisdecom}
in the second and third terms involving supremum of \eqref{newnorm1}
and using the property~\eqref{orthogonality-average-rotor},
we write
\begin{eqnarray}
\begin{aligned} \label{equivalencenorm_proof2}
\interleave \textbf{s}_{k}\interleave_{F}  &\geq  \widetilde{\gamma}\left| \int_{\partial F}  \textbf{s}_{k} \cdot \textbf{t}_{\partial F}  \right | + \frac{\gamma h_{F} \int_{F} \widetilde{\textrm{rot}_{F}\hspace{0.05cm}\textbf{s}_{k}} \widetilde{\textrm{rot}_{F}\hspace{0.05cm}\textbf{s}_{k}}  }{\|\widetilde{\textrm{rot}_{F}\hspace{0.05cm}\textbf{s}_{k}}\|_{F} }\\
&\hspace{0.4cm}  + \frac{\hat{\gamma}h_{F}^{\frac12} \int_{\partial F}((\bm{\nabla}_{F} \hspace{0.05cm}q^{0}_{k+1}+\textbf{x}_{F}^{\perp}q_{k-1})\cdot \textbf{t}_{\partial F})(\bm{\nabla}_{F} q^{0}_{k+1}\cdot \textbf{t}_{\partial F}) }{\|\bm{\nabla}_{F} q^{0}_{k+1}\cdot \textbf{t}_{\partial F}\|_{\partial F}} \\
&\hspace{0.4cm}+\sup_{p_{\beta_{F}}\in \mathbb{P}_{\beta_{F}}(F)}\frac{h_{F}^{-1} \int_{F} (\bm{\nabla}_{F} q^{0}_{k+1}+\textbf{x}_{F}^{\perp}q_{k-1})\cdot \textbf{x}_{F} p_{\beta_{F}}   }{\|p_{\beta_{F}} \|_{F}}\\
& \geq  \widetilde{\gamma}\left| \int_{\partial F}  \textbf{s}_{k} \cdot \textbf{t}_{\partial F}  \right | +\gamma  h_{F}  \|\widetilde{\textrm{rot}_{F}\hspace{0.05cm}\textbf{s}_{k}}\|_{F}+\hat{\gamma} h_{F}^{\frac12} \|\bm{\nabla}_{F} q^{0}_{k+1}\cdot \textbf{t}_{\partial F}\|_{\partial F} \\&\hspace{0.4cm}   - \hat{\gamma}h_{F}^{\frac12}  \| \textbf{x}_{F}^{\perp}q_{k-1}\cdot \textbf{t}_{\partial F}\|_{\partial F}  +  \sup_{p_{\beta_{F}}\in \mathbb{P}_{\beta_{F}}(F)}\frac{h_{F}^{-1} \int_{F} \bm{\nabla}_{F} q^{0}_{k+1}\cdot \textbf{x}_{F} p_{\beta_{F}}    }{\|p_{\beta_{F}} \|_{F}}.
 \end{aligned}
\end{eqnarray}
We estimate every term  on the right-hand side of~\eqref{equivalencenorm_proof2} from below.
We begin with the term involving $\|\widetilde{\textrm{rot}_{F}\hspace{0.05cm}\textbf{s}_{k}}\|_{F}$:
\begin{eqnarray}
\begin{aligned} \label{equivalencenorm_proof3}
& \|\textrm{rot}_{F}\hspace{0.05cm}\textbf{s}_{k}\|_{F} 
\leq   \|\widetilde{\textrm{rot}_{F}\hspace{0.05cm}\textbf{s}_{k}}\|_{F} +\left\|\frac{1}{|F|}\int_{F}\textrm{rot}_{F}\hspace{0.05cm}\textbf{s}_{k}  \right\|_{F}\\
&= \|\widetilde{\textrm{rot}_{F}\hspace{0.05cm}\textbf{s}_{k}}\|_{F} +\frac{1}{|F|}\left\|\int_{\partial F}  \textbf{s}_{k}\cdot \textbf{t}_{\partial F} \right\|_{F}
=  \|\widetilde{\textrm{rot}_{F}\hspace{0.05cm}\textbf{s}_{k}}\|_{F} +\frac{1}{ |F|^{\frac12}} \left|\int_{\partial F}  \textbf{s}_{k}\cdot \textbf{t}_{\partial F} \right|.
 \end{aligned}
\end{eqnarray}
Further, using \eqref{Polynomialinverseestimates}  and the fact that $\textrm{rot}_{F}\hspace{0.05cm}\textbf{s}_{k}=\textrm{rot}_{F}\hspace{0.05cm}(\textbf{x}_{F}^{\perp}q_{k-1}) $,   we obtain
\[
\begin{split}
& h_{F}^{\frac12}  \| \textbf{x}_{F}^{\perp}q_{k-1} \!\cdot\! \textbf{t}_{\partial F}\|_{\partial F}
\!\lesssim\! \| \textbf{x}_{F}^{\perp}q_{k-1}\|_{F}
\!\lesssim\!  h_{F}\| \textrm{rot}_{F} \hspace{0.05cm}  (\textbf{x}_{F}^{\perp}q_{k-1})\|_{F} 
\!= \!h_{F}\| \textrm{rot}_{F} \hspace{0.05cm} \textbf{s}_{k}\|_{F}.
 \end{split}
\]
Inserting this and~\eqref{equivalencenorm_proof3} in~\eqref{equivalencenorm_proof2},
recalling that~$\widetilde \gamma = (\gamma h_F) / \vert F \vert^{\frac12}$,
and using $\textrm{rot}_{F}\hspace{0.05cm}\textbf{s}_{k}=\textrm{rot}_{F}\hspace{0.05cm}(\textbf{x}^{\perp}q_{k-1})$, Corollary~\ref{cor31}, and~\eqref{thisdecom},
we arrive at
\begin{eqnarray*}
\begin{aligned} \label{equivalencenorm_proof2_1}
&\interleave \textbf{s}_{k}\interleave_{F}
   \geq \widetilde{\gamma}\left| \int_{\partial F}  \textbf{s}_{k} \cdot \textbf{t}_{\partial F}  \right | +\gamma  h_{F}  \| \textrm{rot}_{F}\hspace{0.05cm}\textbf{s}_{k} \|_{F}-  \frac{\gamma h_{F}}{|F|^{\frac12}}\left| \int_{\partial F}  \textbf{s}_{k} \cdot \textbf{t}_{\partial F}  \right | \\
&\hspace{0.4cm} + \hat{\gamma}h_{F}^{\frac12} \|\bm{\nabla}_{F} q^{0}_{k+1}\cdot \textbf{t}_{\partial F}\|_{\partial F} 
-C\hat{\gamma}h_{F}\| \textrm{rot}_{F} \hspace{0.05cm} \textbf{s}_{k}\|_{F}   
+ \!\!\!\!\!\!\!\!\sup_{p_{\beta_{F}}\in \mathbb{P}_{\beta_{F}}(F)}\!\!\!\!\!\!\!\!
\frac{h_{F}^{-1} \int_{F} \bm{\nabla}_{F} q^{0}_{k+1}\cdot \textbf{x}_{F} p_{\beta_{F}}    }{\|p_{\beta_{F}} \|_{F}}\\
&\geq (\gamma \!-\! C\hat{\gamma})h_{F} \| \textrm{rot}_{F} \hspace{0.05cm} \textbf{s}_{k}\|_{F}
   \!+\! \hat{\gamma} h_{F}^{\frac12} \|\bm{\nabla}_{F} q^{0}_{k+1} \!\cdot\! \textbf{t}_{\partial F}\|_{\partial F}
   +  \!\!\!\!\!\!\!\!\sup_{p_{\beta_{F}}\in \mathbb{P}_{\beta_{F}}(F)}\!\!\!\!\!\!\!\!
   \frac{h_{F}^{-1} \int_{F} \bm{\nabla}_{F} q^{0}_{k+1}\!\cdot\! \textbf{x}_{F} p_{\beta_{F}}    }{\|p_{\beta_{F}} \|_{F}} \\
&\geq  (\gamma \!-\! C\hat{\gamma})h_{F} \| \textrm{rot}_{F} \hspace{0.05cm} \textbf{s}_{k}\|_{F}
   \!+\! C  \|\bm{\nabla}_{F} q^{0}_{k+1}\|_{F}
\gtrsim \|\textbf{x}_{F}^{\perp}q_{k-1}\|_{F} + \|\bm{\nabla}_{F} q^{0}_{k+1}\|_{F}
\gtrsim \| \textbf{s}_{k}\|_{F},
 \end{aligned}
\end{eqnarray*}
where we have fixed the parameter~$\gamma=2C\hat{\gamma}$. 
The parameter $\hat\gamma$ was fixed in the proof of Corollary \ref{cor31}, sufficiently large but independent of~$F$.
Thus, \eqref{equivalencenorm} follows.

Before proceeding with the proof of the other two bounds,
we observe the validity of the following inverse estimate on the space $\textbf{SV}_{k}^{e}(F)$,
which can be proven as inequality~\eqref{inversediv2dedge}:
\begin{align} \label{inverserot}
\|\textrm{rot}_{F}\hspace{0.05cm} \textbf{v}_{h}\|_{F}
& \lesssim h^{-1}_{F}\| \textbf{v}_{h}\|_{F}\hspace{0.7cm} \forall \textbf{v}_{h}\in \textbf{SV}_{k}^{e}(F).
  \end{align}
Estimate~\eqref{equivalencenorm2} is proven using~\eqref{newnorm1}, \eqref{Polynomialinverseestimates1} recalling that~$\textbf{v}_{h} \cdot \textbf{t}_{\partial F}$ is a piecewise polynomial, \eqref{Traceinequality2}, and \eqref{inverserot}:
\[
\begin{split}
\interleave \textbf{v}_{\!h}\interleave_{\!F}
& \!\lesssim\! h_{\!F}^{\frac12}\|\textbf{v}_{\!h} \!\cdot\! \textbf{t}_{\partial F}\|_{\!\partial F}
\!+\! \|\textbf{v}_{\!h}\|_{F}
\!+\! h_{\!F}\|\textrm{rot}_{F}\hspace{0.05cm}\textbf{v}_{\!h}\|_{\!F}
\!\lesssim\! \|\textbf{v}_{\!h}\|_{\!F}
\!+\! h_{\!F} \|\textrm{rot}_{F}\hspace{0.05cm}\textbf{v}_{\!h}\|_{\!F}
\!\lesssim\! \|\textbf{v}_{\!h} \|_{\!F}.
\end{split}
\]
As for estimate~\eqref{equivalencenorm3},
from~\eqref{newnorm1}, \eqref{L1BOUND},
the inequality $\|\bm{\nabla}_{F} p_{k+1}\cdot \textbf{t}_{e}\|_{L^{\infty}(e)} \lesssim  h_{e}^{-\frac12} \|\bm{\nabla}_{F} p_{k+1}\cdot \textbf{t}_{e}\|_{e}$ for all~$e$ in~$\partial F$,
and the fact that the number of edges on each face $F$ is uniformly bounded, it follows that
\[
\begin{split}
\interleave \textbf{v}\interleave_{F} 
& \!\lesssim\! \sum_{e\subseteq \partial F}\left| \int_{e}  \textbf{v}  \!\cdot\! \textbf{t}_{e}  \right | +h_{F} \|\textrm{rot}_{F}\hspace{0.05cm}\textbf{v}   \|_{F} 
\!+\! \|\textbf{v} \|_{F}
+ \!\!\!\!\!\!\!\!\!\!\sup_{p_{k+1}\in \mathbb{P}_{k+1}(F)} \!\!\!\!\!\!\!\!\!\!
\frac{h_{F}^{\frac12} \!\!\!\! \sum\limits_{e\subseteq \partial F}\!\!\!\!
\int_{e}(\textbf{v}  \cdot \textbf{t}_{e})(\bm{\nabla}_{F} p_{k+1}\!\cdot\! \textbf{t}_{\partial F}) }{\|\bm{\nabla}_{F} p_{k+1}\!\cdot\! \textbf{t}_{\partial F}\|_{\partial F}} \\
& \lesssim  \|\textbf{v} \|_{F} +h^{\varepsilon}_{F}  | \textbf{v}    |_{\varepsilon,F} 
+h_{F} \|\textrm{rot}_{F}\hspace{0.05cm}\textbf{v}   \|_{F}. 
\end{split}
\]
\end{proof}

The following result contains a useful estimate for the projection $\mathbf{\Pi}_{S}^{e}$.
\begin{thm}
For each $\textbf{v}\in \textbf {H}^{s}(F)$, $0 < s\leq k+1$ with ${\rm rot}_{F}\hspace{0.05cm}\textbf{v}\in  L^{2}(F)$,
 we have
\begin{align}
 \label{projectionerror1}
 &\|\textbf{v}-\mathbf{\Pi}_{S}^{e} \textbf{v}\|_{F} 
 \lesssim  h_{F}^{s} |\textbf{v}|_{s,F}
 + h_{F} \| {\rm rot}_{F}\hspace{0.05cm}\textbf{v} \|_{F}.
\end{align}
The second term on the right-hand side can be neglected if $s \ge 1$.
\end{thm}
\begin{proof}
For any $\textbf{p}_{k}\in (\mathbb{P}_{k}(F))^{2}$,
from~\eqref{equivalencenorm},
the fact that~$\interleave  \mathbf{\Pi}_{S}^{e} \cdot \interleave_{F}$ is equal to~$\interleave \cdot \interleave_{F}$, and finally \eqref{equivalencenorm3}, we obtain
\begin{eqnarray}
\begin{aligned} \label{proj_est_1}
\|\textbf{v}-\mathbf{\Pi}_{S}^{e} \textbf{v}\|_{F}
&\!\leq\! \|\textbf{v} \!-\! \textbf{p}_{k}\|_{F}
\!+\! \|\mathbf{\Pi}_{S}^{e}( \textbf{v} \!-\! \textbf{p}_{k})\|_{F}
\!\lesssim\! \|\textbf{v} \!-\! \textbf{p}_{k}\|_{F}
\!+\! \interleave\mathbf{\Pi}_{S}^{e}( \textbf{v} \!-\! \textbf{p}_{k})\interleave_{F} \\
& = \|\textbf{v} \!-\! \textbf{p}_{k}\|_{F}
\!+\! \interleave  \textbf{v} \!-\! \textbf{p}_{k} \interleave_{F} \\
&\lesssim \|  \textbf{v}-\textbf{p}_{k} \|_{F}+h^{s}_{F}|  \textbf{v}-\textbf{p}_{k} |_{s,F}+ h_{F} \|\textrm{rot}_{F}\hspace{0.05cm} (  \textbf{v}-\textbf{p}_{k}) \|_{F}.
\end{aligned}
\end{eqnarray}
If $s\geq 1$, then \eqref{proj_est_1} and standard polynomial approximation estimates yield
\[
\|\textbf{v} \!-\! \mathbf{\Pi}_{S}^{e} \textbf{v}\|_{F}
\!\lesssim\! h_{F}^{s} |\textbf{v}|_{s,F}.
\]
Instead, if $0 <s< 1$, then we replace~$\textbf{p}_{k}$ by the average vector constant~$\textbf{p}_{0}$ of~$\textbf{v}$
over~$F$ in~\eqref{proj_est_1}.
The Poincar\'e inequality gives
\[
\begin{split}
 \|\textbf{v}-\mathbf{\Pi}_{S}^{e} \textbf{v}\|_{F}
& \lesssim \|  \textbf{v}-\textbf{p}_{0} \|_{F}+h^{s}_{F}|  \textbf{v}-\textbf{p}_{0} |_{s,F}+ h_{F} \|\textrm{rot}_{F}\hspace{0.05cm} (  \textbf{v}-\textbf{p}_{0}) \|_{F}\\
& \lesssim h_{F}^{s}|\textbf{v}|_{s,F} + h_{F}\|\textrm{rot}_{F}\hspace{0.05cm}\mathbf{v}\|_{F}.
\end{split}
\]
\end{proof}

We define an interpolation operator ${\textbf{I}}_{h}^{e}$ for functions in~$\textbf{SV}_{k}^{e}(F)$
by requiring that the values of the DoFs \eqref{dof1}, \eqref{dof3}, and~\eqref{DoFsinternal} of~$\textbf{I}^{e}_{h} \textbf{v}$ are equal to those of~$\textbf{v}$.
Combining \eqref{dof1}  with \eqref{dof3}, we obtain the  following  property:
\begin{align} \label{commuproperty}
\textrm{rot}_{F}\hspace{0.05cm} (\textbf{I}^{e}_{h} \textbf{v})= \Pi_{k-1}^{0,F}(\textrm{rot}_{F} \hspace{0.05cm} \textbf{v}).
\end{align}
We prove the following interpolation estimates for $\textbf{I}^{e}_{h}$ on the serendipity edge virtual element space  $\textbf{SV}_{k}^{e}(F)$.
\begin{thm}
\label{theorem31}
For each $\textbf{v}\in \textbf {H}^{s}(F) $, $0 < s\leq k+1$, with ${\rm rot}_{F}\hspace{0.05cm}\textbf{v}\in  H^{r}(F)$,  $0\leq r\leq k$, we have
\begin{align} \label{interpolation1}
&\|\textbf{v}-{\textbf{I}}_{h}^{e}\textbf{v}\|_{F}
\lesssim h_{F}^{s} |\textbf{v}|_{s,F}
 + h_{F} \| {\rm rot}_{F}\hspace{0.05cm}\textbf{v}\|_{F},\\
 \label{interpolation2}
& \|{\rm rot}_{F}\hspace{0.05cm}(\textbf{v}-{\textbf{I}}_{h}^{e}\textbf{v})\|_{F}
\lesssim h_{F}^{r}|{\rm rot}_{F}\hspace{0.05cm}\textbf{v}|_{r,F}.
\end{align}
The second term on the right-hand side of \eqref{interpolation1} can be neglected if $s \ge 1$.
\end{thm}
\begin{proof}
As for \eqref{interpolation2}, by \eqref{commuproperty} and standard polynomial approximation properties, we have
\[
\|\textrm{rot}_{F}\hspace{0.05cm}(\textbf{v}-{\textbf{I}}_{h}^{e}\textbf{v})\|_{F}
= \|\textrm{rot}_{F}\hspace{0.05cm}\textbf{v}- {\Pi}_{k-1}^{0,F}(\textrm{rot}_{F} \hspace{0.05cm} \textbf{v})\|_{F}
\lesssim h_{F}^{r}|\textrm{rot}_{F}\hspace{0.05cm}\textbf{v}|_{r,F}.
\]
The remainder of the proof is devoted to proving~\eqref{interpolation1}.
Observe that~\eqref{standardinterpolation22dedge_proof1} and~\eqref{commuproperty} imply
$\textrm{rot}_F (\widetilde{\textbf{I}}^{e}_{h}\textbf{v}-{\textbf{I}}^{e}_{h}\textbf{v})=0$,
which yields the existence of a function $\phi\in H^{1}(F)$ such that $\widetilde{\textbf{I}}^{e}_{h}\textbf{v}-{\textbf{I}}^{e}_{h}\textbf{v} = \nabla_F \phi$, satisfying weakly
\begin{align} \label{standardinterpolation1_proof3}
\Delta_{F} \phi= \textrm{div}_{F}\hspace{0.05cm} (\widetilde{\textbf{I}}^{e}_{h}\textbf{v} - {\textbf{I}}^{e}_{h}\textbf{v}) \ \textrm{in}\ F,
\ \  \phi=0\ \textrm{on}\ \partial F.
\end{align}
The boundary conditions in~\eqref{standardinterpolation1_proof3} follow from the fact that
\[
\partial _t \phi{}_{|\partial F} = 
(\widetilde{\textbf{I}}^{e}_{h}\textbf{v}-{\textbf{I}}^{e}_{h}\textbf{v}) \cdot \mathbf t{}_{|\partial F} = 0,
\]
since the definitions of ${\textbf{I}}^{e}_{h}$ and $\widetilde{\textbf{I}}^{e}_{h}$ entail
\[
 (\widetilde{\textbf{I}}^{e}_{h}\textbf{v}-{\textbf{I}}^{e}_{h}\textbf{v})\cdot {\textbf{t}}_{e} =0 \hspace{0.2cm} \forall e\subseteq \partial F.
\]
Since
\begin{align}
 \label{standardinterpolation1_proof7} \|\widetilde{\textbf{I}}_{h}^{e}\textbf{v}-{\textbf{I}}^{e}_{h}\textbf{v}\|_{F}
 = 
 \|\bm{\nabla}_{F} \phi\|_{F},
\end{align}
it suffices to estimate the right-hand side of \eqref{standardinterpolation1_proof7}.
By the fact that $\textrm{div}_{F}\hspace{0.05cm} (\widetilde{\textbf{I}}^{e}_{h}\textbf{v}-{\textbf{I}}^{e}_{h}\textbf{v})\in \mathbb{P}_{k}(F)$ and \eqref{decompo1}, there exists a polynomial $q_{k}\in \mathbb{P}_{k}(F)$ such that
  \begin{align}
   \label{standardinterpolation2_proof1}
\textrm{div}_{F}\hspace{0.05cm}( \textbf{x}_{F} q_{k})=\textrm{div}_{F}\hspace{0.05cm} (\widetilde{\textbf{I}}_{h}^{e} \textbf{v}-{\textbf{I}}^{e}_{h}\textbf{v}),
  \end{align}
with
\begin{align} \label{standardinterpolation2_proof2}
\|\textbf{x}_{F}q_{k}\|_{F}
\overset{\eqref{inversediv2dedge}}{\lesssim}
h_{F}\|\textrm{div}_{F}\hspace{0.05cm} (\widetilde{\textbf{I}}_{h}^{e} \textbf{v}- \textbf{I}^{e}_{h}\textbf{v})\|_{F}
\lesssim \| \widetilde{\textbf{I}}_{h}^{e} \textbf{v}-{\textbf{I}}^{e}_{h}\textbf{v} \|_{F} .
  \end{align}
Moreover, $\mathbf{\Pi}_{S}^{e}{\textbf{I}}_{h}^{e}\textbf{v}=\mathbf{\Pi}_{S}^{e}{\textbf{v}}$
since $\textbf{I}_{h}^{e}\textbf{v}$ and $\textbf{v}$ share the same DoFs~\eqref{dof1}, \eqref{dof3}, \eqref{DoFsinternal},
and the value of the projection $\mathbf{\Pi}_{S}^{e}$ only depends on such DoFs. Thus, we write
\begin{eqnarray}
\begin{aligned}[t] \label{standardinterpolation2_proof10}
& \|\bm{\nabla}_{F}\phi\|^{2}_{F} =
(\widetilde{\textbf{I}}^{e}_{h}\textbf{v}-{\textbf{I}}^{e}_{h}\textbf{v},\bm{\nabla}_{F}\phi)_{F}
\overset{\text{IBP}}{=}
-(\textrm{div}_{F}\hspace{0.05cm} (\widetilde{\textbf{I}}_{h}^{e} \textbf{v}-{\textbf{I}}^{e}_{h}\textbf{v}),\phi)_{F}\\
& \overset{\eqref{standardinterpolation2_proof1}}{=}
-(\textrm{div}_{F}\hspace{0.05cm} ( \textbf{x}_{F}q_{k}),\phi)_{F}
\overset{\eqref{standardinterpolation1_proof3}}{=}
(\textbf{x}_{F}q_{k},\bm{\nabla}_{F}\phi)_{F}
= (\textbf{x}_{F}q_{k}, \widetilde{\textbf{I}}^{e}_{h}\textbf{v}-{\textbf{I}}^{e}_{h}\textbf{v})_{F} \\
&  \!\!\overset{\eqref{eq:standard_interpolation2}, \eqref{svspace}}{=}\!
(\textbf{x}_{F}q_{k}, \textbf{v}-\mathbf{\Pi}_{S}^{e}{\textbf{I}}^{e}_{h}\textbf{v})_{F}
\!\overset{\eqref{eq:standard_projection4}}{=}\!
(\textbf{x}_{F}q_{k}, \textbf{v}-\mathbf{\Pi}_{S}^{e}\textbf{v})_{F}
\!\lesssim\! \|\textbf{x}_{F}q_{k}\|_{F}\|\textbf{v}-\mathbf{\Pi}_{S}^{e}\textbf{v} \|_{F}\\
& \overset{\eqref{standardinterpolation2_proof2}}{\lesssim}
\|\widetilde{\textbf{I}}_{h}^{e}\textbf{v}-{\textbf{I}}^{e}_{h}\textbf{v}\|_{F}\|\textbf{v}-\mathbf{\Pi}_{S}^{e}\textbf{v}\|_{F}
\overset{\eqref{projectionerror1}}{\lesssim}
\left( h_{F}^{s} |\textbf{v}|_{s,F} + h_F \|\textrm{rot}_{F}\hspace{0.05cm}\textbf{v}\|_{F}\right)
\|\widetilde{\textbf{I}}^{e}_{h}\textbf{v}-{\textbf{I}}^{e}_{h}\textbf{v}\|_{F},
\end{aligned}
\end{eqnarray}
where the term~$\Vert \textrm{rot}_{F}\hspace{0.05cm}\textbf{v} \Vert_{F}$ can be ignored if $s\geq 1$.

Substituting~\eqref{standardinterpolation2_proof10} into \eqref{standardinterpolation1_proof7},
and by using the triangle inequality and~\eqref{standardinterpolation12dedge}, estimate~\eqref{interpolation1} follows.
\end{proof}

 \subsection{Face virtual element spaces on polygons} \label{sec4faceVEM}
Since 2D face virtual element spaces can be viewed as a $\pi/2$ rotation of the 2D edge ones,
we can extend all above definitions and results to standard and serendipity face virtual element spaces in 2D;
see Refs.~\cite{lourencobrezzidassi2017cmame,beiraobrezzidasimaru2018siam,beiraobrezzimariniru2015,bebremaru2016RLMASerendipity}.
The face virtual element space on the face~$F$ is defined as
\[
\textbf{V}_{k}^{f}(F)
\!=\! 
\big\{\! \textbf{v}_{h} \!\in\! \textbf{L}^{2}(F) \!:\! \textrm{div}_{F}\hspace{0.05cm} \textbf{v}_{h} \!\in\! \mathbb{P}_{k-1}(F), \  
\textrm{rot}_{F} \hspace{0.05cm}  \textbf{v}_{h} \!\in\!  \mathbb{P}_{k}(F), \
\textbf{v}_{h} \!\cdot\! \textbf{n}_{e} \!\in\! \mathbb{P}_{k}(e) \ \ \forall e \!\subseteq\! \partial F \!\big\},
\]
and is endowed  with  the unisolvent DoFs\cite{lourencobrezzidassi2017cmame,beiraobrezzidasimaru2018siam}
\begin{align}
\label{dof1face}
&\bullet  \  \int_{e} \textbf{v}_{h}\cdot \textbf{n}_{e} p_{k}&&\forall p_{k}\in \mathbb{P}_{k}(e),  \hspace{0.1cm} \forall e\subseteq \partial F;\\
\label{dof2face}
&\bullet\  \int_{F}  \textbf{v}_{h}\cdot \textbf{x}_{F}^{\perp}p_{k}    &&\forall p_{k} \in \mathbb{P}_{k}(F);\\
\label{dof3face}
&\bullet \  \int_{F} \textrm{div}_{F}
\hspace{0.05cm}\textbf{v}_{h}   p^{0}_{k-1}  &&\forall p^{0}_{k-1}\in \mathbb{P}_{k-1}^{0}(F)   \hspace{0.2cm} \textrm{only\ for}\  k>1.
   \end{align}
We define the DoFs interpolation operator $\widetilde{\textbf{I}}_{h}^{f}$ on the space $\textbf{V}_{k}^{f}(F)$
by requiring that the values of the DoFs~\eqref{dof1face}, \eqref{dof2face}, and~\eqref{dof3face} of $\widetilde{\textbf{I}}^{f}_{h} \textbf{v}$
are equal to those of $\textbf{v}\in \textbf{H}^{s}(F) \cap \textbf{H}(\textrm{div}_{F},F)$,   $s>0$.
We  can easily extend the interpolation estimates of edge virtual element spaces, to the  face case; see Theorem \ref{thm31edge2d}.
\begin{thm} \label{theorem41original}
For each $\textbf{v}\in \textbf {H}^{s}(F) $,  $0 <s\leq k+1$ with ${\rm div}_{F}\hspace{0.05cm}\textbf{v}\in  H^{r}(F)$,  $0\leq r\leq k$, we have
\begin{align} \label{interpolation1div2dfaceor}
& \|\textbf{v}-\widetilde{\textbf{I}}_{h}^{f}\textbf{v}\|_{F}
\lesssim h_{F}^{s} |\textbf{v}|_{s,F}
 + h_F \|{\rm div}_{F}\hspace{0.05cm}\textbf{v} \|_{F},\\
 \label{interpolation2div}
& \|{\rm div}_{F}\hspace{0.05cm}(\textbf{v}-\widetilde{\textbf{I}}_{h}^{f}\textbf{v})\|_{F}
\lesssim h_{F}^{r}|{\rm div}_{F}\hspace{0.05cm}\textbf{v}|_{r,F}.
\end{align}
The second term on the right-hand side of \eqref{interpolation1div2dfaceor} can be neglected if $s \ge 1$.
\end{thm}

By rotating everything by $\pi/2$ corresponding to edge elements,   we  can also introduce a well defined  projection $\mathbf{\Pi}_{S}^{f}:  \textbf{V}_{k}^{f}(F)\rightarrow  (\mathbb{P}_{k}(F))^{2}$  by
\[
\begin{split}
&  \int_{\partial F}[(\textbf{v}_{h}-\mathbf{\Pi}_{S}^{f} \textbf{v}_{h})\cdot \textbf{n}_{\partial F}][\textbf{curl}_{F} \hspace{0.05cm} p_{k+1}\cdot \textbf{n}_{\partial F}]=0 \hspace{0.5cm}\forall p_{k+1}\in \mathbb{P}_{k+1}(F);\\
&\int_{\partial F} (\textbf{v}_{h}- \mathbf{\Pi}_{S}^{f} \textbf{v}_{h})\cdot \textbf{n}_{\partial F}  =0;\\
&\int_{F} \textrm{div}_{F}\hspace{0.05cm}(\textbf{v}_{h}-  \mathbf{\Pi}_{S}^{f} \textbf{v}_{h}) p_{k-1}^{0}   =0\hspace{0.5cm} \forall p_{k-1}^{0}\in \mathbb{P}_{k-1}^{0}(F)   \hspace{0.2cm} \textrm{only\ for}\  k>1;\\
&\int_{F}(\textbf{v}_{h}-\mathbf{\Pi}_{S}^{f} \textbf{v}_{h})\cdot\textbf{x}^{\perp}_{F} p_{\beta_{F}}=0\hspace{0.75cm}\forall p_{\beta_{F}} \in \mathbb{P}_{\beta_{F}}(F)\ \textrm{only\ for}\ \beta_{F}\geq 0.
\end{split}
\]
Eventually,  we  introduce the serendipity face virtual element space  on the face $F$
\[
 \textbf{SV}_{k}^{f}(F)=\left\{\textbf{v}_{h}\in \textbf{V}_{k}^{f}(F): \int_{F}(\textbf{v}_{h}-\mathbf{\Pi}_{S}^{f} \textbf{v}_{h})\cdot\textbf{x}^{\perp}_{F} p=0\hspace{0.2cm} \forall p \in \mathbb{P}_{\beta_{F}|k}(F)\right\},
\]
which is endowed with the set of unisolvent DoFs \eqref{dof1face} and \eqref{dof3face}, plus the moments
\begin{align}
\label{DoFsinternalf2dface}
\int_{F} \textbf{v}_{h}\cdot \textbf{x}^{\perp}_{F} p_{\beta_{F}}    \hspace{0.3cm} \forall p_{\beta_{F}}\in \mathbb{P}_{\beta_{F}}(F)   \hspace{0.2cm} \textrm{only\ for}\ \beta_{F}\geq 0.
\end{align}
We define the DoFs interpolation operator ${\textbf{I}}_{h}^{f}$ on the serendipity face virtual element space~$\textbf{SV}_{k}^{f}(F)$
by requiring that the values of the DoFs~\eqref{dof1face}, \eqref{dof3face}, and~\eqref{DoFsinternalf2dface}
of~$\textbf{I}^{f}_{h} \textbf{v}$ are equal to those of~$\textbf{v}$.
We inherit interpolation estimates from serendipity edge spaces.
In fact, the following result is proven as the rotated version of Theorem \ref{theorem31} (and thus also needs the additional mesh assumption (\textbf{MC})).
\begin{thm}
\label{theorem41}
For each $\textbf{v}\in \textbf {H}^{s}(F) $,  $0 < s\leq k+1$, with ${\rm div}_{F}\hspace{0.05cm}\textbf{v}\in  H^{r}(F)$, $0\leq r\leq k$, we have
\begin{align} \label{interpolation1div}
&\|\textbf{v}-{\textbf{I}}_{h}^{f}\textbf{v}\|_{F}
\lesssim h_{F}^{s} |\textbf{v}|_{s,F}
 + h_F \|{\rm div}_{F}\hspace{0.05cm}\textbf{v}\|_{F},\\
 \label{interpolation2ddiv}
&\|{\rm div}_{F}\hspace{0.05cm}(\textbf{v}-{\textbf{I}}_{h}^{f}\textbf{v})\|_{F}
\lesssim h_{F}^{r}|{\rm div}_{F}\hspace{0.05cm}\textbf{v}|_{r,F}.
\end{align}
The second term on the right-hand side of \eqref{interpolation1div} can be neglected if $s \ge 1$.
\end{thm}

\section{Interpolation properties of  edge and face virtual element spaces  in 3D} \label{sec4}
In this section,
we prove interpolation properties for general order face and edge virtual element spaces on polyhedra.
More precisely we consider 
standard face virtual element spaces in Section~\ref{sec41};
standard edge virtual element spaces in Section~\ref{sec42};
serendipity edge virtual element space in Section~\ref{sec43}.

\subsection{Standard face virtual element space on polyhedrons}
\label{sec41}
We consider the face virtual element space\cite{beiraobrezzidasimaru2018siam}
\[
\begin{split}
\textbf{V}_{k-1}^{f}(E)=\big\{\textbf{v}_{h}\in \textbf{L}^{2}(E): \textrm{div} \hspace{0.05cm} \textbf{v}_{h}\in \mathbb{P}_{k-1}&(E),   \textbf{curl}  \hspace{0.05cm}  \textbf{v}_{h} \in ( \mathbb{P}_{k}(E))^{3}, \\ &\textbf{v}_{h} \cdot \textbf{n}_{F}\in  \mathbb{P}_{k-1}(F) \   \forall F\subseteq \partial E\big\},
\end{split}
\]
and endow it with the unisolvent set of DoFs\cite{beiraobrezzidasimaru2018siam,beiraobrezzimariniru2015}
\begin{align*}
&\bullet  \  \int_{F}\textbf{v}_{h}\cdot \textbf{n}_{F}p_{k-1}&&\forall p_{k-1}\in \mathbb{P}_{k-1}(F),  \hspace{0.1cm} \forall F\subseteq \partial E;\\
&\bullet\  \int_{E}  \textbf{v}_{h}\cdot \textbf{x}_{E} \wedge \textbf{p}_{k}    &&\forall \textbf{p}_{k} \in (\mathbb{P}_{k}(E))^{3};\\
&\bullet \  \int_{E} \textrm{div}
\hspace{0.05cm}\textbf{v}_{h}   p^{0}_{k-1}  &&\forall p^{0}_{k-1}\in \mathbb{P}_{k-1}^{0}(E)  \hspace{0.2cm} \textrm{only\ for}\  k>1.
\end{align*}
A simple computation reveals that the $\textbf{L}^{2}$ projection $\mathbf{\Pi}_{k+1}^{0,E}$:$\textbf{V}_{k-1}^{f}(E)\rightarrow (\mathbb{P}_{k+1}(E))^{3}$ is computable by means of such DoFs.

We first prove the following auxiliary bound for functions in~$\textbf{V}_{k-1}^{f}(E)$.
\begin{lem}
\label{lem3dfacebound}
For each $\textbf{v}_{h}\in \textbf{V}_{k-1}^{f}(E)$, we have
\begin{align} \label{3dfacePRIORIB}
\|\textbf{v}_{h}\|_{E} 
\lesssim h_{E} \|{\rm div}\hspace{0.05cm}\textbf{v}_{h} \|_{E}+h_{E}^{\frac12}\|\textbf{v}_{h}\cdot \textbf{n}_{\partial E}\|_{\partial E} 
+\!\!\!\!\!\!\!\sup_{\textbf{p}_{k}\in (\mathbb{P}_{k}(E))^3}\!\!\!\!\!\!
\frac{ \int_{E} \textbf{v}_{h}\cdot \textbf{x}_{E}\wedge \textbf{p}_{k}}{\|\textbf{x}_{E}\wedge \textbf{p}_{k}\|_{E}}.
\end{align}
\end{lem}
\begin{proof}
The following Helmholtz decomposition of~$\textbf{v}_{h}$ is valid;
see Proposition 3.1 in Ref.~\cite{beiraolourenco2020}:
\begin{align}
 \label{standardinterpolation1_proof63dface}  \textbf{v}_{h}=\textbf{curl} \hspace{0.05cm}\bm{\rho}+\bm{\nabla}  \psi,
\end{align}
where
the function $\psi\in H^{1}(E)\setminus\mathbb{R}$ satisfies weakly
\[
\Delta \psi= \textrm{div} \hspace{0.05cm}\textbf{v}_{h} \ \textrm{in}\ E,\ \ \bm{\nabla} \psi\cdot \textbf{n}_{\partial E}
= \textbf{v}_{h} \cdot \textbf{n}_{\partial E}\ \textrm{on}\ \partial E,
\]
and the function $\bm{\rho}\in   \textbf{H}(\textbf{curl}, E) \cap  \textbf{H}(\textrm{div,E})$  satisfies weakly
 \begin{align}
 \label{standardinterpolation1_proof33dface}
\textbf{curl} \hspace{0.05cm}\textbf{curl} \hspace{0.05cm}\bm{\rho} = \textbf{curl} \hspace{0.05cm} \textbf{v}_{h} \ \textrm{in}\ E,\ \  \textrm{div} \hspace{0.05cm}\bm{\rho}=0\ \textrm{in}\  E,\ \ \bm{\rho}\wedge \textbf{n}_{\partial E} =0\ \textrm{on}\ \partial E.
\end{align}
We have
\begin{align} \label{standardinterpolation1_proof73dface}
(\textbf{curl} \bm{\rho} , \bm{\nabla}\psi)_{E}=0,
\qquad
\|\textbf{v}_{h} \|^{2}_{E}
=\|\textbf{curl} \hspace{0.05cm}\bm{\rho}\|^{2}_{E}
  +\|\bm{\nabla}  \psi\|^{2}_{E}.
\end{align}
By using \eqref{standardinterpolation1_proof63dface}, an integration by parts,   \eqref{Traceinequality4} and \eqref{poincarefridine}, it is immediate  that
\begin{eqnarray}
\begin{aligned}[b] \label{standardinterpolation1_proof93dface}
& \|\bm{\nabla}  \psi\|^{2}_{E} 
\overset{\eqref{standardinterpolation1_proof63dface}, \eqref{standardinterpolation1_proof73dface}}{=}
(\bm{\nabla}  \psi,\textbf{v}_{h} )_{E} 
\overset{\text{IBP}}{=}
\int_{\partial E} \textbf{v}_{h}\cdot\textbf{n}_{\partial E} \psi - \int_{  E} \textrm{div}\hspace{0.05cm} \textbf{v}_{h} \psi
 \!\lesssim\! \|\textbf{v}_{h} \!\cdot\! \textbf{n}_{\partial E}\|_{\partial E}\|\psi \|_{\partial E} \\
& \quad \!+\! \|  \textrm{div}\hspace{0.05cm} \textbf{v}_{h}\|_{E}\|  \psi\|_{E}
\!\overset{\eqref{Traceinequality4},\eqref{poincarefridine}}{\lesssim}\!  
\!\left(h_{E} \|  \textrm{div}\hspace{0.05cm} \textbf{v}_{h}\|_{E}
\!+\! h_{E}^{\frac12}\|\textbf{v}_{h} \!\cdot\! \textbf{n}_{\partial E}\|_{\partial E} \!\right)\! \|\bm{\nabla}  \psi\|_{E}.
\end{aligned}
\end{eqnarray}
Since $\textbf{curl}\hspace{0.05cm} \textbf{v}_{h}\in (\mathbb{P}_{k}(E))^3$
with $\textrm{div}\hspace{0.05cm}  (\textbf{curl}\hspace{0.05cm}  \textbf{v}_{h})=0$,
\eqref{decompo23d} implies the existence of~$\textbf{q}_{k}\in (\mathbb{P}_{k}(E))^3$ such that
\begin{align} \label{standardinterpolation2_proof13dface}
\textbf{curl} \hspace{0.05cm}( \textbf{x}_{E}\wedge \textbf{q}_{k})
=\textbf{curl} \hspace{0.05cm}  \textbf{v}_{h}
\ \textrm{and}\
\| \textbf{x}_{E}\wedge \textbf{q}_{k}\|_{E}
\lesssim h_{E}\|\textbf{curl} \hspace{0.05cm} \textbf{v}_{h} \|_{E}.
  \end{align}
The following inverse estimate inequality involving face virtual element functions is the three dimensional version of~\eqref{inverserot}
and is based on the existence of a shape-regular decomposition of~$E$ into tetrahedra
(see Remark~\ref{rem21}):
\begin{align} \label{3dcurlinverseesti}
\|\textbf{curl} \hspace{0.05cm}  \textbf{v}_{h} \|_{E} 
\lesssim h_{E}^{-1} \|\textbf{v}_{h} \|_{E}
\hspace{0.3cm} \forall \textbf{v}_{h} \in \textbf{V}_{k-1}^{f}(E).
\end{align}
Next, we  estimate the first term on the right-hand side of \eqref{standardinterpolation1_proof73dface}:
\begin{eqnarray}
\begin{aligned} \label{standardinterpolation2_proof103dface}
\|\textbf{curl}\hspace{0.05cm} \bm{\rho}\|^{2}_{E}
& \overset{\text{IBP}}{=}
\int_{E}  \bm{\rho}  \cdot \textbf{curl}\hspace{0.05cm} \textbf{curl}\hspace{0.05cm} \bm{\rho}
\overset{\eqref{standardinterpolation1_proof33dface}}{=}
\int_{E}  \bm{\rho}  \cdot   \textbf{curl}\hspace{0.05cm} \textbf{v}_{h}
\overset{\eqref{standardinterpolation2_proof13dface}}{=}
\int_{E}  \bm{\rho}  \cdot   \textbf{curl} \hspace{0.05cm}( \textbf{x}_{E}\wedge \textbf{q}_{k})\\
& \overset{\text{IBP},\eqref{standardinterpolation1_proof63dface}}{=}
\int_{E}   (\textbf{v}_{h}-\bm{\nabla}  \psi)   \cdot ( \textbf{x}_{E}\wedge \textbf{q}_{k})
+ \int_{\partial E} (\bm{\rho}\wedge \textbf{n}_{\partial E} )\cdot  ( \textbf{x}_{E}\wedge \textbf{q}_{k}) \\
& \overset{\eqref{standardinterpolation1_proof33dface}}{\leq}
\Big(\sup_{\textbf{p}_{k}\in (\mathbb{P}_{k}(E))^3}\frac{ \int_{E} \textbf{v}_{h}\cdot \textbf{x}_{E}\wedge \textbf{p}_{k}}{\|\textbf{x}_{E}\wedge \textbf{p}_{k}\|_{E}}
+ \|\bm{\nabla}  \psi\|_{E} \Big)\|\textbf{x}_{E}\wedge \textbf{q}_{k}\|_{E}\\
& \overset{\eqref{standardinterpolation2_proof13dface}, \eqref{3dcurlinverseesti}}{\lesssim}
\Big( \sup_{\textbf{p}_{k}\in (\mathbb{P}_{k}(E))^3}\frac{ \int_{E} \textbf{v}_{h}\cdot \textbf{x}_{E}\wedge \textbf{p}_{k}}{\|\textbf{x}_{E}\wedge \textbf{p}_{k}\|_{E}}
+ \|\bm{\nabla}  \psi\|_{E}\Big)\| \textbf{v}_{h} \|_{E}.
\end{aligned}
\end{eqnarray}
Bound~\eqref{3dfacePRIORIB} easily follows by combining~\eqref{standardinterpolation1_proof73dface},
\eqref{standardinterpolation1_proof93dface}, and~\eqref{standardinterpolation2_proof103dface}.
\end{proof}

The DoFs interpolation operator $\widetilde{\textbf{I}}_{h}^{f}$ on the space~$\textbf{V}_{k-1}^{f}(E)$ is well defined for functions
in $\textbf{H}^{s}(E)\cap \textbf{H}(\textrm{div}, E)$, $s>1 \slash 2$:
\begin{subequations} \label{dofface3Dinter}
\begin{align} \label{dof1face3Dinter}
&\int_{F}( \textbf{v} -\widetilde{\textbf{I}}_{h}^{f} \textbf{v})\cdot \textbf{n}_{F} p_{k-1}=0&&\forall p_{k-1}\in \mathbb{P}_{k-1}(F),  \hspace{0.1cm} \forall F\subseteq \partial E;\\
\label{dof2face3Dinter}
&  \int_{E}  (\textbf{v} - \widetilde{\textbf{I}}_{h}^{f} \textbf{v})\cdot \textbf{x}_{E} \wedge \textbf{p}_{k}   =0 &&\forall \textbf{p}_{k} \in (\mathbb{P}_{k}(E))^{3};\\
\label{dof3face3Dinter}
& \int_{E} \textrm{div}
\hspace{0.05cm}(\textbf{v} -\widetilde{\textbf{I}}_{h}^{f} \textbf{v})  p^{0}_{k-1}  =0&&\forall p^{0}_{k-1}\in \mathbb{P}_{k-1}^{0}(E)  \hspace{0.2cm} \textrm{only\ for}\  k>1.
\end{align}
\end{subequations}
From~\eqref{dof1face3Dinter} and~\eqref{dof3face3Dinter}, we have
\begin{align} \label{3dfacecommuni}
\textrm{div}
\hspace{0.05cm} (\widetilde{\textbf{I}}_{h}^{f} \textbf{v})={\Pi}_{k-1}^{0,E}(\textrm{div}
\hspace{0.05cm} \textbf{v} ).
\end{align}
Next, we prove interpolation estimates for the three-dimensional face virtual element space $\textbf{V}_{k-1}^{f}(E)$.
\begin{thm} \label{theorem3dfaceinter}
For each $\textbf{v}\in \textbf {H}^{s}(E) $, $1\slash2 < s\leq {k}$,
with ${\rm div}\hspace{0.05cm}\textbf{v}\in  H^{r}(E)$, $0\leq r\leq k$,we have
\begin{align} \label{interpolation1div3D}
&\|\textbf{v}- \widetilde{\textbf{I}}_{h}^{f} \textbf{v}\|_{E}
\lesssim h_{E}^{s} |\textbf{v}|_{s,E}
         + h_E \|{\rm div} \hspace{0.05cm}\textbf{v}\|_{E} ,\\
 \label{interpolation2div3D}
&\|{\rm div} \hspace{0.05cm}(\textbf{v} - \widetilde{\textbf{I}}_{h}^{f} \textbf{v})\|_{E}
\lesssim h_{E}^{r}|{\rm div} \hspace{0.05cm}\textbf{v}|_{r,E}.
\end{align}
The second term on the right-hand side  of \eqref{interpolation1div3D} can be neglected if $s \ge 1$.
\end{thm}
  \begin{proof}
By~\eqref{3dfacecommuni} and standard polynomial approximation properties,
we immediately get~\eqref{interpolation2div3D}.
Hence, we focus on bound~\eqref{interpolation1div3D}.

First, we observe that~\eqref{dof1face3Dinter} implies
\begin{equation}\label{cavallino}
h_{E}^{\frac12}\|(\mathbf{\Pi}_{k-1}^{0,E} \mathbf{v}- \widetilde{\textbf{I}}^{f}_{h}\textbf{v})\cdot \textbf{n}_{\partial E}\|_{\partial E}
 \le
h_{E}^{\frac12}\|(\mathbf{\Pi}_{k-1}^{0,E} \mathbf{v}- \textbf{v})\cdot \textbf{n}_{\partial E}\|_{\partial E}.
\end{equation}
Using the facts that $\mathbf{\Pi}_{k-1}^{0,E} \mathbf{v}\in(\mathbb{P}_{k-1}(E))^{3}\subseteq  \textbf{V}_{k-1}^{f}(E)$
and $\widetilde{\textbf{I}}^{f}_{h}\textbf{v} \cdot \textbf{n}_{F}\in \mathbb{P}_{k-1}(F)$, \eqref{3dfacePRIORIB}, and~\eqref{dofface3Dinter},
it follows that
\begin{eqnarray}
\begin{aligned}[b] \label{standardinterpolation12dedge_proof23dface}
&\| \mathbf{\Pi}_{k-1}^{0,E} \mathbf{v}- \widetilde{\textbf{I}}^{f}_{h}\textbf{v} \|_{E}
\overset{\eqref{3dfacePRIORIB}}{\lesssim}
h_{E} \|\textrm{div}\hspace{0.05cm}(\mathbf{\Pi}_{k-1}^{0,E} \mathbf{v}- \widetilde{\textbf{I}}^{f}_{h}\textbf{v})\|_{E}\\
& + h_{E}^{\frac12}\|(\mathbf{\Pi}_{k-1}^{0,E} \mathbf{v}- \widetilde{\textbf{I}}^{f}_{h}\textbf{v})\cdot \textbf{n}_{\partial E}\|_{\partial E}
+ \sup_{\textbf{p}_{k}\in (\mathbb{P}_{k}(E))^3}\frac{ \int_{E} (\mathbf{\Pi}_{k-1}^{0,E} \mathbf{v}- \widetilde{\textbf{I}}^{f}_{h}\textbf{v})\cdot \textbf{x}_{E}\wedge \textbf{p}_{k}}{\|\textbf{x}_{E}\wedge \textbf{p}_{k}\|_{E}} \\
& \overset{\eqref{dof2face3Dinter}, \eqref{cavallino}}{\lesssim}
h_{E}\|\textrm{div}\hspace{0.05cm}(\textbf{v} - \mathbf{\Pi}_{k-1}^{0,E} \mathbf{v} )\|_{E}
+h_{E}\|\textrm{div}\hspace{0.05cm}( \mathbf{v}- \widetilde{\textbf{I}}^{f}_{h}\textbf{v})\|_{E} \\ 
& \hspace{0.5cm} +h_{E}^{\frac12}\|(\textbf{v} - \mathbf{\Pi}_{k-1}^{0,E} \mathbf{v}) \cdot\textbf{n}_{\partial E}\|_{\partial E}
+ \| \textbf{v} - \mathbf{\Pi}_{k-1}^{0,E} \mathbf{v} \|_{E}.
\end{aligned}
\end{eqnarray}
We apply the triangle inequality  and    \eqref{standardinterpolation12dedge_proof23dface} to obtain
\begin{eqnarray}
\begin{aligned}[t] \label{standardinterpolation12dface_proof3}
&\|\textbf{v}- \widetilde{\textbf{I}}^{f}_{h}\textbf{v}\|_{E} \leq  \|\textbf{v}-\mathbf{\Pi}_{k-1}^{0,E} \mathbf{v}\|_{E}
+ \| \mathbf{\Pi}_{k-1}^{0,E} \mathbf{v} - \widetilde{\textbf{I}}^{f}_{h} \textbf{v} \|_{E}
\lesssim \|\textbf{v}-\mathbf{\Pi}_{k-1}^{0,E} \mathbf{v}\|_{E} \\
& +  h_{E}\|\textrm{div}\hspace{0.05cm}(\textbf{v} - \mathbf{\Pi}_{k-1}^{0,E} \mathbf{v}) \|_{E}
+h_{E}\|\textrm{div}\hspace{0.05cm}( \mathbf{v}- \widetilde{\textbf{I}}^{f}_{h}\textbf{v})\|_{E} 
+h_{E}^{\frac12 }\|(\textbf{v} - \mathbf{\Pi}_{k-1}^{0,E} \mathbf{v}) \cdot\textbf{n}_{\partial E}\|_{\partial E}.
\end{aligned}
 \end{eqnarray}
If~$s \ge 1$, standard polynomial approximation properties lead to
\[
\begin{split}
& \|\textbf{v}-  \widetilde{\textbf{I}}^{f}_{h}\textbf{v}\|_{E}
\!\!\!\!\!\overset{\eqref{standardinterpolation12dface_proof3}, \eqref{Traceinequality3_1}}{\lesssim}\!\!\!\!\!
\|\textbf{v}-\mathbf{\Pi}_{k-1}^{0,E} \mathbf{v}\|_{E}
+  h_{E}\|\textrm{div}\hspace{0.05cm}(\textbf{v} - \mathbf{\Pi}_{k-1}^{0,E} \mathbf{v})\|_{E}
+h_{E}\|\textrm{div}\hspace{0.05cm}( \mathbf{v}- \widetilde{\textbf{I}}^{f}_{h}\textbf{v})\|_{E} \\
& \hspace{0.5cm}
+ h_{E} | \textbf{v} - \mathbf{\Pi}_{k-1}^{0,E} \mathbf{v} |_{1, E}
\overset{\eqref{interpolation2div3D}}{\lesssim }
h_{E}^{s}\left(|\textbf{v}|_{s,E}
+ |\textrm{div}\hspace{0.05cm}\mathbf{v}|_{s-1,E}\right)
\lesssim h_{E}^{s} |\textbf{v}|_{s,E}.
\end{split}
\]
Instead, if~$1\slash2 < s < 1$, we replace the term $\mathbf{\Pi}_{k-1}^{0,E} \mathbf{v}$ by $\mathbf{\Pi}_{0}^{0,E} \mathbf{v}$ in \eqref{standardinterpolation12dedge_proof23dface} and \eqref{standardinterpolation12dface_proof3},
use standard polynomial approximation properties,
and write
\[
\begin{split}
& \|\textbf{v}-  \widetilde{\textbf{I}}^{f}_{h}\textbf{v}\|_{E}
\overset{\eqref{Traceinequality3_1}}{\lesssim}
\|\textbf{v}-\mathbf{\Pi}_{0}^{0,E} \mathbf{v}\|_{E}
+  h_{E} \|\textrm{div}\hspace{0.05cm}(\textbf{v} - \mathbf{\Pi}_{0}^{0,E} \mathbf{v})\|_{E}  \\
&\hspace{0.5cm} +h_{E}\|\textrm{div}\hspace{0.05cm}( \mathbf{v}- \widetilde{\textbf{I}}^{f}_{h}\textbf{v})\|_{E}  
+h^{s}_{E} | \textbf{v} - \mathbf{\Pi}_{0}^{0,E} \mathbf{v} |_{s, E}
\overset{\eqref{interpolation2div3D}}{\lesssim}
h_{E}^{s}|\textbf{v}|_{s,E} +h_{E} \|\textrm{div}\hspace{0.05cm}\mathbf{v}\|_{E}.
\end{split}
\]
\end{proof}

\subsection{Standard edge virtual element space  on polyhedrons} \label{sec42}
As in Ref.~\cite{beiraobrezzidasimaru2018siam,bebremaru2016RLMASerendipity},
we first introduce the boundary space
\begin{align} \label{B-standard-edge-3D}
\!\mathcal{B}_{k}(\partial E)\!
=\! \left\{\!\textbf{v}_{\!h} \!\in\! \textbf{L}^{\!\!2}(\partial E): \! 
\textbf{v}_{\!h}^{\!F} \!\in\! \textbf{V}_{\!k}^{\!e}(F)\ 
\forall F \!\subseteq\! \partial E, \textbf{v}_{\!h} \!\cdot\! \textbf{t}_{e}\
\textrm{is\ continuous}\ \forall e \!\subseteq\! \partial F \!\right\},
\end{align}
where $\textbf{v}_{h}^{F}$ denotes the tangential component of the vector $\textbf{v}_{h}$ over~$F$ given by
\begin{align} \label{tangentpaprt}
\textbf{v}_{h}^{F}=(\textbf{v}_{h}-(\textbf{v}_{h}\cdot \textbf{n}_{F}) \textbf{n}_{F})|_{F}.
\end{align}
The standard edge virtual element space in 3D  is defined as\cite{beiraobrezzidasimaru2018siam}
   \begin{align*}
\textbf{V}_{k}^{e}(E)=\{\textbf{v}_{h}\in \textbf{L}^{2}(E): \ 
& \textrm{div} \hspace{0.05cm} \textbf{v}_{h}\in \mathbb{P}_{k-1}(E), \textbf{curl}  \hspace{0.05cm} \textbf{curl} \hspace{0.05cm} \textbf{v}_{h} \in (\mathbb{P}_{k}(E))^3,  \\
& \textbf{v}_{h}^{F}\in \textbf{V}_{k}^{e}(F)\ \forall F \subseteq \partial E, \textbf{v}_{h}\cdot \textbf{t}_{e}\ \textrm{is\ continuous}\ \forall e\subseteq \partial F\}.
\end{align*}
We endow the space~$\textbf{V}_{k}^{e}(E)$ with the following set of DoFs:
\begin{align}
\label{dof13dedge}
&\bullet \   \int_{e}\textbf{v}_{h}\cdot \textbf{t}_{e}p_{k}&&\forall p_{k}\in \mathbb{P}_{k}(e),  \hspace{0.1cm} \forall e\subseteq \partial F;\\
\label{dof23dedge}
&\bullet \  \int_{F} \textbf{v}^{F}_{h}\cdot \textbf{x}^{F}_{F} p_{k}  &&\forall p_{k}\in \mathbb{P}_{k}(F);\\
\label{dof33dedge}
&\bullet\    \int_{F} \textrm{rot}_{F}\hspace{0.05cm}\textbf{v}^{F}_{h}  p_{k-1}^{0}   &&\forall p_{k-1}^{0}\in \mathbb{P}_{k-1}^{0}(F)  \hspace{0.2cm} \textrm{only\ for}\  k>1;  \\
\label{dof43dedge}
&\bullet\  \int_{E} \textbf{curl}\hspace{0.05cm} \textbf{v}_{h}\cdot \textbf{x}_{E} \wedge \textbf{p}_{k}    &&\forall \textbf{p}_{k} \in (\mathbb{P}_{k}(E))^{3};\\
\label{dof53dedge}
&\bullet \  \int_{E} \textbf{v}_{h}\cdot \textbf{x}_{E} p_{k-1}  &&\forall p_{k-1}\in \mathbb{P}_{k-1}(E).
\end{align}
The unisolvence of the above DoFs is proven in Section 8.6 of Ref.~\cite{bebremaru2016RLMASerendipity}.
From Proposition 3.7 in Ref.~\cite{beiraobrezzidasimaru2018siam},
the $\textbf{L}^{2}$ projection~$\mathbf{\Pi}_{k}^{0,E}$ from~$\textbf{V}_{k}^{e}(E)$ to $(\mathbb{P}_{k}(E))^{3}$ can be computed by such DoFs.

Next, we recall a well-posedness result for \textbf{curl}-\textbf{curl} systems; for the sake of completeness, we discuss its proof.

\begin{lem}  \label{weposednesscurcur}
For any given $\textbf{v}_{h}\in \textbf{V}_{k}^{e}(E)$, the problem
\begin{eqnarray} \label{eq:4.1}
\left\{
\begin{aligned}
& {\bf curl} \hspace{0.05cm} {\bf curl} \hspace{0.05cm}\bm{\rho}= {\bf curl} \hspace{0.05cm}\textbf{v}_{h},\hspace{0.9cm} {\rm div}\hspace{0.05cm}\bm{\rho}=0   && \textrm{in} \  E,\\
&{\bf curl} \hspace{0.05cm}\bm{\rho}\wedge \textbf{n}_{\partial E}= \textbf{v}_{h}\wedge \textbf{n}_{\partial E}, \hspace{0.4cm}  \bm{\rho}\cdot \textbf{n}_{\partial E}=0 && \textrm{on} \  \partial E,
\end{aligned}
\right.
\end{eqnarray}
has a unique solution $\bm{\rho}$ in $\textbf{H}( {\bf curl}, E)\cap  \textbf{H}({\rm div},E)$.
Moreover, the following a priori bound is valid:
\begin{align} \label{prioribound}
\|\bm{\rho}\|_{E} +h_{E} \|{\bf curl} \hspace{0.05cm}\bm{\rho}\|_{E}
\lesssim h^{2}_{E} \|{\bf curl} \hspace{0.05cm}\textbf{v}_{h}\|_{E}
+h_{E}^{\frac32}\|\textbf{v}_{h}\wedge \textbf{n}_{\partial E}\|_{\partial E}.
  \end{align}
\end{lem}
\begin{proof}
To see that~\eqref{eq:4.1} is well-posed, we introduce the auxiliary variable $\bm{\sigma}:={\bf curl} \hspace{0.05cm}\bm{\rho}$.
Then, \eqref{eq:4.1} can be equivalently decomposed into the two following   problems:\\
$\bullet$ for given $\textbf{v}_{h}\in \textbf{V}_{k}^{e}(E)$, find $\bm{\sigma}\in  \textbf{H}( {\bf curl}, E)\cap  \textbf{H}({\rm div},E)$ such that
\begin{eqnarray} \label{eq:4.1_1}
\left\{
\begin{aligned}
& {\bf curl} \hspace{0.05cm} {\bm \sigma}= {\bf curl} \hspace{0.05cm}\textbf{v}_{h},\hspace{0.4cm} {\rm div}\hspace{0.05cm}\bm{\sigma}=0   && \textrm{in} \  E,\\
& \bm{\sigma}\wedge \textbf{n}_{\partial E}= \textbf{v}_{h}\wedge \textbf{n}_{\partial E}  && \textrm{on} \  \partial E;
\end{aligned}
\right.
\end{eqnarray}
$\bullet$ find  $\bm{\rho}\in  \textbf{H}({\bf curl} , E)\cap  \textbf{H}({\rm div},E)$ such that
 \begin{eqnarray}
\label{eq:4.1_2}
\left\{
\begin{aligned}
& {\bf curl}  \hspace{0.05cm}\bm{\rho}= \bm{ \sigma},\hspace{0.4cm} {\rm div}\hspace{0.05cm}\bm{\rho}=0  && \textrm{in} \  E,\\
&  \bm{\rho}\cdot \textbf{n}_{\partial E}=0 && \textrm{on} \  \partial E.
\end{aligned}
\right.
\end{eqnarray}
Since the above div-${\bf curl}$ systems are uniquely solvable\cite{aualexa2005,beiraobrezzimariniru2015},
\eqref{eq:4.1} has a unique solution.
Next, we prove~\eqref{prioribound}. We first observe that
\begin{equation}
\begin{aligned} \label{prioriproof1}
\|{\bf curl} \hspace{0.05cm}\bm{\rho}\|_{E}
\overset{\eqref{eq:4.1_2}}{=}
\| \bm{\sigma}\|_{E}
\overset{\eqref{poincarefridinecurl},\eqref{eq:4.1_1}}{\lesssim}
 h_{E} \|{\bf curl} \hspace{0.05cm}\textbf{v}_{h}\|_{E}
+h_{E}^{\frac12}\|\textbf{v}_{h}\wedge \textbf{n}_{\partial E}\|_{\partial E}.
\end{aligned}
\end{equation}
Furthermore, we have
\begin{equation}
\begin{aligned} \label{prioriproof2}
\|\bm{\rho}\|_{E}
\overset{\eqref{poincarefridinediv}, \eqref{eq:4.1_2}}{\lesssim}
h_{E} \|{\bf curl}\hspace{0.05cm}\bm{\rho}\|_{E}
\overset{\eqref{prioriproof1}}{\lesssim}
h^{2}_{E} \|{\bf curl} \hspace{0.05cm}\textbf{v}_{h}\|_{E}
+h_{E}^{\frac32}\|\textbf{v}_{h}\wedge \textbf{n}_{\partial E}\|_{\partial E}.
\end{aligned}
\end{equation}
The assertion follows combining~\eqref{prioriproof1} and~\eqref{prioriproof2}.
\end{proof}
We could have proved Lemma~\ref{weposednesscurcur}
by writing~\eqref{eq:4.1} in mixed form\cite{kikuchi1989,sunjiguangNM2016}.
In the following result, we prove an auxiliary bound for  functions in $\textbf{V}_{k}^{e}(E)$.
\begin{lem} \label{lem3dedgebound}
For each $\textbf{v}_{h}\in \textbf{V}_{k}^{e}(E)$, we have
\begin{eqnarray}
\begin{aligned}[b] \label{3dedgePRIORIB}
& \|\textbf{v}_{h}\|_{E} 
\lesssim   \!\!\sum_{F\subseteq \partial E}\!\!\!
    \left(h_{F}^{\frac32} \| {\bf curl}\hspace{0.05cm} \textbf{v}_{h} \cdot\textbf{n}_{F} \|_{F}
    \!+\!  h_{F}   \|\textbf{v}_{h}^{F}\cdot\textbf{t}_{\partial F} \|_{\partial F} + 
    \!\!\!\!\!\!\sup_{p_{k}\in \mathbb{P}_{k}(F)}\!\!\!\!\!\!\!
    \frac{ h_{F}^{\frac12}\int_{F}  \textbf{v}^{F}_{h} \cdot \textbf{x}^{F}_{F} p_{k}}{\|\textbf{x}^{F}_{F} p_{k}\|_{F}}\right)   \\
&   \qquad \hspace{0.5cm}  + \!\!\!\!\!\!\sup_{\textbf{p}_{k}\in (\mathbb{P}_{k}(E))^3}\!\!\!\!\!\!
\frac{h_{E} \int_{E}  {\bf curl}  \hspace{0.05cm}  \textbf{v}_{h}\cdot \textbf{x}_{E}\wedge \textbf{p}_{k}}{\|\textbf{x}_{E}\wedge \textbf{p}_{k}\|_{E}}
+ \!\!\!\!\!\!\sup_{p_{k-1}\in \mathbb{P}_{k-1}(E)}\!\!\!\!\!\!
    \frac{\int_{E}  \textbf{v}_{h}   \cdot\textbf{x}_{E}p_{k-1}   }{\|\textbf{x}_{E}p_{k-1}\|_{E}}.
\end{aligned}
\end{eqnarray}
\end{lem}
\begin{proof}
We first prove that there exist $\psi\in H^{1}(E)\setminus\mathbb{R}$ and $\bm{\rho}\in  \textbf{H}(\textbf{curl}, E)\cap  \textbf{H}(\textrm{div},E)$  such that
  the following Helmholtz decomposition of $\textbf{v}_{h}$ is valid:
\begin{align}
 \label{standardinterpolation1_proof63dedge}  \textbf{v}_{h}=\textbf{curl} \hspace{0.05cm}\bm{\rho}+\bm{\nabla}  \psi.
\end{align}
 To prove \eqref{standardinterpolation1_proof63dedge}, we   define a function $\psi\in H^{1}(E)$ satisfying weakly
 \begin{align}
 \label{standardinterpolation1_proof23dedge}
 \Delta \psi= \textrm{div} \hspace{0.05cm}\textbf{v}_{h} \ \textrm{in}\ E,\ \  \psi =0\ \textrm{on}\  \partial E,
\end{align}
and a function $\bm{\rho}\in  \textbf{H}(\textbf{curl}, E)\cap  \textbf{H}(\textrm{div},E)$  satisfying weakly
 \begin{eqnarray}
\label{standardinterpolation1_proof33dedge}
\left\{
\begin{aligned}
& {\bf curl} \hspace{0.05cm} {\bf curl} \hspace{0.05cm}\bm{\rho}= {\bf curl} \hspace{0.05cm}\textbf{v}_{h},\hspace{0.9cm} {\rm div}\hspace{0.05cm}\bm{\rho}=0   && \textrm{in} \  E,\\
&{\bf curl} \hspace{0.05cm}\bm{\rho}\wedge \textbf{n}_{\partial E}= \textbf{v}_{h}\wedge \textbf{n}_{\partial E}, \hspace{0.4cm}  \bm{\rho}\cdot \textbf{n}_{\partial E}=0 && \textrm{on} \  \partial E.
\end{aligned}
\right.
\end{eqnarray}
Lemma~\ref{weposednesscurcur} implies the well posedness of~\eqref{standardinterpolation1_proof33dedge}.
Identity~\eqref{standardinterpolation1_proof63dedge} easily follows from~\eqref{standardinterpolation1_proof23dedge}, \eqref{standardinterpolation1_proof33dedge},
and the fact that~$E$ is simply connected.
We also have
\begin{align} \label{standardinterpolation1_proof73dedge} 
(\textbf{curl} \hspace{0.05cm}\bm{\rho},\bm{\nabla}  \psi)_{E}=0,
\qquad 
\|\textbf{v}_{h} \|^{2}_{E}=\|\textbf{curl} \hspace{0.05cm}\bm{\rho}\|^{2}_{E}+\|\bm{\nabla}  \psi\|^{2}_{E}.
\end{align}
Since $\|  \bm{\rho}^{F}\|_{F}=\|  \bm{\rho} \wedge \textbf{n}_{F}\|_{F}$ for all~$F$ in~$\partial E$,
cf. \eqref{tangentpaprt}, we obtain
\begin{eqnarray}
\begin{aligned} \label{standardinterpolation1_proof93dedge}
& \|\textbf{curl}\hspace{0.05cm} \bm{\rho}\|^{2}_{E}  
\overset{\text{IBP}}{=}
\int_{E}  \bm{\rho}  \cdot \textbf{curl}\hspace{0.05cm} \textbf{curl}\hspace{0.05cm} \bm{\rho} \,
- \int_{\partial E}  (\textbf{curl}\hspace{0.05cm}\bm{\rho}\wedge \textbf{n}_{\partial E})  \cdot   \bm{\rho}  \\
& \overset{\eqref{standardinterpolation1_proof33dedge}, \eqref{tangentpaprt}}{=}
\int_{E}  \bm{\rho}  \cdot \textbf{curl}\hspace{0.05cm} \textbf{v}_{h} \,
 - \sum_{F \subset \partial E} \int_{F}  (\textbf{v}_{h} \wedge \textbf{n}_{F})  \cdot \bm{\rho}^F \\
& \le
\| \bm{\rho}\|_{E}  \| \textbf{curl}\hspace{0.05cm} \textbf{v}_{h}\|_{E} 
+ \| \textbf{v}_{h} \wedge \textbf{n}_{\partial E}\|_{\partial E}  \|\bm{\rho}\wedge \textbf{n}_{\partial E}\|_{\partial E}\\
& \overset{\eqref{Traceinequality5}, \eqref{standardinterpolation1_proof33dedge}}{\lesssim}
\| \bm{\rho}\|_{E}  \| \textbf{curl}\hspace{0.05cm} \textbf{v}_{h}\|_{E} + \left(h_{E}^{-\frac12} \| \bm{\rho}\|_{E} 
+h_{E}^{\frac12} \| \textbf{curl}\hspace{0.05cm} \bm{\rho}\|_{E}   \right)\| \textbf{v}_{h} \wedge \textbf{n}_{\partial E}\|_{\partial E}   \\
& \overset{\eqref{poincarefridinediv},\eqref{standardinterpolation1_proof33dedge}}{\lesssim}
\left( h_{E} \| \textbf{curl}\hspace{0.05cm} \textbf{v}_{h}\|_{E}
+h_{E}^{\frac12}\| \textbf{v}^{F}_{h} \|_{\partial E} \right) \| 
\textbf{curl}\hspace{0.05cm}\bm{\rho}\|_{E}.
\end{aligned}
\end{eqnarray}
In view of \eqref{decompo13d} and  $\textrm{div} \hspace{0.05cm}  \textbf{v}_{h}\in \mathbb{P}_{k-1}(E)$,
there exists $q_{k-1}\in \mathbb{P}_{k-1}(E)$ such that
\begin{align} \label{standardinterpolation2_proof1_13dedge}
\textrm{div}\hspace{0.05cm} (\textbf{x}_{E}q_{k-1})
=\textrm{div}\hspace{0.05cm}  \textbf{v}_{h}
\qquad \textrm{and} \qquad
\|\textbf{x}_{E}q_{k-1}\|_{E}
\lesssim h_{E}\|\textrm{div} \hspace{0.05cm}  \textbf{v}_{h}\|_{E}.
\end{align}
We obtain
\begin{eqnarray}
 \begin{aligned}[t]
 \label{standardinterpolation2_proof103dedgepriori}
&\hspace{0.5cm}\|\bm{\nabla}  \psi\|^{2}_{E}  
\overset{\text{IBP}, \eqref{standardinterpolation1_proof23dedge}}{=}
- \int_{E} \textrm{div}\hspace{0.05cm} \textbf{v}_{h} \psi
\overset{\eqref{standardinterpolation2_proof1_13dedge}}{=}
- \int_{E}  \textrm{div}\hspace{0.05cm} (\textbf{x}_{E}q_{k-1}) \psi \\
& \!\!\overset{\text{IBP}, \eqref{standardinterpolation1_proof23dedge}, \eqref{standardinterpolation1_proof63dedge}}{=}\!\!
\int_{E}  (\textbf{x}_{E}q_{k-1}) \!\cdot\! (\textbf{v}_{h}\!-\!\textbf{curl} \hspace{0.05cm}\bm{\rho})
\!\leq\! 
\|\textbf{x}_{E}q_{k-1}\|_{E}\|\textbf{curl} \hspace{0.05cm}\bm{\rho}\|_{E}
\!+\! \int_{E}  \textbf{v}_{h} \!\cdot\! \textbf{x}_{E}q_{k-1}   \\
& \!\!\overset{\eqref{standardinterpolation1_proof93dedge}, \eqref{standardinterpolation2_proof1_13dedge}}{\lesssim}\!\!\!\!\!\!\!
h_{E}\left( h_{E}\|  \textbf{curl}\hspace{0.05cm} \textbf{v}_{h}\|_{E}
+h_{E}^{\frac12}\| \textbf{v}^{F}_{h} \|_{\partial E}
+ \!\!\!\!\!\!\!\!\!\sup_{p_{k-1}\in \mathbb{P}_{k-1}(E)}\!\!\!\!\!
\frac{\int_{E}  \textbf{v}_{h}   \cdot\textbf{x}_{E}p_{k-1}   }{\|\textbf{x}_{E}p_{k-1}\|_{E} }\right)\| \textrm{div}\hspace{0.05cm}\textbf{v}_{h} \|_{E}.\\
\end{aligned}
\end{eqnarray}
Recall that the 2D and 3D spaces here analyzed
constitute an exact complex\cite{beiraobrezzidasimaru2018siam},
whence~$\textbf{curl}\hspace{0.05cm} \textbf{v}_{h} \in \textbf{V}_{k-1}^{f}(E)$.
Since $\textbf{v}_{h} \in  \textbf{V}_{k}^{e}(F)$ for each~$F$ in~$\partial E$, we have
\begin{equation}\label{4.36.5}
\begin{split}
&\| \textbf{v}_{h}^{F}\|_{F}
\overset{\eqref{prioribound2dedge}}{\lesssim}
h_{F}\|\textrm{rot}_{F}\hspace{0.05cm} \textbf{v}_{h}^{F}\|_{F}+h_{F}^{\frac12}\| \textbf{v}_{h}^{F}\cdot\textbf{t}_{\partial F}\|_{\partial F}+ 
\!\!\!\!\sup_{p_{k}\in \mathbb{P}_{k}(F)}\!\!\!\!
\frac{ \int_{F}  \textbf{v}^{F}_{h} \cdot \textbf{x}^{F}_{F} p_{k}}{\|\textbf{x}^{F}_{F} p_{k}\|_{F}},\\
&  \|\textbf{curl}  \hspace{0.05cm}  \textbf{v}_{h}\|_{E} 
\overset{\eqref{3dfacePRIORIB}}{\lesssim}
h_{E}^{\frac12}\|\textbf{curl}  \hspace{0.05cm}  \textbf{v}_{h}\cdot \textbf{n}_{\partial E}\|_{\partial E}  +
\!\!\!\!\!\!\!\!\sup_{\textbf{p}_{k}\in (\mathbb{P}_{k}(E))^3}\!\!\!\!\!\!\!\!
\frac{ \int_{E} \textbf{curl}  \hspace{0.05cm}  \textbf{v}_{h}\cdot \textbf{x}_{E}\wedge \textbf{p}_{k}}{\|\textbf{x}_{E}\wedge \textbf{p}_{k}\|_{E}}.
\end{split}
\end{equation}
By the fact that $\textrm{div}\hspace{0.05cm}\textbf{v}_{h} \in \mathbb{P}_{k-1}(E)$
and employing arguments similar to those used in proving~\eqref{3dcurlinverseesti},
we have the following inverse estimate  involving   edge virtual element functions in 3D:
\[
\| \textrm{div}\hspace{0.05cm}\textbf{v}_{h} \|_{E} 
\lesssim h_{E}^{-1}  \|  \textbf{v}_{h} \|_{E} \hspace{0.3cm} \forall \textbf{v}_{h}\in \textbf{V}_{k}^{e}(E).
\]
We plug this and~\eqref{4.36.5}
in~\eqref{standardinterpolation2_proof103dedgepriori},
and deduce
\begin{eqnarray}
 \begin{aligned}[t]
 \label{standardinterpolation2_proof103dedge}
& \|\bm{\nabla}  \psi\|^{2}_{E} 
\lesssim \Bigg[  \sup_{p_{k-1}\in \mathbb{P}_{k-1}(E)} \frac{\int_{E}  \textbf{v}_{h}   \cdot\textbf{x}_{E}p_{k-1}   }{\|\textbf{x}_{E}p_{k-1}\|_{E} } \\
&\hspace{0.5cm} + h_{E}^{\frac12} \sum_{F\subseteq\partial E}\left(h_{F}\|\textrm{rot}_{F}\hspace{0.05cm} \textbf{v}_{h}^{F}\|_{F}
+ h_{F}^{\frac12}\|\textbf{v}_{h}^{F}\cdot \textbf{t}_{\partial F}\|_{\partial F}
+ \sup_{p_{k}\in \mathbb{P}_{k}(F)}\frac{ \int_{F}  \textbf{v}^{F}_{h} \cdot \textbf{x}^{F}_{F} p_{k}}{\|\textbf{x}^{F}_{F} p_{k}\|_{F}}\right) \\
& \hspace{0.5cm} +h_{E}   \left ( h_{E}^{\frac12}\|\textbf{curl}  \hspace{0.05cm}  \textbf{v}_{h}\cdot \textbf{n}_{\partial E}\|_{\partial E} 
+\sup_{\textbf{p}_{k}\in (\mathbb{P}_{k}(E))^3}\frac{ \int_{E} \textbf{curl}  \hspace{0.05cm}  \textbf{v}_{h}\cdot \textbf{x}_{E}\wedge \textbf{p}_{k}}{\|\textbf{x}_{E}\wedge \textbf{p}_{k}\|_{E}}\right) \Bigg]\| \textbf{v}_{h} \|_{E}.
\end{aligned}
\end{eqnarray}
Inserting~\eqref{standardinterpolation1_proof93dedge}
and~\eqref{standardinterpolation2_proof103dedge}
into~\eqref{standardinterpolation1_proof73dedge},
using  $h_{F}\approx h_{E}$, and noting that
$\textrm{rot}_{F}\hspace{0.05cm} \textbf{v}_{h}^{F}
= (\textbf{curl}\hspace{0.05cm} \textbf{v}_{h})|_{F} \cdot\textbf{n}_{F}$ 
for all~$F$ in~$\partial E$, yield
\begin{align*}
& \|\textbf{v}_{h} \|^{2}_{E}
\lesssim \Bigg[ h_{E} \| \textbf{curl}\hspace{0.05cm} \textbf{v}_{h}\|_{E}+h_{E}^{\frac12}\| \textbf{v}^{F}_{h} \|_{\partial E} 
+  \!\!\!\!\!\!\!\!\sup_{p_{k-1}\in \mathbb{P}_{k-1}(E)}\!\!\!\!\!\!\!\!
\frac{\int_{E}  \textbf{v}_{h}   \cdot\textbf{x}_{E}p_{k-1}   }{\|\textbf{x}_{E}p_{k-1}\|_{E} }\\
& \hspace{0.5cm} +  h_{E}^{\frac12} \sum_{F\subseteq \partial E} \left(h_{F}\|\textrm{rot}_{F}\hspace{0.05cm} \textbf{v}_{h}^{F}\|_{F}
+ h_{F}^{\frac12}\|\textbf{v}_{h}^{F}\cdot \textbf{t}_{\partial F}\|_{\partial F}
+ \sup_{p_{k}\in \mathbb{P}_{k}(F)}\frac{ \int_{F}  \textbf{v}^{F}_{h} \cdot \textbf{x}^{F}_{F} p_{k}}{\|\textbf{x}^{F}_{F} p_{k}\|_{F}}\right)\\
& \hspace{0.5cm}
+h_{E}   \left ( h_{E}^{\frac12}\|\textbf{curl}  \hspace{0.05cm}  \textbf{v}_{h}\cdot \textbf{n}_{\partial E}\|_{\partial E}  
+\sup_{\textbf{p}_{k}\in (\mathbb{P}_{k}(E))^3}\frac{ \int_{E} \textbf{curl}  \hspace{0.05cm}  \textbf{v}_{h}\cdot \textbf{x}_{E}\wedge \textbf{p}_{k}}{\|\textbf{x}_{E}\wedge \textbf{p}_{k}\|_{E}}\right)\Bigg]\|\textbf{v}_{h}\|_{E} \\
& \lesssim
\Bigg[ \sum_{F\subseteq \partial E} \left(h_{F}^{\frac32} \| \textbf{curl}\hspace{0.05cm} \textbf{v}_{h} \cdot\textbf{n}_{F} \|_{F}+ h_{F}   \|\textbf{v}_{h}^{F}\cdot\textbf{t}_{\partial F} \|_{\partial F} + 
\!\!\!\!\!\!\!\sup_{p_{k}\in \mathbb{P}_{k}(F)}\!\!\!\!\!\!\!
\frac{ h_{F}^{\frac12}\int_{F}  \textbf{v}^{F}_{h} \cdot \textbf{x}^{F}_{F} p_{k}}{\|\textbf{x}^{F}_{F} p_{k}\|_{F}}\right)\\
&  \qquad\qquad
+ \!\!\!\!\!\!\!\! \sup_{\textbf{p}_{k}\in (\mathbb{P}_{k}(E))^3}\!\!\!\!\!\!\!\!
\frac{h_{E} \int_{E} \textbf{curl}  \hspace{0.05cm}  \textbf{v}_{h}\cdot \textbf{x}_{E}\wedge \textbf{p}_{k}}{\|\textbf{x}_{E}\wedge \textbf{p}_{k}\|_{E}}  
+ \!\!\!\!\!\!\!\!\! \sup_{p_{k-1}\in \mathbb{P}_{k-1}(E)}\!\!\!\!\!
\frac{\int_{E}  \textbf{v}_{h}   \cdot\textbf{x}_{E}p_{k-1}   }{\|\textbf{x}_{E}p_{k-1}\|_{E}} \Bigg]\|  \textbf{v}_{h}\|_{E}.
\end{align*}
\end{proof}

For each sufficiently regular $\textbf{v}$, we define the DoFs interpolation operator $\widetilde{\textbf{I}}_{h}^{e}$ on $\textbf{V}_{k}^{e}(E)$ by
   \begin{subequations}
   \label{dof3dedgeseinterini}
\begin{align}
\label{dof13dedgeseinterini}
&\int_{e}(\textbf{v}-\widetilde{\textbf{I}}_{h}^{e} \textbf{v})\cdot \textbf{t}_{e}p_{k}=0&&\forall p_{k}\in \mathbb{P}_{k}(e),  \hspace{0.1cm} \forall e\subseteq \partial F;\\
\label{dof23dedgeseinterini}
&\int_{F} (\textbf{v}-\widetilde{\textbf{I}}_{h}^{e} \textbf{v})^{F} \cdot \textbf{x}^{F}_{F} p_{k}  =0&&\forall p_{k}\in \mathbb{P}_{k}(F);\\
\label{dof33dedgeseinterini}
&\int_{F} \textrm{rot}_{F}\hspace{0.05cm}(\textbf{v}-\widetilde{\textbf{I}}_{h}^{e} \textbf{v})^{F}   p_{k-1}^{0} =0  &&\forall p_{k-1}^{0}\in \mathbb{P}_{k-1}^{0}(F)  \hspace{0.2cm} \textrm{only\ for}\  k>1;  \\
\label{dof43dedgeseinterini}
&\int_{E} \textbf{curl}\hspace{0.05cm} (\textbf{v}-\widetilde{\textbf{I}}_{h}^{e} \textbf{v})\cdot \textbf{x}_{E} \wedge \textbf{p}_{k}   =0 &&\forall \textbf{p}_{k} \in (\mathbb{P}_{k}(E))^{3};\\
\label{dof53dedgeseinterini}
& \int_{E} (\textbf{v}-\widetilde{\textbf{I}}_{h}^{e}\textbf{v})\cdot \textbf{x}_{E} p_{k-1}  =0&&\forall p_{k-1}\in \mathbb{P}_{k-1}(E).
\end{align}
\end{subequations}
Next, we prove interpolation estimates for the operator $\widetilde{\textbf{I}}_{h}^{e}$.
The following result includes different requirements on the regularity of the objective function; see also Remark \ref{rem:koala}.
Below, given any non-negative real number $s$, the symbol $[s]$ will denote the highest integer strictly smaller than~$s$
($[\cdot]$ differs from the $\text{floor}(\cdot)$ function;
for instance, $[1]=0$ while~$\text{floor}(1)=1$).
\begin{thm} \label{theorem3dedgeinter}
For each $\textbf{v}\in \textbf {H}^{s}(E)$, $1\slash2 < s\leq k+1$,
with ${\bf curl}\hspace{0.05cm}\textbf{v} \in \textbf {H}^{r}(E)$,
$1\slash2$ $< r \leq k$,
for $\widehat{r} = \min{\{r,[s]\}}$,
we have
\begin{align} \label{interpolation1curl3D1}
& \|\textbf{v}-\widetilde{\textbf{I}}_{h}^{e}\textbf{v}\|_{E}
\lesssim   h_{E}^{s} |    \textbf{v}|_{s,E} +  h_{E}^{\widehat{r}+1}| {\bf curl}\hspace{0.05cm} \textbf{v}  |_{\widehat{r},E} +  h_{E} \| {\bf curl}\hspace{0.05cm} \textbf{v}  \|_{E} , \\
\label{interpolation2curl3D}
&\|{\bf curl} \hspace{0.05cm}(\textbf{v}-\widetilde{\textbf{I}}_{h}^{e}\textbf{v})\|_{E}
\lesssim  h_{E}^{r}|{\bf curl}\hspace{0.05cm}\textbf{v}|_{r,E}.
\end{align}
The  third  term  on the right-hand side of \eqref{interpolation1curl3D1} can be neglected if $s \ge 1$.
\end{thm}

\begin{proof}
Following Proposition 4.2 in Ref.~\cite{beiraobrezzidasimaru2018siam},   we have
 \begin{align}
 \label{relationF_E}
 \textbf{curl}\hspace{0.05cm}(\widetilde{\textbf{I}}_{h}^{e} \textbf{v})= \widetilde{\textbf{I}}_{h}^{f} (\textbf{curl}\hspace{0.05cm}\textbf{v}).\end{align}
Recalling~\eqref{interpolation1div3D},
bound~\eqref{interpolation2curl3D} immediately follows.
 
Next, we prove bound~\eqref{interpolation1curl3D1}.
We define the natural number $\bar{k} = [s] \le k$ and consider $\mathbf{\Pi}_{\bar k}^E$ the (vector valued version of the)
projection operator from $H^{s}(E)$ in $\mathbb{P}_{\bar{k}}(E)$
defined in Ref.~\cite{Verfuerth1999}.
Such an operator guarantees the following approximation properties
\begin{equation} \label{God-bless}
\| \mathbf{v} - \mathbf{\Pi}_{\bar k}^E \mathbf{v} \|_E 
\lesssim h_E^s |\mathbf v|_{s,E} \, , \quad
\| \textbf{curl} \hspace{0.05cm} \mathbf{v} -\textbf{curl} \hspace{0.05cm} \mathbf{\Pi}_{\bar k}^E \mathbf{v} \|_E
\lesssim h_E^{\widehat{r}} | \textbf{curl}  \hspace{0.05cm} \mathbf{v}  |_{\widehat{r},E} .
\end{equation}
To show the second bound in~\eqref{God-bless},
it suffices to recall the properties of $\mathbf{\Pi}_{\bar k}^E \mathbf{v}$.
In particular, see~\cite[eqs. ($2.1$) and ($2.2$)]{Verfuerth1999},
all its partial derivatives (up to order ${\bar k}$) have the same average as those of~$\mathbf v$.
This implies that also the derivatives (up to one order less)
of the \textbf{curl} of the two functions have the same average.
The estimate follows from iterative applications of the Poincar\'e inequality.

Since~$\mathbf{\Pi}_{\bar k}^E \mathbf{v}\in (\mathbb{P}_{k}(E))^{3}\subseteq \textbf{V}_{k}^{e}(E)$, we obtain
\begin{eqnarray}  \label{standardinterpolation13dcurl_proof}
\begin{aligned}[t]
&\| \mathbf{\Pi}_{\bar k}^E \mathbf{v}- \widetilde{\textbf{I}}_{h}^{e}\textbf{v} \|_{E}
\overset{\eqref{3dedgePRIORIB}}{\lesssim}
\sum_{F\subseteq \partial E}h_{F}^{\frac32} \| \textbf{curl}\hspace{0.05cm}  ( \mathbf{\Pi}_{\bar k}^E \mathbf{v}- \widetilde{\textbf{I}}_{h}^{e}\textbf{v}) \cdot\textbf{n}_{F}\|_{F}\\
& +   \sum_{F\subseteq \partial E}
h_{F} \big(\|( \mathbf{\Pi}_{\bar k}^E \mathbf{v}- \widetilde{\textbf{I}}_{h}^{e}\textbf{v})^{F}\cdot\textbf{t}_{\partial F} \|_{\partial F}+
\!\!\!\!\!\!\!\sup_{p_{k}\in \mathbb{P}_{k}(F)}\!\!\!\!\!\!\!
\frac{ h_{F}^{\frac12}\int_{F}  ( \mathbf{\Pi}_{\bar k}^E \mathbf{v}- \widetilde{\textbf{I}}_{h}^{e}\textbf{v})^{F} \cdot \textbf{x}^F_{F} p_{k}}{\|\textbf{x}^F_{F} p_{k}  \|_{F}} \big)  \\
& + \!\!\!\!\!\!\!\!\sup_{\textbf{p}_{k}\in (\mathbb{P}_{k}(E))^3}\!\!\!\!\!\!\!\!
\frac{ h_{E} \int_{E} \textbf{curl}  \hspace{0.05cm}  ( \mathbf{\Pi}_{\bar k}^E \mathbf{v} \!-\! \widetilde{\textbf{I}}_{h}^{e}\textbf{v}) \!\cdot\! \textbf{x}_{E} \!\wedge\! \textbf{p}_{k}}{\| \textbf{x}_{E} \!\wedge\! \textbf{p}_{k}\|_{E}} 
+\!\!\!\!\!\!\!\!\!\!\!\sup_{p_{k-1}\in \mathbb{P}_{k-1}(E)}\!\!\!\!\!\!\!\!\!\!
\frac{ \int_{E}  ( \mathbf{\Pi}_{\bar k}^E \mathbf{v} \!-\! \widetilde{\textbf{I}}_{h}^{e}\textbf{v})   \!\cdot\! \textbf{x}_{E}p_{k-1}   }{\|\textbf{x}_{E}p_{k-1}  \|_{E}}
\!:=\!  \sum_{i=1}^{5}T_{i} . 
\end{aligned}
 \end{eqnarray}
We estimate the five terms on the right-hand side of~\eqref{standardinterpolation13dcurl_proof} separately.
First, we observe
\begin{align}
\label{eq173}
\textbf{curl}\hspace{0.05cm}  (\widetilde{\textbf{I}}_{h}^{e}\textbf{v})|_{F} \cdot\textbf{n}_{F}
\overset{\eqref{relationF_E}, \eqref{dof1face3Dinter}}{=}
\Pi_{k-1}^{0,F}((\textbf{curl}\hspace{0.05cm}\textbf{v})|_{F}\cdot\textbf{n}_{F})\hspace{0.3cm} \forall F\subseteq \partial E.
\end{align}
As for the term~$T_1$,
the triangle inequality and the trivial continuity of the
projector~$\Pi_{k-1}^{0,F}$ in the $L^2$ norm implies
\begin{eqnarray}
\begin{aligned}[b]
 \label{standardinterpolation13dcurl_proof_1}
T_{1}
&\leq  \sum_{F\subseteq \partial E}  h_{F}^{\frac32}  
\Big(\| \textbf{curl}\hspace{0.05cm}  ( \mathbf{v} \!-\! \mathbf{\Pi}_{\bar k}^E \textbf{v}) \cdot\textbf{n}_{F} \|_{F}
\! +\! \| \textbf{curl}\hspace{0.05cm}    \mathbf{v}\!\cdot\!\textbf{n}_{F} 
\!-\! \Pi_{k-1}^{0,F} (\textbf{curl}\hspace{0.05cm} \mathbf{v} \!\cdot\! \textbf{n}_{F}) \|_{F} \Big) \\
& \overset{\eqref{eq173}}{\lesssim}
\sum_{F\subseteq \partial E}  h_{F}^{\frac32} \| \textbf{curl}\hspace{0.05cm}  (\mathbf{v} \!-\! \mathbf{\Pi}_{\bar k}^E \textbf{v}) \cdot\textbf{n}_{F} \|_{F}.
\end{aligned}
 \end{eqnarray}
We estimate the terms~$T_3$, $T_4$, and~$T_5$ as follows:
\begin{eqnarray}
\begin{aligned}[t]
 \label{standardinterpolation13dcurl_proof_4}
&\sum_{i=3}^{5}T_i 
\overset{\eqref{dof23dedgeseinterini}, \eqref{dof43dedgeseinterini}, \eqref{dof53dedgeseinterini}}{=}
\sum_{F\subseteq \partial E}\sup_{p_{k}\in \mathbb{P}_{k}(F)} \frac{ h_{F}^{\frac12}\int_{F}  ( \mathbf{v} - \mathbf{\Pi}_{\bar k}^E  \textbf{v})^{F} \cdot \textbf{x}^F_{F} p_{k}}{\|\textbf{x}^F_{F} p_{k}  \|_{F}} \\
&   \hspace{0.5cm} +
\!\!\!\!\!\!\!\!\sup_{p_{k-1}\in \mathbb{P}_{k-1}(E)}\!\!\!\!\!\!\!\!
\frac{ \int_{E}  ( \mathbf{v} - \mathbf{\Pi}_{\bar k}^E \textbf{v})   \cdot\textbf{x}_{E}p_{k-1}   }{\|\textbf{x}_{E}p_{k-1}  \|_{E}}
+\!\!\!\!\!\!\!\!\sup_{\textbf{p}_{k}\in (\mathbb{P}_{k}(E))^3}\!\!\!\!\!\!\!\!
\frac{ h_{E} \int_{E} \textbf{curl}  \hspace{0.05cm}  (\mathbf{v} - \mathbf{\Pi}_{\bar k}^E \textbf{v})\cdot \textbf{x}_{E}\wedge \textbf{p}_{k}}{\| \textbf{x}_{E}\wedge \textbf{p}_{k}\|_{E}}\\
& \lesssim \sum_{F\subseteq \partial E}   h_{F}^{\frac12} \| (\mathbf{v} - \mathbf{\Pi}_{\bar k}^E \textbf{v})^{F}\|_{F} 
+ \| \mathbf{v} - \mathbf{\Pi}_{\bar k}^E \textbf{v} \|_{E}
+ h_{E}\| \textbf{curl}  \hspace{0.05cm}  (\mathbf{v} - \mathbf{\Pi}_{\bar k}^E \textbf{v}) \|_{E}.
\end{aligned}
 \end{eqnarray}
Inserting~\eqref{standardinterpolation13dcurl_proof_1} and~\eqref{standardinterpolation13dcurl_proof_4}
into~\eqref{standardinterpolation13dcurl_proof} yields
 \begin{eqnarray}
\begin{aligned}[b] \label{standardinterpolation13dcurl_prooflast}
& \| \mathbf{\Pi}_{\bar k}^E \mathbf{v}- \widetilde{\textbf{I}}_{h}^{e}\textbf{v} \|_{E}
\!\lesssim\!\! \sum_{F\subseteq \partial E} \Big(h_{F}^{\frac32} 
\| \textbf{curl}\hspace{0.05cm}  (\mathbf{v} - \mathbf{\Pi}_{\bar k}^E \textbf{v}) \cdot\textbf{n}_{F} \|_{F}
\!+\! h^{\frac12}_{F}   \|(\mathbf{v} - \mathbf{\Pi}_{\bar k}^E\textbf{v} )^{F}\|_{F}  \Big)\\
& \qquad\qquad\qquad\qquad + \| \mathbf{v} - \mathbf{\Pi}_{\bar k}^E \textbf{v} \|_{E}
+ h_{E}\| \textbf{curl}  \hspace{0.05cm}  (\mathbf{v} - \mathbf{\Pi}_{\bar k}^E \textbf{v}) \|_{E}+T_2.
\end{aligned}
\end{eqnarray}
We are left to estimate the term~$T_2$.
If $s>1$, then by~\eqref{dof13dedgeseinterini},
\eqref{Traceinequality3_1} with $1\slash2$ $< \delta < \min\{1,s- 1\slash2\}$, and~\eqref{Traceinequality3_1_1} with~$\varepsilon=\delta$, we have
\begin{eqnarray}
\begin{aligned} \label{standardinterpolation13dcurl_proof_3slarge_proof3} 
T_2 
& \!=\! h_{F} \| (\mathbf{\Pi}_{\bar k}^E \mathbf{v} \!-\! \widetilde{\textbf{I}}_{h}^{e}\textbf{v})^{F} \!\cdot\! \textbf{t}_{\partial F} \|_{\partial F}
\!\lesssim\! \sum_{F\subseteq \partial E}
h_{\!F}^{\!\frac12}\|(\mathbf{v} \!-\! \mathbf{\Pi}_{\bar k}^E \textbf{v})^{F} \|_{F}
+ h_{\!F}^{\!\delta+\frac12}|(\mathbf{v} \!-\! \mathbf{\Pi}_{\bar k}^E \textbf{v})^{F} |_{\delta,F})\\ 
&\lesssim \sum_{F\subseteq \partial E} 
h_{F}^{\frac12}(h_{E}^{-\frac12} \|\mathbf{v} - \mathbf{\Pi}_{\bar k}^E \textbf{v} \|_{E} 
+ h_{E}^{\delta-\frac12}| \mathbf{v} - \mathbf{\Pi}_{\bar k}^E \textbf{v} |_{\delta,E})\\ & \hspace{0.5cm} 
+ h_{F}^{\delta+\frac12} (h_{E}^{-(\delta+\frac12)} \|\mathbf{v} - \mathbf{\Pi}_{\bar k}^E \textbf{v}\|_{E} +  |\mathbf{v} - \mathbf{\Pi}_{\bar k}^E \textbf{v} |_{\delta+\frac12,E})\\ 
&= \| \mathbf{v} - \mathbf{\Pi}_{\bar k}^E \textbf{v} \|_{E} +  h_{E}^{\delta}|\mathbf{v} - \mathbf{\Pi}_{\bar k}^E\textbf{v}|_{\delta,E}+ h_{E}^{\delta+\frac12} |\mathbf{v} - \mathbf{\Pi}_{\bar k}^E \textbf{v}|_{\delta+\frac12,E}.\\
\end{aligned}
 \end{eqnarray}  
Substituting~\eqref{standardinterpolation13dcurl_proof_3slarge_proof3} into~\eqref{standardinterpolation13dcurl_prooflast},
and using~\eqref{Traceinequality3_1}, \eqref{God-bless},
and standard polynomial approximation properties, we obtain
\begin{eqnarray}
\begin{aligned} \label{standardinterpolation13dcurl_prooflast_1}
&\| \mathbf{\Pi}_{\bar k}^E \mathbf{v}- \widetilde{\textbf{I}}_{h}^{e}\textbf{v} \|_{E}
\lesssim    \sum_{F\subseteq \partial E} \big( h_{F}^{\frac32}   
\| \textbf{curl}\hspace{0.05cm}  (\mathbf{v} - \mathbf{\Pi}_{\bar k}^E \textbf{v}) \cdot\textbf{n}_{F} \|_{F} 
+ h^{\frac12}_{F}   \| (\mathbf{v} - \mathbf{\Pi}_{\bar k}^E \textbf{v} )^{F}\|_{F}  \big) \\
& \!\! + h_{E}\| \textbf{curl}  \hspace{0.05cm} (\mathbf{v} \!-\! \mathbf{\Pi}_{\bar k}^E \textbf{v}) \|_{E}
\!+\! \|\mathbf{v} \!-\! \mathbf{\Pi}_{\bar k}^E\textbf{v} \|_{E} 
\!+\!  h_{E}^{\delta}|\mathbf{v} \!-\! \mathbf{\Pi}_{\bar k}^E \textbf{v}|_{\delta,E}
\!+\! h_{E}^{\delta+\frac12}|\mathbf{v} \!-\! \mathbf{\Pi}_{\bar k}^E \textbf{v}|_{\delta+\frac12,E}\\
& \lesssim     h_{E}\| \textbf{curl}  \hspace{0.05cm}  ( \mathbf{v} - \mathbf{\Pi}_{\bar k}^E \textbf{v}) \|_{E} 
+ h_{E}^{\widehat{r}+1} | \textbf{curl}  \hspace{0.05cm} (\mathbf{v} - \mathbf{\Pi}_{\bar k}^E \textbf{v}) |_{\widehat{r},E}
+\|\mathbf{v} - \mathbf{\Pi}_{\bar k}^E \textbf{v}  \|_{E}\\
& +h_{E}^{s} |\mathbf{v}-\mathbf{\Pi}_{\bar k}^E\textbf{v}|_{s,E}
\lesssim  h_{E}^{\widehat{r}+1} | \textbf{curl}  \hspace{0.05cm}    \textbf{v}  |_{\widehat{r},E} 
+ h_{E}^{s} | \textbf{v} |_{s,E} .
\end{aligned} 
 \end{eqnarray}  
Instead, if $1\slash2 < s \le 1$, we replace the term $\bm{\Pi}_{\bar k}^{0,E}$ by $\bm{\Pi}_{0}^{0,E}$
in~\eqref{standardinterpolation13dcurl_proof} and~\eqref{standardinterpolation13dcurl_prooflast}.  
For the term $T_2$,  by the fact that
$( \mathbf{\Pi}_{0}^{0,E} \mathbf{v}-\widetilde{\textbf{I}}_{h}^{e}\textbf{v})^{F}\cdot\textbf{t}_{e}\in \mathbb{P}_{0}(e)\ 
\forall e\subseteq \partial F$,
\eqref{dof13dedgeseinterini},
\eqref{L1BOUND}
and the property that $\|p_{k}\|_{L^{\infty}(e)}\leq C h_{e}^{-\frac12} \|p_{k}\|_{e}$,
we arrive at
\[
\begin{split}
T_2
&\lesssim \sum_{F\subseteq \partial E} h_{F} 
\sum_{e\subseteq \partial F} \sup_{p_{0}\in \mathbb{P}_{0}(e)}
\frac{( (\mathbf{\Pi}_{0}^{0,E}\mathbf{v} - \widetilde{\textbf{I}}_{h}^{e}\textbf{v})^{F}\cdot\textbf{t}_{e}, p_{0} )_{e}  }{\|p_{0}\|_{e}}\\
& =  \sum_{F\subseteq \partial E}h_{F}  
\sum_{e\subseteq \partial F} \sup_{p_{0}\in \mathbb{P}_{0}(e)}
\frac{( (\mathbf{v} - \mathbf{\Pi}_{0}^{0,E} \textbf{v})^{F}\cdot\textbf{t}_{e}, p_{0} )_{e}  }{\|p_{0}\|_{e}}\\
& \!\lesssim\! \!\sum_{F\subseteq \partial E}h^{\frac12}_{F}\!  
\Big( \|(\mathbf{v} \!-\! \mathbf{\Pi}_{0}^{0,E} \textbf{v} )^{F}\|_{F}
\!+\! h_{F}^{\varepsilon}|( \mathbf{v} \!-\! \mathbf{\Pi}_{0}^{0,E} \textbf{v})^{F}|_{\varepsilon,F}
\!+\! h_{F}\|{\rm rot}_{F} \hspace{0.05cm}(\mathbf{v} \!-\! \mathbf{\Pi}_{0}^{0,E} \textbf{v})^{F}\|_{F} \Big).
\end{split}
\]
Combining this and~\eqref{standardinterpolation13dcurl_prooflast},
we choose~$\varepsilon =s-\frac12$
and apply~\eqref{Traceinequality3_1} with~$\delta=r$, \eqref{Traceinequality4} with~$\delta=s$,
\eqref{Traceinequality3_1_1}, and standard polynomial approximation properties, yielding
\begin{eqnarray}
\begin{aligned}[t] \label{standardinterpolation13dcurl_proof177}
& \|\mathbf{\Pi}_{0}^{0,E} \mathbf{v}- \widetilde{\textbf{I}}_{h}^{e}\textbf{v}\|_{E} \\
&  \lesssim    \sum_{F\subseteq \partial E} \Big( h_{F}^{\frac32}    
\| \textbf{curl}\hspace{0.05cm}  (\textbf{v} - \mathbf{\Pi}_{0}^{0,E} \mathbf{v}) \cdot\textbf{n}_{F} \|_{F} 
+ h^{\frac12}_{F}   \|(\textbf{v} - \mathbf{\Pi}_{0}^{0,E} \mathbf{v})^{F}\|_{F}  \\
& \hspace{0.5cm}+ h_{F}^{\varepsilon+\frac12}|(\textbf{v} - \mathbf{\Pi}_{0}^{0,E} \mathbf{v})^{F}|_{\varepsilon,F}\Big) 
+ \| \mathbf{v} - \mathbf{\Pi}_{0}^E \textbf{v} \|_{E}
+ h_{E}\| \textbf{curl}  \hspace{0.05cm} (\textbf{v} - \mathbf{\Pi}_{0}^{0,E} \mathbf{v}) \|_{E}\\ 
& \lesssim   h_{E} \| \textbf{curl}\hspace{0.05cm} (\textbf{v} - \mathbf{\Pi}_{0}^{0,E} \mathbf{v}) \|_{E} 
+ h^{r+1}_{E}| \textbf{curl}\hspace{0.05cm}  (\textbf{v} - \mathbf{\Pi}_{0}^{0,E} \mathbf{v}) |_{r,E}
+ \| \textbf{v} - \mathbf{\Pi}_{0}^{0,E} \mathbf{v}\|_{E} \\
&\hspace{0.5cm}  + h_{E}^{s}| \textbf{v} - \mathbf{\Pi}_{0}^{0,E} \mathbf{v}|_{s, E}
\lesssim  h^{s} |\textbf{v}|_{s,E}
+h^{r+1}_{E} | \textbf{curl}\hspace{0.05cm} \textbf{v}  |_{r,E}
+ h_{E} \| \textbf{curl}\hspace{0.05cm} \textbf{v}  \|_{E} .
\end{aligned}
 \end{eqnarray} 
The assertion follows from a triangle inequality and standard polynomial approximation properties.
\end{proof}

\begin{remark}\label{rem:koala}
Theorem \ref{theorem3dedgeinter} represents an optimal approximation result in the $H_{\rm curl}$ norm.
For low values of $s$ it requires some additional regularity on $\textbf{curl}\hspace{0.05cm}\textbf{v}$ due to the definition of the interpolation operator.
This requirement could be slightly relaxed employing arguments similar to those used in the proof of Lemma~\ref{lemma44}, requiring that $\textbf{curl}\hspace{0.05cm}\textbf{v}\cdot\textbf{n}_{F} $
is integrable on the element faces.
This comes at the price of more cumbersome technicalities, which we prefer to avoid.
Besides, for~$s > 3\slash 2$, we can choose $r=s-1$ in~\eqref{interpolation1curl3D1}
and eliminate also the second term on the right-hand side.
This same remark applies also to Theorem~\ref{theorem433dedges} below.
\end{remark}

\subsection{Serendipity edge virtual element space  on polyhedrons} \label{sec43}
We first change the boundary space~$\mathcal{B}_{k}(\partial E)$ in~\eqref{B-standard-edge-3D}
into its serendipity version:
\begin{align*}
\mathcal{B}^{S}_{k}(\partial E)
=\left\{\textbf{v}_{h}\in \textbf{L}^{2}(\partial E): 
\textbf{v}_{h}^{F}\in \textbf{SV}_{k}^{e}(F)\ \forall F \subseteq \partial E,
\textbf{v}_{h}\cdot \textbf{t}_{e}\ \textrm{is\ continuous}\ \forall e\subseteq \partial F\right\}.
\end{align*}
The serendipity edge virtual element space in 3D  is defined as\cite{beiraobrezzidasimaru2018siam,bebremaru2016RLMASerendipity}
\begin{align*}
\textbf{SV}_{k}^{e}(E)=\big\{\textbf{v}_{h}\in \textbf{L}^{2}(E):\  & \textrm{div} \hspace{0.05cm} \textbf{v}_{h}\in \mathbb{P}_{k-1}(E),  \textbf{curl}  \hspace{0.05cm} \textbf{curl}  \hspace{0.05cm}  \textbf{v}_{h}  \in  (\mathbb{P}_{k}(E))^3,  \\& \ \textbf{v}_{h}^{F}\in \textbf{SV}_{k}^{e}(F)\ \forall F \subseteq \partial E, \textbf{v}_{h}\cdot \textbf{t}_{e}\ \textrm{is\ continuous}\ \forall e\subseteq \partial F\big\}.
\end{align*}
We endow~$\textbf{SV}_{k}^{e}(E)$ with the unisolvent
DoFs~\eqref{dof13dedge}, \eqref{dof33dedge}, \eqref{dof43dedge}, \eqref{dof53dedge}, and 
\[
\bullet \  \int_{F} \textbf{v}^{F}_{h}\cdot \textbf{x}^{F}_{F} p_{\beta_{F}} \forall p_{\beta_{F}}\in \mathbb{P}_{\beta_{F}}(F)  \hspace{0.2cm} \textrm{only\ for}\ \beta_{F}\geq 0. 
\]
For each  sufficiently regular $\textbf{v}$, the DoFs interpolation operator $\textbf{I}_{h}^{e}$ on the space $\textbf{SV}_{k}^{e}(E)$ can be defined  through the above DoFs
enforcing the same conditions~\eqref{dof13dedgeseinterini}, \eqref{dof33dedgeseinterini}, \eqref{dof43dedgeseinterini}, and~\eqref{dof53dedgeseinterini},
and substituting \eqref{dof23dedgeseinterini} by 
\begin{align*}
    \int_{F} (\textbf{v}-\textbf{I}_{h}^{e} \textbf{v})^{F} \cdot \textbf{x}^{F}_{F} p_{\beta_{F}}  =0&&\forall p_{\beta_{F}}\in \mathbb{P}_{\beta_{F}}(F)  \hspace{0.2cm} \textrm{only\ for}\ \beta_{F}\geq 0. 
\end{align*}
From Proposition~$4.2$ in Ref.~\cite{beiraobrezzidasimaru2018siam}, we have
\begin{equation} \label{madonnina}
    \textbf{curl}\hspace{0.05cm}( \textbf{I}_{h}^{e} \textbf{v})= \widetilde{\textbf{I}}_{h}^{f} (\textbf{curl}\hspace{0.05cm}\textbf{v}).
\end{equation}
Next, we prove interpolation estimates for the operator $\textbf{I}^{e}_{h}$
on the serendipity edge virtual element space $\textbf{SV}_{k}^{e}(E)$.
\begin{thm}\label{theorem433dedges}
For each $\textbf{v}\in \textbf {H}^{s}(E)$, $1\slash2< s\leq k+1$,
with   ${\bf curl}\hspace{0.05cm}\textbf{v} \in \textbf {H}^{r}(E)$,   $1\slash2$ $< r\leq k$, we have
\begin{align}  \label{interpolation1curl3Ds}
&\|\textbf{v}- \textbf{I}_{h}^{e} \textbf{v}\|_{E}
\lesssim   h_{E}^{s} |    \textbf{v}|_{s,E} 
+  h_{E}^{\widehat{r}+1}| {\bf curl}\hspace{0.05cm} \textbf{v}  |_{\widehat{r},E}+h_{E}\| {\bf curl}\hspace{0.05cm} \textbf{v}  \|_{E} ,\\
\label{interpolation2curl3Ds}
& \|{\bf curl} \hspace{0.05cm}(\textbf{v}- \textbf{I}_{h}^{e} \textbf{v})\|_{E}
 \lesssim  h_{E}^{r}|{\bf curl}\hspace{0.05cm}\textbf{v}|_{r,E}.
\end{align}
where $\widehat{r} = \min{\{r,[s]\}}$.  
The  third  term  on the right-hand side of \eqref{interpolation1curl3D1} can be neglected if $s \ge 1$.
\end{thm}
\begin{proof}
The proof of bound~\eqref{interpolation2curl3Ds} is essentially identical to that of~\eqref{interpolation2curl3D};
see~\eqref{madonnina}.

Next, we prove bound~\eqref{interpolation1curl3Ds}.
By the inclusion that $\textbf{SV}_{k}^{e}(E)\subseteq \textbf{V}_{k}^{e}(E)$,
bound~\eqref{3dedgePRIORIB} holds true for functions in~$\textbf{SV}_{k}^{e}(E)$.
Thus, for all~$\textbf{v}_{h}$ in~$\textbf{SV}_{k}^{e}(E)$,
also making use of~\eqref{eq:standard_projection4} and~\eqref{svspace}, we can write
\begin{equation*}
\begin{aligned} \label{3dedgePRIORIBserendipity}
& \|\textbf{v}_{h}\|_{E} 
\!\!\!\!\overset{\eqref{3dedgePRIORIB}}{\lesssim}\!\!\!
\sum_{F\subseteq \partial E}
\Big(h_{F}^{\frac32} \|  {\bf curl}\hspace{0.05cm} \textbf{v}_{h} \!\cdot\! \textbf{n}_{F} \|_{F}
\!+\!  h_{F}   \|\textbf{v}_{h}^{F} \!\cdot\! \textbf{t}_{\partial F} \|_{\partial F}
\!+  \!\!\!\!\!\! \sup_{p_{k}\in \mathbb{P}_{k}(F)}\!\!\!\!\!\!
\frac{ h_{F}^{\frac12}\int_{F}  \mathbf{\Pi}_{S}^{e} \textbf{v}^{F}_{h} \cdot \textbf{x}^{F}_{F} p_{k}}{\|\textbf{x}^{F}_{F} p_{k}\|_{F}}\Big)   \\
& \hspace{0.5cm}  +\sup_{\textbf{p}_{k}\in (\mathbb{P}_{k}(E))^3}\frac{h_{E} \int_{E}  {\bf curl}  \hspace{0.05cm}  \textbf{v}_{h}\cdot \textbf{x}_{E}\wedge \textbf{p}_{k}}{\|\textbf{x}_{E}\wedge \textbf{p}_{k}\|_{E}}+\sup_{p_{k-1}\in \mathbb{P}_{k-1}(E)} \frac{\int_{E}  \textbf{v}_{h}   \cdot\textbf{x}_{E}p_{k-1}   }{\|\textbf{x}_{E}p_{k-1}\|_{E}} .
\end{aligned}
\end{equation*}
Let~$\bm{\Pi}_{\bar k}^{E}$ be the operator introduced in the proof of Theorem \ref{theorem3dedgeinter}. 
By replacing $\textbf{v}_{h}$ with $\bm{\Pi}_{\bar k}^{E}\textbf{v}-\textbf{I}^{e}_{h} \textbf{v}$, we can write
\begin{eqnarray}\label{3dedgePRIORIBserendipityprojecion}
\begin{aligned}[t] 
&  \|\bm{\Pi}_{\bar k}^{E}\textbf{v}- \textbf{I}^{e}_{h} \textbf{v} \|_{E} 
\lesssim  \sum_{F\subseteq \partial E} h_{F}^{\frac32} \|  {\bf curl}\hspace{0.05cm} (\bm{\Pi}_{\bar k}^{E}\textbf{v}- \textbf{I}^{e}_{h} \textbf{v} ) \cdot\textbf{n}_{F} \|_{F}\\
& + \sum_{F\subseteq \partial E} h_{F} 
\Big( \| (\bm{\Pi}_{\bar k}^{E}\textbf{v}- \textbf{I}^{e}_{h} \textbf{v})^{F}\cdot\textbf{t}_{\partial F} \|_{\partial F} 
+  \!\!\!\!\!\!\sup_{p_{k}\in \mathbb{P}_{k}(F)}\!\!\!\!\!\!
 \frac{ h_{F}^{\frac12}\int_{F}  \mathbf{\Pi}_{S}^{e} (\bm{\Pi}_{\bar k}^{E}\textbf{v}- \textbf{I}^{e}_{h} \textbf{v})^{F}  \cdot \textbf{x}^{F}_{F} p_{k}}{\|\textbf{x}^{F}_{F} p_{k}\|_{F}} \Big)   \\
 &  + \!\!\!\!\!\!\!\!\! \sup_{\textbf{p}_k \in (\mathbb{P}_{k}(E))^3}\!\!\!\!\!\!\!\!
 \frac{h_{E} \int_{E}  {\bf curl}  \hspace{0.05cm}  (\bm{\Pi}_{\bar k}^{E}\textbf{v} - \textbf{I}^{e}_{h} \textbf{v}) \!\cdot\! \textbf{x}_{E}\wedge \textbf{p}_{k}}{\|\textbf{x}_{E}\wedge \textbf{p}_{k}\|_{E}}
 + \!\!\!\!\!\!\!\!\!\!\!\sup_{p_{k-1}\in \mathbb{P}_{k-1}(E)}\!\!\!\!\!\!\!\!\!\!\!
 \frac{\int_{E}  (\bm{\Pi}_{\bar k}^{E}\textbf{v}- \textbf{I}^{e}_{h} \textbf{v})   \!\cdot\! \textbf{x}_{E}p_{k-1}   }{\|\textbf{x}_{E}p_{k-1}\|_{E}}.
\end{aligned}
\end{eqnarray}
The difference between~\eqref{standardinterpolation13dcurl_proof}
and~\eqref{3dedgePRIORIBserendipityprojecion}
resides only in the third term on the right-hand side,
whence we only discuss its upper bound.
The other four terms are dealt with exactly as in the proof of Theorem~\ref{theorem3dedgeinter}.

Due to the definition of the interpolation operator~$\textbf{I}^{e}_{h}$, 
the functions $(\textbf{I}^{e}_{h} \textbf{v})^{F}$ and $\textbf{v}^{F}$ share the same DoFs on each face~$F$ of~$E$.
Since the value of the projection $\mathbf{\Pi}_{S}^{e}$ only depends on such DoFs, we have
$\mathbf{\Pi}_{S}^{e}(\textbf{I}^{e}_{h} \textbf{v})^{F}=\mathbf{\Pi}_{S}^{e}\textbf{v}^{F}$.
This allows us to write
\begin{align*}
&\|\mathbf{\Pi}_{S}^{e} (\bm{\Pi}_{\bar k}^{E}\textbf{v}- \textbf{I}^{e}_{h} \textbf{v})^{F} \|_{F} 
= \|\mathbf{\Pi}_{S}^{e} (\textbf{v} - \bm{\Pi}_{\bar k}^{E}\textbf{v})^{F} \|_{F}
\overset{\eqref{equivalencenorm}}{\lesssim}
\interleave\mathbf{\Pi}_{S}^{e} (\textbf{v} - \bm{\Pi}_{\bar k}^{E}\textbf{v})^{F} \interleave_{F}\\
& \!\!\!\overset{\eqref{eq:standard_projection}}{=}\!
\interleave  (\textbf{v} \!-\! \bm{\Pi}_{\bar k}^{E}\textbf{v})^{F} \! \interleave_{F}
\!\!\!\!\!\overset{\eqref{equivalencenorm3}}{\lesssim}\!\!\!
\|(\textbf{v} \!-\! \bm{\Pi}_{\bar k}^{E}\textbf{v})^{F} \! \|_{F} 
\!+\! h_{F}^{\varepsilon} |  (\textbf{v} \!-\! \bm{\Pi}_{\bar k}^{E}\textbf{v})^{F} \!|_{\varepsilon,F} 
\!+\! h_{F}\| \textrm{rot}_{F} (\textbf{v} \!-\! \bm{\Pi}_{\bar k}^{E}\textbf{v})^{F} \!\|_{F}.
\end{align*}
This yields
\begin{eqnarray}
\begin{aligned}[t] \label{eq200}
& \sum_{F\subseteq \partial E}
\!\sup_{p_{k}\in \mathbb{P}_{k}(F)}\!\!\!\!
\frac{h_{F}^{\frac12}\int_{F}  \mathbf{\Pi}_{S}^{e} (\bm{\Pi}_{\bar k}^{E}\textbf{v} \!-\! \textbf{I}^{e}_{h} \textbf{v})^{F}  \!\cdot\! \textbf{x}^{F}_{F} p_{k}}{\|\textbf{x}^{F}_{F} p_{k}\|_{F}}
\!\lesssim\! \!\!\sum_{F\subseteq \partial E}\!\! h_{F}^{\frac12} \|\mathbf{\Pi}_{S}^{e} (\bm{\Pi}_{\bar k}^{E}\textbf{v} \!-\! \textbf{I}^{e}_{h} \textbf{v})^{F} \|_{F} \\
& \lesssim \sum_{F\subseteq \partial E}h_{F}^{\frac12}
\big( \|  (\textbf{v} - \bm{\Pi}_{\bar k}^{E}\textbf{v})^{F} \|_{F}
+h_{F}^{\varepsilon} |  (\textbf{v} - \bm{\Pi}_{\bar k}^{E}\textbf{v})^{F} |_{\varepsilon,F}
+h_{F}\| \textrm{rot}_{F} (\textbf{v} - \bm{\Pi}_{\bar k}^{E}\textbf{v})^{F} \|_{F} \big).
\end{aligned}
\end{eqnarray}
Inserting~\eqref{standardinterpolation13dcurl_proof_1}, \eqref{standardinterpolation13dcurl_proof_4}, and~\eqref{eq200}
into \eqref{3dedgePRIORIBserendipityprojecion}, we derive
\begin{align*}
&\| \mathbf{\Pi}_{\bar k}^{E} \mathbf{v} - \textbf{I}_{h}^{e} \textbf{v} \|_{E} 
\lesssim\sum_{F\subseteq \partial E} 
\Big( h_{F}^{\frac32} \| \textbf{curl}\hspace{0.05cm}  (\textbf{v} - \mathbf{\Pi}_{\bar k}^{E} \mathbf{v}) \cdot \textbf{n}_{F} \|_{F}
+ h^{\frac12}_{F}   \|(\textbf{v} - \mathbf{\Pi}_{\bar k}^{E} \mathbf{v} )^{F} \|_{F}\\
& \qquad\qquad + h_{F}^{\varepsilon+\frac12}|(\textbf{v} - \mathbf{\Pi}_{\bar k}^{E} \mathbf{v})^{F}|_{\varepsilon,F}\Big)
+ \| \textbf{v} - \mathbf{\Pi}_{\bar k}^{E} \mathbf{v} \|_{E}
+ h_{E}\| \textbf{curl}  \hspace{0.05cm}  (\textbf{v} - \mathbf{\Pi}_{\bar k}^{E} \mathbf{v}) \|_{E}
+T_2.
\end{align*}
Bound~\eqref{interpolation1curl3Ds} now follows from the same arguments as in  \eqref{standardinterpolation13dcurl_prooflast_1}-\eqref{standardinterpolation13dcurl_proof177}.
\end{proof}

\begin{remark} \label{remark:minimal-regularity}
Differently from the 2D case,
we proved interpolation estimates in 3D for face and edge elements
for functions in~$H^s$ with~$s>1/2$.
One might possibly try to design quasi-interpolation estimates for functions with minimal regularity, i.e., in~$H^s$, $s>0$, and some extra regularity condition on the divergence/curl, for instance by
taking the steps from the recent work \cite{DongErnGuermond} on finite elements.
Such additional developments are beyond the scope of this work.
\end{remark}

\section{Stability theory of the discrete bilinear forms} \label{sec5}
In this section, we focus on the stability properties of  $L^2$ discrete VEM bilinear forms  proposed for the discretization of  electromagnetic
problems in 2D and 3D\cite{lourencobrezzidassi2017cmame,beiraobrezzidasimaru2018siam}.
In Section \ref{sec51}, we define computable stabilizations for the VEM discretization of $L^2$ bilinear forms associated with  face
and   edge virtual element spaces in 2D, and prove their stability properties;
in Section \ref{sec52}, we consider the corresponding results in 3D.
Note that here we focus the attention on stability forms that have a ``functional'' expression with explicit integrals and projections (i.e. do not depend on the particular basis chosen for the VE space). With some additional work, the present results could be also easily extended to dofi-dofi type stabilizations, which are instead related to the basis adopted for the test polynomial spaces in the DoFs definition.

\subsection {The stability in 2D  edge and face virtual element spaces}
 \label{sec51}
For each face~$F$,
we introduce the discrete $L^{2}$ bilinear form $(\cdot,\cdot)_{F}:$
$\textbf{V}_{k}^{e}(F)\times \textbf{V}_{k}^{e}(F) \to \mathbb R$
as\cite{lourencobrezzidassi2017cmame,beiraobrezzidasimaru2018siam}
\begin{align} \label{bilinearform-1}
 [\textbf{v}_{h},\textbf{w}_{h}]_{e,F}:=(\mathbf{\Pi}_{{k}}^{0,F}\textbf{v}_{h},\mathbf{\Pi}_{k}^{0,F}\textbf{w}_{h})_{F}
 +S_{e}^{F}((\textbf{I}-\mathbf{\Pi}_{k}^{0,F})\textbf{v}_{h}, (\textbf{I}-\mathbf{\Pi}_{k}^{0,F})\textbf{w}_{h}).
\end{align}
In~\eqref{bilinearform-1},
$S_{e}^{F}(\cdot,\cdot)$ denotes any symmetric positive definite
bilinear form computable via the DoFs of~$\textbf{V}_{k}^{e}(F)$
such that there exist two positive constant $C_1$ and $C_2$ independent of the mesh size  for which
 \begin{align}
 \label{bilinearform-2}
 C_{1}\|\textbf{v}_{h}\|_{F}^{2}\leq   S_{e}^{F}(\textbf{v}_{h},\textbf{v}_{h})
 \leq  C_{2}\|\textbf{v}_{h}\|_{F}^{2} \hspace{0.3cm} \forall  \textbf{v}_{h}  \in \textbf{V}_{k}^{e}(F).
\end{align}
There are many stabilization choices in the literature.
We here analyze the following (computable) stabilization
$S_{e}^{F}:\textbf{V}_{k}^{e}(F)\times\textbf{V}_{k}^{e}(F)\rightarrow \mathbb{R}$ given by
 \begin{align*}
S_{e}^{F}(\textbf{v}_{h}, \textbf{w}_{h})
\!=\! h_{F} \!\!\!\sum_{e\subseteq \partial F}\! (\textbf{v}_{h} \!\cdot\! \textbf{t}_{e}, \textbf{w}_{h} \!\cdot\! \textbf{t}_{e})_{e}
\!+\! h^{2}_{F} (\textrm{rot}_{F}\hspace{0.05cm}\textbf{v}_{h}, \textrm{rot}_{F}\hspace{0.05cm}\textbf{w}_{h})_{F}
\!+\! (\mathbf{\Pi}_{{k+1}}^{0,F}\textbf{v}_{h},\mathbf{\Pi}_{k+1}^{0,F}\textbf{w}_{h})_{F}.
\end{align*}
\begin{thm}
\label{stab2dedge}
The stabilization~$S_{e}^{F}(\cdot, \cdot)$ satisfies the stability bounds in~\eqref{bilinearform-2}.
\end{thm}
\begin{proof}
The lower bound in~\eqref{bilinearform-2} is proven as follows:
\begin{align*}
\|\textbf{v}_{h}\|_{F}
& \overset{\eqref{prioribound2dedge}}{\lesssim}
h_{F}\|\textrm{rot}_{F}\hspace{0.05cm}\textbf{v}_{h}\|_{F}
+h_{F}^{\frac12}\|\textbf{v}_{h}\cdot\textbf{t}_{\partial F}\|_{\partial F} 
+ \!\!\!\!\!\!\sup_{p_{k}\in \mathbb{P}_{k}(F)}\!\!\!\!\!\!
\frac{ \int_{F}  \textbf{v}_{h} \cdot \textbf{x}_{F} p_{k}}{\|\textbf{x}_{F} p_{k}\|_{F}} \\
& \lesssim h_{F}\|\textrm{rot}_{F}\hspace{0.05cm}\textbf{v}_{h}\|_{F}
+h_{F}^{\frac12}\|\textbf{v}_{h}\cdot\textbf{t}_{\partial F}\|_{\partial F}    
+ \|\bm{\Pi}_{{k+1}}^{0,F}\textbf{v}_{h}\|_{F},
\end{align*}
Next, we observe that the inverse inequality~\eqref{inverserot} is valid for functions in~$\textbf{V}_{k}^{e}(F)$ as well.
We deduce the upper bound in~\eqref{bilinearform-2}:
\[
\begin{split}
& h_{F}^{\frac12}\| \textbf{v}_{h}\cdot  \textbf{t}_{\partial F}\|_{\partial F}
+h_{F} \|\textrm{rot}_{F}\hspace{0.05cm}\textbf{v}_{h}\|_{F}+\|\bm{\Pi}_{k+1}^{0,F}\textbf{v}_{h}\|_{F}
\overset{\eqref{Polynomialinverseestimates1}}{\lesssim}
\| \textbf{v}_{h}\cdot  \textbf{t}_{\partial F}\|_{-\frac{1}{2},\partial F} \\
& + h_{F} \|\textrm{rot}_{F}\hspace{0.05cm}\textbf{v}_{h}\|_{F}
+\|\bm{\Pi}_{k+1}^{0,F}\textbf{v}_{h}\|_{F}
\overset{\eqref{Traceinequality2}}{\lesssim}
\| \textbf{v}_{h}\|_{F}+h_{F} \|\textrm{rot}_{F}\hspace{0.05cm}\textbf{v}_{h}\|_{F}
\overset{\eqref{inverserot}}{\lesssim}
\|\textbf{v}_{h}\|_{F}.
\end{split}
\]
\end{proof}

In the serendipity case,
we can still define a discrete bilinear form on  $\textbf{SV}_{k}^{e}(F)\times \textbf{SV}_{k}^{e}(F)$ as in~\eqref{bilinearform-1},
substituting the stabilization~$S_{e}^F(\cdot,\cdot)$
by the (computable) serendipity stabilization
 \begin{align*}
 S^{s,F}_{e}(\textbf{v}_{h}, \textbf{w}_{h})
 \!=\! h_{F} \!\! \sum_{e\subseteq \partial F}(\textbf{v}_{h} \!\cdot\! \textbf{t}_{e}, \textbf{w}_{h} \!\cdot\! \textbf{t}_{e})_{e}
 \!+\! h^{2}_{F} (\textrm{rot}_{F}\hspace{0.05cm}\textbf{v}_{h}, \textrm{rot}_{F}\hspace{0.05cm}\textbf{w}_{h})_{F}
 \!+\! (\bm{\Pi}_{S}^{e} \textbf{v}_{h}, \bm{\Pi}_{S}^{e} \textbf{w}_{h})_{F}.
\end{align*}

\begin{thm} \label{stab2dedges}
The stabilization $S_{e}^{s,F}(\cdot, \cdot)$ satisfies the bounds
\[
 C_{1}\|\textbf{v}_{h}\|_{F}^{2}
 \leq   S_{e}^{s,F}(\textbf{v}_{h},\textbf{v}_{h})
 \leq  C_{2}\|\textbf{v}_{h}\|_{F}^{2} \hspace{0.3cm} \forall  \textbf{v}_{h}  \in \textbf{SV}_{k}^{e}(F).
\]
\end{thm}
\begin{proof}
The proof follows along the same lines of that of Theorem~\ref{stab2dedge}.
The only difference resides in the lower bound, while treating the term involving the supremum.
It suffices to observe that, due to~\eqref{eq:standard_projection4} and~\eqref{svspace},
we have
\begin{align*}
\sup_{p_{k}\in \mathbb{P}_{k}(F)}\frac{ \int_{F}  \textbf{v}_{h} \cdot \textbf{x}_{F} p_{k}}{\|\textbf{x}_{F} p_{k}\|_{F}}
=\sup_{p_{k}\in \mathbb{P}_{k}(F)}\frac{ \int_{F}  \bm{\Pi}_{S}^{e}\textbf{v}_{h} \cdot \textbf{x}_{F} p_{k}}{\|\textbf{x}_{F} p_{k}\|_{F}},
\end{align*}
and then apply the Cauchy-Schwarz inequality.
\end{proof}
\begin{remark} \label{remark:stability-2D-face}
The stability theory of standard and serendipity face virtual element spaces in 2D follows from the above stability bounds for edge virtual element spaces,
changing ``$\textbf{t}_{e}$" into ``$\textbf{n}_{e}$"
and ``$\textrm{rot}_{F}$"  into ``$\textrm{div}_{F}$".  
\end{remark}

\subsection {The stability    in 3D    edge and face virtual element spaces} \label{sec52}
We first prove stability properties for 3D face virtual element space.
We introduce the symmetric, positive definite, and computable
bilinear form  $S_{f}^{E}(\cdot, \cdot)$ on $\textbf{V}_{k-1}^{f}(E)\times \textbf{V}_{k-1}^{f}(E)$ defined by
\begin{eqnarray} 
\begin{aligned}[b] \label{stab-face-3D}
S_{f}^{E}(\textbf{v}_{h}, \textbf{w}_{h})
& = h_{E} \sum_{F\subseteq \partial E}
(\textbf{v}_{h} \cdot \textbf{n}_{F}, \textbf{w}_{h} \cdot \textbf{n}_{F})_{F}
+ h^{2}_{E} (\textrm{div}\hspace{0.05cm}\textbf{v}_{h}, \textrm{div}\hspace{0.05cm}\textbf{w}_{h})_{E}\\
& \quad + (\mathbf{\Pi}_{{k+1}}^{0,E}\textbf{v}_{h},\mathbf{\Pi}_{k+1}^{0,E}\textbf{w}_{h})_{E}.
\end{aligned}
\end{eqnarray}
We define the local discrete bilinear form on $\textbf{V}_{k-1}^{f}(E)\times \textbf{V}_{k-1}^{f}(E)$:
\begin{align}\label{added-lore}
 [\textbf{v}_{h},\textbf{w}_{h}]_{f,E}
 :=(\mathbf{\Pi}_{{k-1}}^{0,E}\textbf{v}_{h},\mathbf{\Pi}_{k-1}^{0,E}\textbf{w}_{h})_{E}
 +S_{f}^{E} ((\textbf{I}-\mathbf{\Pi}_{k-1}^{0,E})\textbf{v}_{h}, (\textbf{I}-\mathbf{\Pi}_{k-1}^{0,E})\textbf{w}_{h}),
\end{align}
which is computable and approximates the $L^{2}$ bilinear form $(\cdot,\cdot)_{E}$.
Recalling Lemma~\ref{lem3dfacebound} and employing the same arguments as those used in the proof of Theorem~\ref{stab2dedge},
we have  the following stability property.
\begin{thm}\label{them3dface}
The following stability bounds are valid:
\begin{align}
 \label{bilinearform-3dface}
 C_{1}\|\textbf{v}_{h}\|_{E}^{2}
 \leq  S_{f}^{E}(\textbf{v}_{h},\textbf{v}_{h})
 \leq  C_{2}\|\textbf{v}_{h}\|_{E}^{2} \hspace{0.3cm} \forall  \textbf{v}_{h}  \in \textbf{V}_{k-1}^{f}(E).
\end{align}
\end{thm}
Next, we consider the stability analysis for   the VEM discrete form associated with the 3D   edge virtual element space\cite{beiraobrezzidasimaru2018siam}.
The VEM discrete form  of the $L^{2}$ bilinear form $(\cdot,\cdot)_{E}$ on   $\textbf{V}_{k}^{e}(E)\times\textbf{V}_{k}^{e}(E) $ is defined by
\begin{align} \label{added-lore-ArXiv}
 [\textbf{v}_{h},\textbf{w}_{h}]_{e,E}:=(\mathbf{\Pi}_{k}^{0,E}\textbf{v}_{h},\mathbf{\Pi}_{k}^{0,E}\textbf{w}_{h})_{E}+S_{e}^{E}((\textbf{I}-\mathbf{\Pi}_{k}^{0,E})\textbf{v}_{h}, (\textbf{I}-\mathbf{\Pi}_{k}^{0,E})\textbf{w}_{h}),
\end{align}
where $S_{e}^{E}(\cdot,\cdot)$ is a symmetric, positive definite, and computable
bilinear form defined by
\begin{equation*}
\begin{aligned} \label{stabilization-13dedge}
S_{e}^{E}(\textbf{v}_{h}, \textbf{w}_{h})
&=\sum_{F\subseteq \partial E}\left( h^{2}_{F}(\textbf{v}_{h}\cdot \textbf{t}_{\partial F}, \textbf{w}_{h}\cdot \textbf{t}_{\partial F})_{\partial F}
+ h_{F}  (\bm{\Pi}_{k+1}^{0,F}\textbf{v}_{h}^{F}, \bm{\Pi}_{k+1}^{0,F}\textbf{w}_{h}^{F})_{F}\right) \\
 & \quad +  h_{E}^{2} S_{f}^{E}(\textbf{curl}\hspace{0.05cm}\textbf{v}_{h},\textbf{curl}\hspace{0.05cm}\textbf{w}_{h}).
\end{aligned}
\end{equation*}
Before proving stability properties for the discrete bilinear form $[\cdot,\cdot]_{e,E}$,
we extend the inverse inequalities involving edge and face virtual element functions
in Lemma~$5.3$ of Ref.~\cite{beiraolourenco2020} to the general order case.
Such estimates are critical in the following.
\begin{lem}
The following inverse inequalities  hold true:
\begin{align} \label{newinver_3dface}
&\|\textbf{v}_{h}\|_{E} \lesssim h_{E}^{-1}   \|\textbf{v}_{h}\|_{-1,E} && \forall \textbf{v}_{h}\in \textbf{V}_{k-1}^{f}(E),\\
  \label{newinver_2dedge}
&\|\textbf{v}^{F}_{h}\|_{F}\lesssim h_{F}^{-\frac12}   \|\textbf{v}^{F}_{h}\|_{-\frac12,  F} && \forall \textbf{v}_{h}\in \textbf{V}_{k}^{e}(E), \ \forall F \subseteq \partial E.
 \end{align}
 \end{lem}
  \begin{proof}
  We first prove \eqref{newinver_3dface}.
Recalling \eqref{standardinterpolation2_proof13dface}, for each $ \textbf{v}_{h}\in \textbf{V}_{k-1}^{f}(E)$,  there exists   $\textbf{q}_{k}\in (\mathbb{P}_{k}(E))^3$ with $\textrm{div}\hspace{0.05cm}\textbf{q}_{k}=0$ such that
\begin{align}   \label{standardinterpolation2_proof13dfacelast}
\textbf{curl} \hspace{0.05cm} ( \textbf{v}_{h} \!-\! \textbf{x}_{E} \!\wedge\! \textbf{q}_{k}) \!=\! \textbf{0},
\| \textbf{x}_{E} \!\wedge\! \textbf{q}_{k} \|_{E}
\!\lesssim\! h_{E}\|\textbf{curl} (\textbf{x}_{E} \!\wedge\! \textbf{q}_{k}) \|_{E}
\!\lesssim\! h_{E}\|\textbf{curl} \hspace{0.05cm}  \textbf{v}_{h} \|_{E}.
\end{align}
Moreover, the following polynomial inequality holds true:
\begin{align} \label{standardinterpolation2_proof13dfacelastpoinine}
\| \textbf{x}_{E}\wedge \textbf{q}_{k}\|_{-1,E}
= \sup_{0\neq \textbf{v}\in \textbf{H}_{0}^{1}(E)} \frac{(\textbf{x}_{E} \wedge \textbf{q}_{k}, \textbf{v})_{E}}{|\textbf{v}|_{1,E}}
\overset{\eqref{poincarefridine}}{\lesssim}
h_{E}\| \textbf{x}_{E}\wedge \textbf{q}_{k}\|_{E}.
\end{align}

\noindent \textbf{Part~$1$: proving the auxiliary bound~\eqref{standardinterpolation1_proof93dfacelastderive3} below.}
From~\eqref{standardinterpolation2_proof13dfacelast},
there exists a function $\psi\in H^{1}(E)\setminus\mathbb{R}$ such that
\begin{align} \label{sec5lemma1}
\textbf{v}_{h}-\textbf{x}_{E}\wedge \textbf{q}_{k}= \bm{\nabla} \psi.
\end{align}
Such a function is defined by
\begin{align*}
 \Delta \psi= \textrm{div} \hspace{0.05cm}(\textbf{v}_{h}-\textbf{x}_{E}\wedge \textbf{q}_{k}) \ \textrm{in}\ E,\ \ \bm{\nabla} \psi\cdot \textbf{n}_{\partial E}= (\textbf{v}_{h}-\textbf{x}_{E}\wedge \textbf{q}_{k}) \cdot \textbf{n}_{\partial E}\ \textrm{on}\ \partial E.
\end{align*}
Observing that~$\nabla \psi |_{\partial E} \cdot \textbf{n}_{\partial E}$ is a piecewise polynomial, we have
\begin{align*}
& \|\bm{\nabla}  \psi\|^{2}_{E}  
\!\overset{\text{IBP}}{=}\! 
\int_{\partial E} \bm{\nabla}  \psi\cdot\textbf{n}_{\partial E} \psi 
\!-\! \int_{  E} \psi\Delta \psi
\!\lesssim\! \|\bm{\nabla}  \psi\cdot\textbf{n}_{\partial E}\|_{\partial E}\|\psi \|_{\partial E} + \|  \Delta \psi\|_{E}\|  \psi\|_{E} \\
& \overset{\eqref{Polynomialinverseestimates1}, \eqref{Traceinequality3}}{\lesssim}
h_{E}^{-\frac12}\left(\|\bm{\nabla}  \psi \|_{E}
+ h_{E}\| \Delta \psi\|_{E}\right) \|\psi \|_{\partial E}
+ \| \Delta \psi\|_{E}\|  \psi\|_{E} \\
& \overset{\eqref{Polynomialinverseestimates}}{\lesssim}
h_{E}^{-\frac12}\left(\|\bm{\nabla} \psi \|_{E}
+\| \Delta \psi\|_{-1,E}\right) \|\psi \|_{\partial E} 
+ h_{E}^{-1}\| \Delta \psi\|_{-1,E}\|  \psi\|_{E} \\
& \!\lesssim\! \left(\!h_{E}^{\!-\frac12}   \|\psi \|_{\partial E}
\!+\! h_{E}^{\!-1} \| \psi\|_{E}  \right)\|\bm{\nabla} \psi \|_{E}
\!\!\!\overset{\eqref{Traceinequality4last}}{\lesssim}\!\!\!
\left(h_{E}^{\!-\frac12}   \|\psi \|^{\frac12}_{E} \|\bm{\nabla} \psi \|^{\!\frac12}_{E}
\!+\! h_{E}^{\!-1}  \| \psi\|_{E}  \right)\|\bm{\nabla} \psi \|_{E}.
\end{align*}
Also using~\eqref{poincarefridine}, this implies
\begin{eqnarray}
\begin{aligned} \label{standardinterpolation1_proof93dfacelastderive}
\!\!\!\!\!\!\!\!\!\! \|\bm{\nabla}  \psi\| _{E}
\! \lesssim\! h_{E}^{\!-\frac12}   \|\psi \|^{\!\frac12}_{E} \|\bm{\nabla} \psi \|^{\!\frac12}_{E}
\!+\! h_{E}^{\!-1}  \| \psi\|^{\!\frac12}_{E} h_{E}^{\!\frac12} \|\bm{\nabla} \psi \|^{\!\frac12}_{E}
\!=\! h_{E}^{\!-\frac12} \|\psi \|^{\!\frac12}_{E}\|\bm{\nabla} \psi \|^{\!\frac12}_{E}.
\end{aligned}
\end{eqnarray}
Further, by using the continuous inf-sup condition of the Stokes problem,
see, e.g., Section 8.2.1 in Ref.~\cite{boffibookmixed},
we have the following upper bound on~$\|\psi\|_{E}$:
\begin{eqnarray}
\begin{aligned} \label{standardinterpolation1_proof93dfacelastderive1}
\|\psi\| _{E}
\lesssim  \sup_{\bm{\xi}\in \textbf{H}^{1}_{0}(E)}\frac{(\psi, \textrm{div}\hspace{0.05cm}\bm{\xi})_{E}}{|\bm{\xi}|_{1,E}}
= \sup_{\bm{\xi}\in \textbf{H}^{1}_{0}(E)}\frac{(\bm{\nabla}\psi, \bm{\xi})_{E}}{|\bm{\xi}|_{1,E}}
= \|\bm{\nabla}\psi\|_{-1,E}.
\end{aligned}
\end{eqnarray}
Combining~\eqref{standardinterpolation1_proof93dfacelastderive} and~\eqref{standardinterpolation1_proof93dfacelastderive1}, we arrive at
\begin{eqnarray}
\begin{aligned} \label{standardinterpolation1_proof93dfacelastderive2}
\|\bm{\nabla}  \psi\| _{E}
\lesssim h_{E}^{-1}    \|\bm{\nabla}\psi\|_{-1,E}.
\end{aligned}
\end{eqnarray}
Using~\eqref{sec5lemma1}, \eqref{standardinterpolation1_proof93dfacelastderive2} yields
\begin{eqnarray}
\begin{aligned} \label{standardinterpolation1_proof93dfacelastderive3}
\|\textbf{v}_{h}-\textbf{x}_{E}\wedge \textbf{q}_{k}\| _{E}
\lesssim h_{E}^{-1}    \|\textbf{v}_{h}-\textbf{x}_{E}\wedge \textbf{q}_{k}\|_{-1,E}.
\end{aligned}
\end{eqnarray}

\noindent \textbf{Part~$2$: proving~\eqref{newinver_3dface}.}
We  introduce the auxiliary function $\textbf{z}\in \textbf{H}_{0}^{1}(E)$ that realizes the supremum in the definition of
$\|\textbf{v}_{h}-\textbf{x}_{E}\wedge \textbf{q}_{k}\|_{-1,E}$,
i.e., let~$\textbf z$ be the function in~$\textbf{H}_{0}^{1}(E)$ such that 
\begin{eqnarray}
\begin{aligned} \label{sec5lemma1PROOF3}
\|\textbf{v}_{h}-\textbf{x}_{E}\wedge \textbf{q}_{k}\|_{-1,E}
\lesssim   (\textbf{v}_{h}-\textbf{x}_{E}\wedge \textbf{q}_{k},\textbf{z})_{E}    \ \ \textrm{with}\ \ |\textbf{z}|_{1,E}=1.
   \end{aligned}
\end{eqnarray}
As in Remark~\ref{rem21}, we split~$E$ into shape-regular tetrahedra~$\widetilde{\mathcal{T}}_{h}$.
Define~$\psi_{E}$ as the square of the piecewise   quartic bubble function over $\widetilde{\mathcal{T}}_{h}$, scaled such that $\| \psi_{E}\|_{L^{\infty}(E)}=1$.  
We take $\widetilde{\textbf{w}}_{E}= \psi_{E} \textbf{curl}\hspace{0.05cm}(\textbf{x}_{E}\wedge \textbf{q}_{k})$
and defined its scaled version
$\textbf{w}_{E} = \widetilde{\textbf{w}}_{E}/ |\textbf{curl}\hspace{0.05cm}\widetilde{\textbf{w}}_{E}|_{1, E}$.
We have $\textbf{w}_{E} \in \textbf{H}_{0}^{2}(E)$ and $|\textbf{curl}\hspace{0.05cm} \textbf{w}_{E}|_{1, E}=1$.
Furthermore, by~\eqref{sec5lemma1}, we get 
\begin{align}
\label{ortholast}
(\textbf{v}_{h}-\textbf{x}_{E}\wedge \textbf{q}_{k},\textbf{curl}\hspace{0.05cm}\textbf{w}_{E} )_{E}=(\bm{\nabla} \psi,\textbf{curl}\hspace{0.05cm}\textbf{w}_{E})_{E}=0.
\end{align}
We write
\begin{eqnarray}
\begin{aligned}[b] \label{sec5lemma1PROOF4}
&  (\textbf{x}_{E}\wedge \textbf{q}_{k},\textbf{curl}\hspace{0.05cm}\textbf{w}_{E})_{E}
\overset{\text{IBP}}{=}
(\textbf{curl}\hspace{0.05cm}(\textbf{x}_{E}\wedge \textbf{q}_{k}),\textbf{w}_{E})_{E}\\
&  =  \frac{(\textbf{curl}\hspace{0.05cm}(\textbf{x}_{E}\wedge \textbf{q}_{k}), \psi_{E} \textbf{curl}\hspace{0.05cm}(\textbf{x}_{E}\wedge \textbf{q}_{k}))_{E}}{|\textbf{curl}\hspace{0.05cm}(\psi_{E} \textbf{curl}\hspace{0.05cm}(\textbf{x}_{E}\wedge \textbf{q}_{k}))|_{1, E}}  
\!\!\overset{\eqref{Polynomialinverseestimates}}{\geq}\!\! 
\frac{(\textbf{curl}\hspace{0.05cm}(\textbf{x}_{E}\wedge \textbf{q}_{k}), \psi_{E} \textbf{curl}\hspace{0.05cm}(\textbf{x}_{E}\wedge \textbf{q}_{k}))_{E}}{h_{E}^{-1}\|\textbf{curl}\hspace{0.05cm}(\psi_{E} \textbf{curl}\hspace{0.05cm}(\textbf{x}_{E}\wedge \textbf{q}_{k}))\|_{E}}  \\
& \overset{\eqref{Polynomialinverseestimates}, \eqref{bubblefuncionproer}}{\geq}
\frac{ C \|\textbf{curl}\hspace{0.05cm}(\textbf{x}_{E}\wedge \textbf{q}_{k})\|^{2}_{E}}{h_{E}^{-2}\|  \textbf{curl}\hspace{0.05cm}(\textbf{x}_{E}\wedge \textbf{q}_{k}) \|_{E}} =  C h^{2}_{E} \|\textbf{curl}\hspace{0.05cm}(\textbf{x}_{E}\wedge \textbf{q}_{k})\|_{E} \\
& \overset{\eqref{standardinterpolation2_proof13dfacelast}}{\geq}
C  h_{E}  \| \textbf{x}_{E}\wedge \textbf{q}_{k}\|_{E}
\overset{\eqref{standardinterpolation2_proof13dfacelastpoinine}}{\geq} C_1\| \textbf{x}_{E}\wedge \textbf{q}_{k}\|_{-1,E}.
   \end{aligned}
\end{eqnarray}
From the definition of negative norm~$\|\cdot\|_{-1,E}$,
the fact that $\textbf{curl}\hspace{0.05cm}\textbf{w}_{E}\in \textbf{H}_{0}^{1}(E)$, and~\eqref{ortholast}, we can write
\begin{eqnarray}
 \begin{aligned}
 \label{sec5lemma1PROOF2}
  \|\textbf{v}_{h}\|_{-1,E}&=\sup_{\bm{\xi}\in \textbf{H}_{0}^{1}(E)}\frac{(\textbf{v}_{h},\bm{\xi})_{E}}{|\bm{\xi}|_{1,E}}=\sup_{\bm{\xi}\in \textbf{H}_{0}^{1}(E)}\frac{(\textbf{v}_{h}-\textbf{x}_{E}\wedge \textbf{q}_{k},\bm{\xi})_{E}+(\textbf{x}_{E}\wedge \textbf{q}_{k},\bm{\xi})_{E}}{|\bm{\xi}|_{1,E}}\\
  &\geq  \frac{(\textbf{v}_{h}-\textbf{x}_{E}\wedge \textbf{q}_{k},\textbf{z}+\alpha\textbf{curl}\hspace{0.05cm} \textbf{w}_{E})_{E}+(\textbf{x}_{E}\wedge \textbf{q}_{k},\textbf{z}+\alpha\textbf{curl}\hspace{0.05cm} \textbf{w}_{E})_{E}}{|\textbf{z}+\alpha\textbf{curl}\hspace{0.05cm} \textbf{w}_{E}|_{1,E}}\\
  &\geq \frac{(\textbf{v}_{h}-\textbf{x}_{E}\wedge \textbf{q}_{k},\textbf{z})_{E}+(\textbf{x}_{E}\wedge \textbf{q}_{k},\textbf{z})_{E}+(\textbf{x}_{E}\wedge \textbf{q}_{k},\alpha\textbf{curl}\hspace{0.05cm} \textbf{w}_{E})_{E}}{1+\alpha },
\end{aligned}
\end{eqnarray}
where~$\alpha$ is a positive constant, which we shall fix in what follows.
Next, we obtain
\begin{align*}
\|\textbf{v}_{h}\|_{-1,E}
& \!\!\!\!\!\!\!\!\!\!\overset{\eqref{sec5lemma1PROOF2}, \eqref{sec5lemma1PROOF3}, \eqref{sec5lemma1PROOF4}}{\geq}\!\!
\frac{C\|\textbf{v}_{h}-\textbf{x}_{E}\wedge \textbf{q}_{k}\|_{-1,E}-\|\textbf{x}_{E}\wedge \textbf{q}_{k}\|_{-1,E}+C_1\alpha \| \textbf{x}_{E}\wedge \textbf{q}_{k}\|_{-1,E}}{1+\alpha}\\
& =\frac{C}{1+\alpha} \|\textbf{v}_{h}-\textbf{x}_{E}\wedge \textbf{q}_{k}\|_{-1,E}+ \frac{C_1\alpha-1}{1+\alpha}\|\textbf{x}_{E}\wedge \textbf{q}_{k}\|_{-1,E}\\
& \overset{\eqref{standardinterpolation1_proof93dfacelastderive3},  \eqref{Polynomialinverseestimates}}{\geq}
Ch_{E}(\|\textbf{v}_{h}-\textbf{x}_{E}\wedge \textbf{q}_{k}\|_{E}+\|\textbf{x}_{E}\wedge \textbf{q}_{k}\|_{E})\geq Ch_{E} \|\textbf{v}_{h}\|_{E}.
\end{align*}
where we have fixed $\alpha=2/C_1$. This completes the proof of~\eqref{newinver_3dface}.

\noindent \textbf{Part~$3$: proving~\eqref{newinver_2dedge}.}
We first recall that $\textbf{v}_{h}^{F}$ belongs to $\textbf{V}_{k}^{e}(F)$ for each $\textbf{v}_{h}\in \textbf{V}_{k}^{e}(E)$.
Next, we observe that the inverse estimate~\eqref{newinver_3dface} for functions in~$\textbf{V}_{k-1}^{f}(E)$,
implies an 2D analogous counterpart on the space~$\textbf{V}_{k}^{f}(F)$:
\begin{align*}
\|\textbf{v}_{h}\|_{F} 
\lesssim h_{F}^{-1}   \|\textbf{v}_{h}\|_{-1,F} \hspace{0.4cm }\forall \textbf{v}_{h}\in \textbf{V}_{k}^{f}(F).
 \end{align*}
The counterpart for the 2D edge virtual element space~$\textbf{V}_{k}^{e}(F)$
is obtained via a ``rotation'' argument as in Section~\ref{sec4faceVEM}:
\begin{align*}
\|\textbf{v}_{h}\|_{F}
\lesssim h_{F}^{-1}   \|\textbf{v}_{h}\|_{-1,F} \hspace{0.4cm} \forall \textbf{v}_{h}\in \textbf{V}_{k}^{e}(F).
 \end{align*}
Hence, we arrive at
\begin{align*}
\|\textbf{v}^{F}_{h}\|_{F}
\lesssim h_{F}^{-1}   \|\textbf{v}^{F}_{h}\|_{-1,  F} \hspace{0.4cm} \forall \textbf{v}_{h}\in \textbf{V}_{k}^{e}(E),\  \forall F\subseteq \partial E.
\end{align*}
The assertion follows from classical results in space interpolation theory\cite{2007An}.
\end{proof}

With these tools at hand,  we can prove  the following stability property.
\begin{thm}
\label{thm3dstaedge}
The following stability bounds are valid:
\begin{align} \label{bilinearform-3dedge}
C_{1}\|\textbf{v}_{h}\|_{E}^{2}
\leq  S_{e}^{E}(\textbf{v}_{h}, {\textbf{v}_{h}})
\leq  C_{2}\|\textbf{v}_{h}\|_{E}^{2} \hspace{0.3cm} 
\forall \textbf{v}_{h}  \in \textbf{V}_{k}^{e}(E) .
\end{align}
\end{thm}
\begin{proof}
The lower bound in \eqref{bilinearform-3dedge} is proved as follows:
\begin{equation*}
 \begin{aligned}
 \label{lem3dedgeboundapply}
&\|\textbf{v}_{h}\|_{E}
\!\overset{\eqref{3dedgePRIORIB}}{\lesssim}\!
\!\! \sum_{F\subseteq \partial E}\!\! \left(h_{F}^{\frac32} \|\textbf{curl}\hspace{0.05cm} \textbf{v}_{h} \! \cdot \! \textbf{n}_{F} \|_{F} \!+\!  h_{F}   \|\textbf{v}_{h}^{F} \!\cdot\! \textbf{t}_{\partial F} \|_{\partial F}
+ \!\!\!\!\!\!\sup_{p_{k}\in \mathbb{P}_{k}(F)}\!\!\!\!\!\!
\frac{ h_{F}^{\frac12}\int_{F}  \textbf{v}^{F}_{h} \cdot \textbf{x}^{F}_{F} p_{k}}{\|\textbf{x}^{F}_{F} p_{k}\|_{F}}\right)\\
&\hspace{0.5cm} + \sup_{\textbf{p}_{k}\in (\mathbb{P}_{k}(E))^3}
\frac{h_{E} \int_{E} \textbf{curl}  \hspace{0.05cm}  \textbf{v}_{h}\cdot \textbf{x}_{E}\wedge \textbf{p}_{k}}{\|\textbf{x}_{E}\wedge \textbf{p}_{k}\|_{E}}  +\sup_{p_{k-1}\in \mathbb{P}_{k-1}(E)} \frac{\int_{E}  \textbf{v}_{h}   \cdot\textbf{x}_{E}p_{k-1} }{\|\textbf{x}_{E}p_{k-1}\|_{E}}\\
& \!\overset{\eqref{stab-face-3D}}{\lesssim}\! \! h_{E}  S_{f}^{E} (\textbf{curl}\hspace{0.05cm}\textbf{v}_{h},\textbf{curl}\hspace{0.05cm}\textbf{v}_{h})^{\frac12} 
\!\!+\!\!\!\! \sum_{F\subseteq \partial E} \!\!\!
\left(h_{F}   \|\textbf{v}_{h}^{F}\!\!\cdot\!\textbf{t}_{\partial F} \|_{\partial F}
\!+\!   h_{F}^{\frac12} \|\bm{\Pi}_{k+1}^{0,F}\!\! \textbf{v}^{F}_{h}\|_{F} \right) ,
\end{aligned}
\end{equation*}
As for the upper bound in~\eqref{bilinearform-3dedge}, we write
\begin{equation*}
\begin{aligned}
 \label{vhHeldecom73dedgesa}
&\sum_{F\subseteq \partial E}\! \left(\! h_{F}   \|\textbf{v}_{h}^{F}\cdot\textbf{t}_{\partial F} \|_{\partial F} \!+\! h_{F}^{\frac12} \|\bm{\Pi}_{k+1}^{0,F} \textbf{v}^{F}_{h}\|_{F} \!\right)  \!+\! h_{E}  S_{f}^{E}(\textbf{curl}\hspace{0.05cm}\textbf{v}_{h},\textbf{curl}\hspace{0.05cm}\textbf{v}_{h})^{\frac12}  \\
& \overset{\eqref{Polynomialinverseestimates1}, \eqref{bilinearform-3dface}}{\lesssim}
\sum_{F\subseteq \partial E}\left( h_{F}^{\frac12}\| \textbf{v}^{F}_{h}\cdot  \textbf{t}_{\partial F}\|_{-\frac12,\partial F}
+  h_{F}^{\frac12} \|\bm{\Pi}_{k+1}^{0,F} \textbf{v}^{F}_{h}\|_{F} \right) 
\!+\!  h_{E}  \| \textbf{curl}\hspace{0.05cm} \textbf{v}_{h} \|_{E} 
\!+\!  \|  \textbf{v}_{h} \|_{E} \\
& \overset{\eqref{Traceinequality2}}{\lesssim}
\! \sum_{F\subseteq \partial E}\!\! \left(\! h_{F}^{\frac12}\| \textbf{v}^{F}_{h} \|_{F} 
\!+\! h_{F}^{\frac32} \|\textrm{rot}\hspace{0.05cm} \textbf{v}^{F}_{h} \|_{F}
\!+\! h_{F}^{\frac12} \|\bm{\Pi}_{k+1}^{0,F} \textbf{v}^{F}_{h}\|_{F} \!\right)  
\!\!+\! h_{E}  \|\textbf{curl}  \hspace{0.05cm}  \textbf{v}_{h}\|_{E} \!+\! \| \textbf{v}_{h} \|_{E} \\
& \overset{\eqref{inverserot}, \eqref{newinver_2dedge}}{\lesssim}
\| \textbf{v}_{h}\wedge \textbf{n}_{\partial E} \|_{-\frac12,  \partial E} 
\!+\! h_{E}  \| \textbf{curl}\hspace{0.05cm} \textbf{v}_{h} \|_{E}   
\!+\!  \|  \textbf{v}_{h}\|_{E} 
\!\!\!\!\!\!\!\!\overset{\eqref{Traceinequality4_1}, \eqref{newinver_3dface}}{\lesssim}\!\!
\| \textbf{curl}\hspace{0.05cm} \textbf{v}_{h} \|_{-1,E} 
\!+\!  \|  \textbf{v}_{h}\|_{E} \\
& =  \sup_{\bm{\psi}\in \textbf{H}^{1}_{0}(E) } \frac{   (\textbf{curl}\hspace{0.05cm} \textbf{v}_{h},\bm{\psi})_{E}} {|\bm{\psi}|_{1,E}}
+  \| \textbf{v}_{h}\|_{E} 
\overset{\text{IBP}}{\lesssim}
\|\textbf{v}_{h}\|_{E}.
\end{aligned}
\end{equation*}
\end{proof}

\begin{remark} \label{remark:stability-3D-face}
Following the definition of $S^{s,E}_{e}(\cdot, \cdot)$, we can also define the following alternative stabilization for the case of the serendipity edge virtual element space in 3D:
\begin{equation}
 \begin{aligned} \label{added-lore-2}
S^{^{s,E}}_{e}(\textbf{v}_{h}, \textbf{w}_{h})
&=\sum_{F\subseteq \partial E}
\left( h^{2}_{F}(\textbf{v}_{h}\cdot \textbf{t}_{\partial F}, \textbf{w}_{h}\cdot \textbf{t}_{\partial F})_{\partial F}
+ h_{F} (\bm{\Pi}_{S}^{e}\textbf{v}_{h}^{F}, \bm{\Pi}_{S}^{e}\textbf{w}_{h}^{F})_{F} \right)\\
& \qquad \qquad 
+ h_{E}^{2}S_{f}^{E}(\textbf{curl}\hspace{0.05cm}\textbf{v}_{h},\textbf{curl}\hspace{0.05cm}\textbf{w}_{h}).
\end{aligned}
\end{equation}
The advantage of the variant above is that,
if we substitute $(\textbf{I}-\mathbf{\Pi}_{k-1}^{0,E})$ by $(\textbf{I}-\bm{\Pi}_{S}^{e})$ in the stabilization term of the scalar product \eqref{added-lore-ArXiv},
then the second addendum in definition \eqref{added-lore-2} will vanish, thus leading to a lighter form. 
Employing analogous arguments, we can prove the same stability bounds as in Theorem~\ref{thm3dstaedge} also for choice~\eqref{added-lore-2}.
\end{remark}



\section*{Acknowledgements}
The work of J. M. is partially supported
by the China Scholarship Council (No. $202106280167$)
and the Fundamental Research Funds for the Central Universities (No. xzy $022019040$).
L. B. d. V. was partially supported by the italian PRIN 2017
grant “Virtual Element Methods: Analysis and Applications” and the
PRIN 2020 grant “Advanced polyhedral discretisations of heterogeneous
PDEs for multiphysics problems”.
L. M. acknowledges support from the Austrian Science Fund (FWF) project P33477.

{\footnotesize
\bibliography{bibliography}
\bibliographystyle{plain}
}

\end{document}